%% file: DispersalV4Arxiv.tex
\documentclass[aos,preprint]{imsart}
\usepackage{lmodern}
\usepackage[T1]{fontenc}
\usepackage[latin9]{inputenc}
\usepackage{babel}
\usepackage{amsmath}
\usepackage{amsthm}
\usepackage{amssymb}
\usepackage{graphicx}
\usepackage{tikz}
\usepackage[authoryear]{natbib}
\usepackage[unicode=true,pdfusetitle,
 bookmarks=true,bookmarksnumbered=false,bookmarksopen=false,
 breaklinks=false,pdfborder={0 0 1},backref=false,colorlinks=false]
 {hyperref}

\theoremstyle{plain}
\newtheorem{thm}{\protect\theoremname}

\theoremstyle{plain}
\newtheorem{prop}[thm]{\protect\propositionname}

\theoremstyle{plain}
\newtheorem{cor}[thm]{\protect\corollaryname}

\theoremstyle{remark}
\newtheorem{rem}[thm]{\protect\remarkname}

\theoremstyle{plain}
\newtheorem{lem}[thm]{\protect\lemmaname}

\theoremstyle{definition}

\theoremstyle{plain}
\newtheorem{assumption}[thm]{Assumption}

\theoremstyle{plain}
\newtheorem{definition}[thm]{Definition}

\usepackage{xcolor,soul}
\usepackage{bbm}
\usepackage{a4wide}

\allowdisplaybreaks

\renewcommand{\hat}{\widehat}
\renewcommand{\tilde}{\widetilde}
\renewcommand{\bar}{\overline}

\makeatother

\providecommand{\corollaryname}{Corollary}
\providecommand{\examplename}{Example}
\providecommand{\lemmaname}{Lemma}
\providecommand{\propositionname}{Proposition}
\providecommand{\remarkname}{Remark}
\providecommand{\theoremname}{Theorem}

\begin{document}
\global\long\def\phi{\varphi}%
\global\long\def\epsilon{\varepsilon}%
\global\long\def\theta{\vartheta}%
\global\long\def\P{\mathbb{P}}%
\global\long\def\E{\mathbb{E}}%
\global\long\def\Var{\operatorname{Var}}%
\global\long\def\Cov{\operatorname{Cov}}%
\global\long\def\N{\mathbb{N}}%
\global\long\def\Z{\mathbb{Z}}%
\global\long\def\R{\mathbb{R}}%
\global\long\def\F{\mathcal{F}}%
\global\long\def\le{\leqslant}%
\global\long\def\ge{\geqslant}%
\global\long\def\eins{\mathbbm1}%
\global\long\def\d{\mathrm{d}}%
 
\global\long\def\subset{\subseteq}%
\global\long\def\supset{\supseteq}%
\global\long\def\argmin{\arg\,\min}%
\global\long\def\bull{{\scriptstyle \bullet}}%
\global\long\def\supp{\operatorname{Supp}}%
\global\long\def\sgn{\operatorname{sign}}%

%\title{Dispersal density estimation across scales}
%\author{Marc Hoffmann, and Mathias Trabs\footnote{We thank our colleagues Marie Doumic and Alexander Goldenshluger for helpful discussions. We are grateful to the comments of three referees that convinced us to extend the results of a former version to the case of an unknown scale parameter. M.T. has been financially supported by DFG via the Heisenberg grant TR 1349/4-1.}}
%\date{Universit\'e Paris-Dauphine and Karlsruhe Institute of Technology}

%\maketitle

\begin{frontmatter}
\title{Dispersal density estimation across scales}
\runtitle{Dispersal density estimation across scales}

\begin{aug}
\author{Marc Hoffmann} \and
\author{Mathias Trabs}

%\thankstext{t1}{Corresponding author.}
%\thankstext{t2}{We thank our colleagues Marie Doumic and Alexander Goldenshluger for helpful discussions. The analysis and comments of three referees that convinced us to extend the results of a former version to the case of an unknown scale parameter are gratefully acknowledged. M.T. has been financially supported by DFG via the Heisenberg grant TR 1349/4-1.}

\address{Marc Hoffmann,\\ Universit\'e Paris-Dauphine PSL, \\Place du Mar\'echal De Lattre de Tassigny, \\75016 Paris, France}
\address{Mathias Trabs,\\ Karlsruhe Institute of Technology,\\ Institut f\"ur Stochastik\\ Englerstr. 2,\\ 76131 Karlsruhe, Germany}

\runauthor{Marc Hoffmann and Mathias Trabs}

\affiliation{Universit\'e Paris-Dauphine and Karlsruhe Institute of Technology}

\end{aug}

\begin{abstract}
We consider a space structured population model generated by two point clouds: a homogeneous Poisson process $M$ with intensity $n\to\infty$ as a model for a parent generation together with a Cox point process $N$ as offspring generation, with conditional intensity given by the convolution of $M$ with a scaled dispersal density $\sigma^{-1}f(\cdot/\sigma)$. Based on a realisation of $M$ and $N$, we study the nonparametric estimation of $f$ and the estimation of the physical scale parameter $\sigma>0$ simultaneously for all regimes $\sigma=\sigma_n$. We establish that the optimal rates of convergence do not depend monotonously on the scale and we construct minimax estimators accordingly whether $\sigma$ is known or considered as a nuisance, in which case we can estimate it and achieve asymptotic minimaxity by plug-in. The statistical recons\-truction exhibits a competition between a direct and a decon\-volution problem. Our study reveals in particular the existence of a least favourable intermediate inference scale,  a phenomenon that seems to be new. 
%{ When $\sigma$ is unknown, we develop a pilot estimation strategy that enables us to achieve optimality for estimation $f$ by plug-in while $\sigma$ is considered as a nuisance parameter.}
\end{abstract}


%{\small
%\noindent {\bf Mathematics Subject Classification (2010):} 62G05, 62G07, 62M30, 60G57 .\\
%\noindent {\bf Keywords:} Nonparametric estimation and minimax theory; point processes; statistical inference across scales; dispersal models; deconvolution.
%}

\end{frontmatter}

\section{Introduction}
 
\subsection{Statistical inference across scales} \label{sec: motivation}

Data behave differently at different scales. The interplay between the information parameter ({\it e.g.} number of observations, inverse of the noise level, time length of measurements) and the physical scale of the observables may affect the structure of the underlying statistical model. We encode this idea by extending the familiar notion of a statistical experiment to a family 
\begin{equation} \label{eq: ref scale exp}
\mathcal E =  \big\{\mathsf P_{f,\sigma}^{n}: f \in \Theta\big\}_{n \ge 1, \sigma >0}
\end{equation}
where the probability measures $\mathsf P_{f,\sigma}^{n}$ are simultaneously indexed by an information parameter $n\ge 1$ and a physical scale $\sigma >0$, and that we shall refer to as a (family of) statistical experiment(s) across scales. 
Depending on the choice of $\sigma = \sigma_n$ varying with the information rate $n$, the statistical geometrical properties of $\mathcal E$ (such as LAN type conditions or asymptotic equivalence features as discussed in \citet{le2012asymptotic, van2002statistical}) may differ. In particular, the choice of an optimal procedure may be dictated by different regimes governed by $\sigma_n$.\\ 

It is therefore desirable to understand the larger picture given by \eqref{eq: ref scale exp} \emph{simultaneously for all subsequences} $\sigma_n$. In an asymptotic setting, we may attempt to realise the following program:

\begin{itemize}
\item Identify the optimal estimation for $f$ (in an asymptotic minimax sense for a given loss function) for an arbitrary (but known) $\sigma = \sigma_n$.
\item Considering $\sigma = \sigma_n$ as unknown, estimate simultaneously $\sigma$ and $f$ and achieve optimality for $f$ in this setting.
\end{itemize}

%We refer to \citet{duval2011statistical, chorowski2018nonparametric, nickl2016high} where this point of view is put to fruition in the statistics of diffusion and L\'evy processes. In particular, the robustness of certain estimation methods is studied across time scales.\\ 

In this paper, we build a family of statistical experiments across spatial scales that exhibit nontrivial behaviours at certain critical levels and for which different estimation procedures with different rates of convergence enter into competition as the scale varies.  This can be of crucial importance in practice, and is in stark contrast with the results in \citet{duval2011statistical, nickl2016high, chorowski2018nonparametric}  where some robustness of estimation methods and of the minimax rates of convergence is observed across time scales
for L\'evy and diffusion processes.\\
%
%We informally call this model  {\it dispersal inference} and describe it in more details in the next section: it may serve in several applications, ranging from queuing theory to particle diffusions or biological dispersion of species, and even molecular genetics, see Section \ref{sec: applications} below. Depending on the physical scale parameter $\sigma_n$ (varying from a microscopic to a macroscopic regime with the information parameter $n$), the data exhibit different properties and so do optimal procedures. 
%Yet, we are able to give a full picture of nonparametric estimation across scales in a minimax setting. In particular, our approach reveals a least favourable scale $\sigma_n$ that depends on the underlying smoothness of the signal and which is not the smallest nor the biggest possible in our range, a somehow counterintuitive result. This phenomenon presumably lays its roots beyond our approach, see the discussion Section \ref{sec: discussion}.   
%

\subsection{A model for dispersal estimation} \label{sec: dispersal inference setting}

\subsubsection*{Informal description}
We start with two random points $X, Y \in \R^d$, where $X \in \mathcal O$ for some domain $\mathcal O \subset \R^d$ represents the trait of a parent in a spatially structured population, and $Y \in \R^d$ is the location of (one of) its children. We are interested in recovering the dispersal distribution of $Y-X$. This means that $Y-X$ has a density function 
\begin{equation} \label{eq: def informal dispersal}
f_\sigma(z) = \sigma^{-d}f(z/\sigma),\qquad z\in\R^d,
\end{equation}
%with a fixed baseline normalisation
%$$\mathrm{Var}(f) = \int_{\R^d} |z|^2f(z)\d z=\sigma_0^2$$
with a physical (dispersal) scale parameter $\sigma >0$ which determines the order of $\E[|Y-X|^2]^{1/2}$, where $|\cdot|$ denotes the Euclidean distance. 
%We are interested in recovering $f$ nonparametrically. 
The parameter of interest is the density function $f$.
If we observe an $n$-sample $(X_i,Y_i)_{1 \le i \le n}$, the $Y_i-X_i$ have common distribution $f_\sigma$, and we are in a classical density estimation framework;  the scale $\sigma$ is irrelevant. Assume now that we are rather given two point clouds $\mathcal X$ and $\mathcal Y$ in $\R^d$ with 
$$\mathcal X = \{X_i: i= 1,\ldots, n\}\;\;\text{and}\;\;\mathcal Y = \{Y_j: j=1,\ldots n\},$$
 {\it i.e.} we do not know the match between a parent and its offspring, hence we do not observe the variables $Y_i-X_i$ anymore. Inferring $f$ in such a setting is the topic of the paper.\\
 
 The scale parameter $\sigma$ now becomes crucial. Heuristically, if $\sigma \ll n^{-1/d}$, {\it i.e.} the dispersal scale is small with respect to the typical distance between the locations of the parents population $\mathcal X$, then we may guess the parents-offspring match by a nearest distance procedure, {\it i.e.} take $X_{(j)}$ as the solution to
\[
|Y_{j}-X_{(j)}|=\min\big\{|Y_{j}-X_{i}|:i=1,\dots,n\big\},
\]
and proceed as if the $Y_{j}-X_{(j)}$ were an $n$-sample with distribution $f_\sigma$, up to controlling the mismatch error. 
%(We develop this approach for $d=1$ in Section \ref{sec:naive}.) 
However, if $\sigma \gg n^{-1/d}$, the mismatch error explodes and alternative methods need to be found. For instance, for a child trait $Y_j$ with parent trait $X_{i_j}$, writing $Y_j = X_{i_j}+\sigma D_{j}$, we see that $D_{j}$ has density $f$, therefore, if the parent distribution $p$ is known, then the $Y_j$ have common distribution $p \ast f_\sigma$, where $\ast$ denotes convolution. We may then implement a deconvolution approach to recover $f$ and ignore the potential information given by the point cloud of the parent traits $\mathcal X$. This has some price, namely the ill-posedness of an inverse problem, and has to be assessed with some care.  
%(We develop this approach for $d=1$ in Section \ref{sec:naive}.) 
Our objective is to formalise this model and these approaches in order to encompass potential applications as described in Section \ref{sec: applications} below. In particular, we need not impose that $\mathcal X$ and $\mathcal Y$ have the same size, allowing for a random number of parents and children. To provide a complete and transparent picture, we will greatly simplify the technicalities of our approach by restricting ourselves to the one-dimensional case $d=1$, with $\mathcal O = [0,1]$. Extensions to more general domains $\mathcal O$ for the state space of the parents as well as in higher dimension $d >1$ are available and discussed in Section \ref{sec: discussion}.

\subsubsection*{Formal construction of the model}
Random point clouds are equivalently represented by random finite point measures. The location traits of the parent generation, {\it i.e.} the point cloud $\mathcal X  \subset \mathcal O = [0,1]$ is modelled as a homogeneous Poisson point
process 
\[
M(\d x)=\sum_{j}\delta_{X_{j}}(\d x)
\]
on the unit interval $[0,1]$ with intensity measure $m(\d x) =  n \lambda \d x$, where
$n\to\infty$ and $\lambda>0$ is fixed. Note that the size $|\mathcal X|$ is random, with $\E[|\mathcal X|] = n\lambda$. Given a realisation of $M$, the point cloud $\mathcal Y\subset \R$ that represents the traits of the offspring  is generated by a Cox point
process 
\[
N(\d y)=\sum_{j}\delta_{Y_{j}}(\d y)
\]
with (conditional) intensity measure 
\[
\mu\big(M\ast f_{\sigma}\big)(y)\d y=\sum_{i}\mu f_{\sigma}(y-X_{i})\d y,
\]
where the  \emph{dispersal
density } is $f_{\sigma}=\sigma^{-1}f(\cdot/\sigma)$ as in \eqref{eq: def informal dispersal} with \emph{dispersal scale} parameter $\sigma >0$. The parameter $\mu>0$ represents the average number of an offspring given one parent.
Hence, $f_\sigma$ describes the distribution of the random variable $Y_{j}-X_{i_j}$ (when the child $j$ has parent $i_j$). The distance between the traits of the children
and the trait of their parent is of order $\sigma$.
The expected size of the offspring population ({\it i.e.} the average size of $\mathcal Y$) is $n \lambda\mu$. In Figure \ref{fig: simu} we simulate a realisation of the $(M,N)$ process, for different values of $\sigma = \sigma_n$ depending on $n$.\\ 
%With $\lambda$ and $\mu$ fixed,
%the order of magnitude of $\mathcal X$ and $\mathcal Y$ is $n$.\\

\begin{figure}
\centering{}\input{IMGp+c_n10}\caption{\label{fig: simu}A realisation of $(M,N)$ for different values of $\sigma = \sigma_n = n^{-a}$, with $a=0, 0.5, 1, 1.5$. The match between parents traits (green points) and their offspring traits (purple diamonds) is graphically obvious for small $\sigma_n = n^{-1.5}$ but becomes more difficult if not impossible as $\sigma_n$ increases. In the statistical experiment generated by $(M,N)$, we are only given one horizontal line at a scale $\sigma_n$.}
\end{figure}
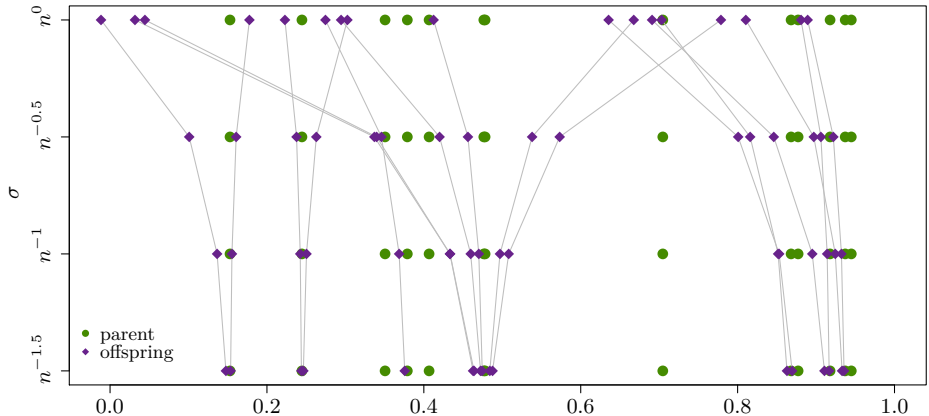

Keeping up with Section \ref{sec: motivation}, we study a statistical experiment of the form \eqref{eq: ref scale exp}, generated by the observation of $(\mathcal X,\mathcal Y)$ or equivalently $(M,N)$, with information $n$ and scale $\sigma$. The unknown parameter is $f$ and $\lambda, \mu$ are considered as nuisance parameters (assumed to be known for the moment). 
 %The key feature is that the relationships between offspring and their parents is not observed.  
Our aim is to reconstruct $f$ asymptotically in a minimax sense as $n \rightarrow \infty$, 
simultaneously for all scaling regimes $\sigma=\sigma_{n}$.

\subsection{Dispersal inference in applications} \label{sec: applications}

We briefly present some specific application dom\-ains compatible with the approach of dispersal inference as described in Section \ref{sec: dispersal inference setting}.  Admit\-tedly, further adjustments may be needed in order to be directly applicable; in particular, there might be a natural one-to-one
correspondence between parents and children or not.  
%Yet, we see that our methodology suggests potential interests in the following domains. 

\subsubsection*{Example 1: Service time estimation in M/G/$\infty$ queuing models}
Dating back to \citet{brown1970}, certain M/G/$\infty$ queuing models are embedded in our approach, see in particular the recents results of \citet{goldenshluger2015, goldenshluger2016, goldenshluger2019nonparametric}. Here, the state space of the parents $\mathcal O \subset \R$ represents time. The parent location trait is identified with an input arrival of a request to a server, according to Poisson arrivals at rate $\lambda$. Once the (random) service time of the request is fullfilled, an output is observed, that corresponds to the location trait of an offspring. See for instance \citet{baccelli2009inverse} where the emphasis is put on queueing systems where the service time cannot be observed. A small $\sigma$ compared to $1/\lambda$ indicates that the service time is small compared to the order of magnitude of a typical interarrival between two queries, in which case one may take the time between an input and an output as a proxy for the service time. Otherwise, this is no longer true and alternative methods must be sought. Most aforementioned studies assume $\sigma\lambda=1$. The case where $\sigma$ is larger than $1/\lambda$ has been adressed by \citet{BlanghapsEtAl2013} where it is still required that $\lambda\sigma$ is bounded.

The goal is to estimate the density of service time, that matches exactly with the dispersal density $f$ of our model. However, in the M/G/$\infty$ model, to an incoming call, one associate one output exactly, which is slightly more stringent than having $\mu = 1$ only. 
See also \citet{hall2004nonparametric} and Section \ref{sec: exactly one} for a more specific discussion in that direction. 

\subsubsection*{Example 2: Poisson random convolution in functional genomics}
This is actually the application that stimulated our approach, following informal discussions with our colleague Marie Doumic that are formulated in \citet{jiangEtAl2015}. See also the recent work by \citet{bonnet2022uniform}. 
The objective is to propose a model of distance interaction between motifs (or occurences of transcription regulatory elements) along DNA sequences. Related literature using point processes alternatives is developed for instance in \citet{gusto2005fado, carstensen2010multivariate}.

Dispersal inference proposes an alternative approach compatible with \citet{jiangEtAl2015}, at least at a conceptual level: along the DNA sequence, transcription binding sites are observed according to a Poisson rate $\lambda$ and serve as a parent generation model. Conditional to their parent location, transcription start sites (TSS) along the sequence are drawn via a random distribution $f$ that we wish to infer, the dispersal distribution. Depending on the dispersal scale $\sigma$, we are back to our original problem and we obtain a continuous nonparametric alternative to the model described in \citet{jiangEtAl2015}.

\subsubsection*{Example 3: Dispersal distance in plants genetics}

Introgression from cultivated to wild plants is a challenging problem for evolutionary ecology, especially in the context of genetically engineered crops. The study of gene flow from crops to wild relatives starts with understanding the typical dispersal distribution -- in a  spatial sense -- between plants and their offspring. Although our model is too simple to account for various heterogeneity in natural environment, we emphasise some encouraging similarities: in the study of \citet{arnaud2003evidence}, plants of interest and their offspring are distributed along a river bank. This accounts for a dispersal density a state space understood as a one-dimensional manifold (a curve), which is similar (and rate equivalent) to estimating a one-dimensional density, cf. \cite{berenfeld2021density}. 

Beyond the specific case of measuring dispersal along such idealised geometric features, the problem of estimating the distance between parents and saplings (accounting for seed dispersal from maternal or paternal parents plus pollen movement) is explicitly addressed in \citet{isagi2000microsatellite} via microsatellite analysis. Other plant based dispersal  issues such as seed versus pollen dispersal from spatial genetic structure are discussed in \citet{heuertz2003estimating}, see also \citet{lavorel1995dispersal}  and the references therein. 

\subsubsection*{Example 4: Estimating diffusivity based on counting occupation numbers of particles}
In a suspension of particles in a fluid, a Poissonian number of particles is recorded as they enter a fixed domain $\mathcal A$ and likewise when they leave $\mathcal A$. Applications in  
 fluctuation spectroscopy enables one to infer the diffusivity (or other parameters) from such counting data, assuming that the particles have velocity $(V_t)_{t \ge 0}$ with random dynamics governed by a diffusion process
 $$\d V_t =  -\beta V_t\d t+\sqrt{2\beta D}\d W_t$$
 where  $(W_t)_{t \ge 0}$ is a Wiener process and $\beta >0$ is a thermal relaxation parameter. There exist explicit formulae that relate the sojourn time of a particle within $\mathcal A$ and the diffusivity $D$, when the process is at equilibrium, see in particular \citet{binghamDunham1997}. We thus have a typical dispersal inference problem, where the dispersal density corresponds here to the sojourn time of the particles. See also the recent paper by \citet{goldenshlugerJacobovic2021}.

\subsection{Results and organisation of the paper} \label{sec: main results + orga}

\medskip{}

%We first restrict ourselves to dimension $d=1$.
% In order to face the competition between an apparently direct estimation problem for small $\sigma$, intuitively whenever $\sigma_n \ll n^{-1/d} = n^{-1}$, and a convolution-like problem when $\sigma_n \gg n^{-1/d}$, 
We first analyse the interaction between parents $\mathcal X$ and children $\mathcal Y$ via the correlation structure between the measures $M(\d x)$ and $N(\d y)$. In Proposition \ref{prop:expFormula} in Section \ref{sec:estimator} below,  building on the approach of \citet{goldenshluger2016},  we establish the formula
%thanks to Campbell's formula, that we use}:
\begin{equation} \label{eq: first rep}
\frac{1}{n\lambda \mu} \E\Big[\frac{M(\d x)N(\d y)}{\d x\d y}\Big] = n\lambda (f_{\sigma} \ast p)(y) + f_{\sigma}(x-y),
\end{equation}
where $p= {\eins_{[0,1]}}$ denotes the density function of the parent distribution. Formula \eqref{eq: first rep} reveals the competition between a direct approach and a convolution problem, as mentioned above. From the observation of $(M,N)$, we have access to empirical averages of the form 
$$\sum_{i,j}\varphi(X_i,Y_j) = \int_{[0,1] \times \R} \varphi(x,y)M(\d x)N(\d y),$$ 
for test functions $\varphi$. We can take advantage of the information given by the first term in the right-hand side of \eqref{eq: first rep} by picking $\varphi$ of the form $\varphi(x,y) = \psi(y)$ and thus ignoring the information given by the parent generation. For the second term, we pick $\varphi$ of the form $\varphi(x,y) = \psi\big((x-y)/\sigma)\big)$ and we can take benefit from the interplay between the parent generation and its offspring. This results in generic estimators of the form
\begin{equation} \label{eq: inspiration}
\sum_{i,j}\varphi^\star\big(Y_j,(X_i-Y_j)/\sigma\big)
\end{equation}
for a specific choice of $\varphi^\star$. 
%In particular, by picking $\varphi^\star\big(y,(x-y)/\sigma\big)$ either of the form $\psi(y)$ or $\psi\big((x-y)/\sigma)$, we obtain
%\begin{equation} \label{eq: rep convol}
%\tfrac{1}{n\lambda\mu}\E\Big[\sum_{j}\psi(Y_j)\Big] = n\sigma \lambda \int_\R \psi(y)(f_\sigma \ast p)(y)\d y  + \mathcal R_1(\sigma, \psi,f)
%\end{equation}
%and
%\begin{equation} \label{eq: rep direct}
%\tfrac{1}{n\lambda\mu}\E\Big[\sum_{i,j}\psi(X_i-Y_j)/\sigma\big)\Big] = \int_\R \psi(z)f(z)\d z+\sigma n \,{\mathcal R}_2(\sigma, \lambda, \psi, f),
%\end{equation}
%where the remainder terms $\mathcal R_i$ are bounded above by a constant times $\|f\|_\infty \|\psi\|_{L^1}$, independently of $\sigma$.
%Theses representations delineate two regimes: in \eqref{eq: rep convol}, whenever $\sigma \gg n^{-1}$, a convolution structure dominates and the parent data are irrelevant, whereas in \eqref{eq: rep direct} where the interaction between parents and children is used, we rather have a direct problem  whenever $\sigma \ll n^{-1}$. Note also that the bound on the remainder terms $\mathcal R_i$ involves the $L^1$-norm of the test function $\psi$, allowing for $\psi$ to depend on $n$ and concentrate to a Dirac mass at a given point without perturbing the different regimes, yielding the possibility to implement kernel methods. \\
Whereas these heuristics give an overall flavour of the statistical model structure, the general situation is more subtle.
% (note in particular the remaining dependence on $\sigma$ in the term $ \lambda \int_0^1 \psi(y)f_\sigma \ast p(y)\d y$ in the representation \eqref{eq: rep convol}).
In Section~\ref{sec:estimator}, we elaborate on the properties of the point process $(M,N)$ to obtain an estimator of $f(z_0)$ for an arbitrary point $z_0 \in \R$. It takes the form
\begin{equation*}\widehat f^\star_{h_1,h_2}(z_0) =\begin{cases}\displaystyle
                                   \frac{1}{n\lambda h_1}\hat{f}_{h_{1},h_{2}}(z_0),\qquad &\text{for large scales},\\
                                   \displaystyle\frac{1}{h_2}\hat{f}_{h_{1},h_{2}}(z_0)-\sigma n\lambda,\qquad&\text{for small scales}
                                 \end{cases}
\end{equation*}
where 
\begin{align*}
\hat{f}_{h_{1},h_{2}}(z_0)= & \frac{1}{n\lambda\mu\sigma h_{1}}\sum_{i,j}\psi'\Big(\frac{z_0}{h_{2}}-\frac{Y_{j}}{\sigma h_{2}}\Big){\psi}\Big(\frac{z_0}{h_{1}}-\frac{Y_{j}-X_{i}}{\sigma h_{1}}\Big)
\end{align*}
is inspired by \eqref{eq: inspiration}
%a nonparametric estimator that exploits 
%the representations \eqref{eq: rep convol} and \eqref{eq: rep direct} simultaneously arising from
%the correlation structure of $(M,N)$ across scales $\sigma$
for a suitable kernel $\psi$ (and $\psi'$ its derivative).
\begin{figure}
\centering{}\input{IMGrates-rev}
\caption{\label{fig:rates}Dependence of the optimal convergence rate $r_{n}$ on the
dispersal rate $\sigma_{n}$ in a $(\log_{n}\sigma_{n},\log_{n}r_{n})$-plot.}
\end{figure}
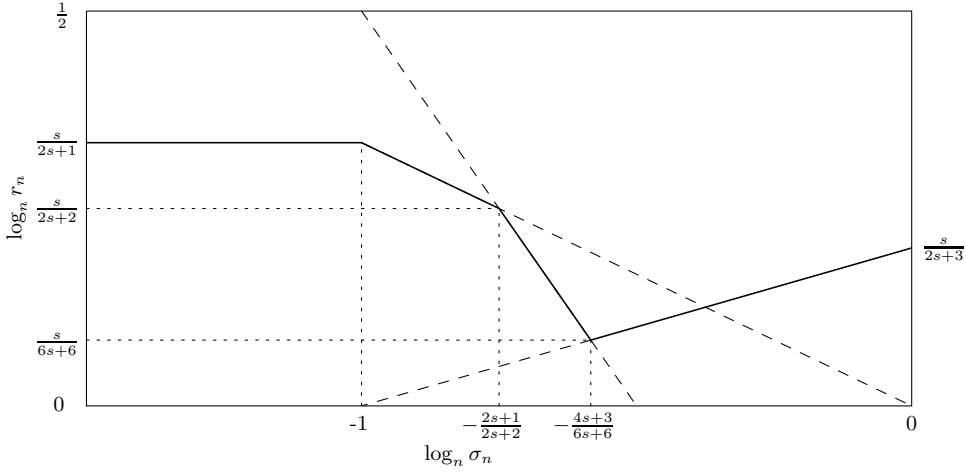
For an optimised choice of the bandwidths $h_1$ and $h_2$, we prove in Theorem \ref{thm: rates main}
that
\[
\sup_f\E\Big[\big(\hat{f}_{h_1,h_2}^\star(z_0)-f(z_0)\big)^{2}\Big]^{1/2} \lesssim r_{n},
\]
where the notation $A\lesssim B$ is equivalent to the Landau notation $A=O(B)$, the supremum is taken over H\"older balls of regularity $s>0$ locally around $z_0$, and 
\begin{equation}\label{eq:rate}
r_{n}=\begin{cases}
n^{-s/(2s+1)}, & \text{if }\sigma\le n^{-1},\\
\sigma^{s/(2s+1)}, & \text{if }\sigma\in[n^{-1},n^{-(2s+1)/(2s+2)}),\\
\sigma\sqrt{n}, & \text{if }\sigma\in\big[n^{-(2s+1)/(2s+2)},n^{-(4s+3)/(6s+6)}\big),\\
(n\sigma)^{-s/(2s+3)}, & \text{if }\sigma\in[n^{-(4s+3)/(6s+6)},1].
\end{cases}
\end{equation}
The shape of the rate of convergence $r_n = r_n(\sigma)$ as $\sigma = \sigma_n$ varies is illustrated in Figure \ref{fig:rates}. We prove in Theorem \ref{thm:lowerBound} that this result is indeed optimal:
\[
\liminf_{n\to\infty}\inf_{\hat{\theta}}\sup_{f}r_{n}^{-1}\E_{f}\Big[\big(\hat{\theta}-f(z_0)\big)^{2}\Big]^{1/2}>0,
\]
where the infimum is taken over all estimators $\widehat \vartheta$ built upon the point clouds $\mathcal X$ and $\mathcal Y$, and the supremum  is taken over H\"older balls of regularity $s>0$ locally around $z_0$.
%This proves that $r_n$ is the minimax rate of convergence for dispersal estimation in pointwise squared-error loss over H\"older smoothness classes. 

As illustrated in Figure \ref{fig:rates}, a direct estimation regime with (the usual) minimax rate $r_n = n^{-s/(2s+1)}$ dominates for $\sigma \ll n^{-1}$ (the far left side of the picture), whereas for fixed $\sigma$, we have $r_n = n^{-s/(2s+3)}$, {\it i.e.} the minimax rate of convergence of an inverse problem of order one (the far right side of the picture). However, when $\sigma_n$ slowly goes to $0$, the inverse problem minimax rate deteriorates to $(n\sigma_n)^{-s/(2s+3)}$. Surprisingly, other regimes appear in the intermediate regime $\sigma_n \in [n^{-1}, n^{-(4s+3)/(6s+6)}]$. In particular, we find a worst case region, around the scale $\sigma_n \approx n^{-(4s+3)/(6s+6)}$ that yields the exotic minimax rate $n^{-s/(6s+6)}$. We discuss this phenomenon in detail in Section \ref{sec:rates}. In Section \ref{sec: unknown sigma}, we consider the case of an unknown scale $\sigma = \sigma_n$. We first show that it is possible to estimate $\sigma$ %, {\color{red}even} in the microsopic regime $\sigma \ll n^{-1}$ 
so that we can ultimately decide whether $n\sigma$ is sufficiently large to apply the $n\sigma \rightarrow \infty$ asymptotics or not. We establish in particular in Section \ref{sec: esti unknown sigma} a bound for the relative error $(\widehat \sigma-\sigma)/\sigma$ of our estimator $\widehat \sigma$. This is the gateway for a plug-in strategy to estimate $f$ optimally when $\sigma$ is unknown and considered as a nuisance parameter as we demonstrate in Section \ref{sec: esti f plugin}. The sensitivity of the plug-in estimator $\widehat f^\star_{h_1,h_2}(z_0) = \widehat f^\star_{h_1,h_2}(z_0)(\widehat \sigma)$ is controlled  via the smoothness of the process $\sigma \mapsto \widehat f^\star_{h_1,h_2}(z_0)(\sigma)$ via a chaining argument based on Kolmogorov-Chentsov criterion. We show that the optimal rates are achievable in probability.\\

The rest of the paper is organised as follows: In Section~\ref{sec:estimator}
we construct an estimator of $f(z_0)$ that takes $\mathcal X$ and $\mathcal Y$ as inputs and that adjusts to the scale $\sigma = \sigma_n$. Convergence rates for the estimator and matching minimax lower bounds are given in Section~\ref{sec:rates} and Section~\ref{sec:lowerBound}, respectively. The estimation of $\sigma$ is studied in Section \ref{sec: esti unknown sigma} while the estimation of $f$ when $\sigma$ is unknown via plug-in is undertaken in Section \ref{sec: esti f plugin}. A discussion with possible extensions is the content of Section~\ref{sec: discussion}. 
%In particular, we develop alternative estimation approaches in a slightly simpler model in Section~\ref{sec:naive}. 
A short numerical simulation study is proposed in Section \ref{sec: numerical example}. All proofs are postponed to Section~\ref{sec:Proofs}.

\section{Main results}

\subsection{Construction of estimators across scales\label{sec:estimator}}

\subsubsection*{The correlation structure of $(M,N)$}

%As informally explained in Section \ref{sec: main results + orga}, 
Our starting point is the analysis of the correlation between $M$ and $N$, inspired by the approach developed in \citet{goldenshluger2016}. 
Yet, in our study, the fact 
that the parent data $\mathcal X$ are distributed on
a bounded interval has a considerable
impact
on the correlation structure of $(M,N)$. 

\begin{prop} 
\label{prop:expFormula}Let $(A_{i})_{1 \le i \le I}$ and $(B_{j})_{1 \le j \le J}$
be two families of disjoint Borel subsets of $[0,1]$ and $\R$, respectively.
Then for any $(\eta_{1},\dots,\eta_{I})\in\R^{I}$ and $(\xi_{1},\dots,\xi_{J})\in\R^{J}$
we have
\begin{align}
 & \log\E\Big[\exp\Big(\sum_{i=1}^{I}\eta_{j}M(A_{i})+\sum_{j=1}^{J}\xi_{j}N(B_{j})\Big)\Big]\label{eq:campbell}\\
 & \quad=n\lambda\sum_{i=1}^{I}(e^{\eta_{i}}-1)|A_{i}|+n\lambda\int_{0}^{1}\Big(\exp\Big(\mu\sum_{j=1}^{J}(e^{\xi_{j}}-1)\int_{B_{j}}f_{\sigma}(y-x)\d y\Big)-1\Big)\d x\nonumber \\
 & \qquad\qquad+n\lambda\sum_{i=1}^{I}(e^{\eta_{i}}-1)\int_{A_{i}}\Big(\exp\Big(\mu\sum_{j=1}^{J}(e^{\xi_{j}}-1)\int_{B_{j}}f_{\sigma}(y-x)\d y\Big)-1\Big)\d x,\nonumber 
\end{align}
where $|A|$ denotes the Lebesgue measure of $A \subset \R$. 
\end{prop}
The proof relies on Campbell's exponential formula and is postponed to Section \ref{sec: proof campbell}. By differentiating the result of Proposition~\ref{prop:expFormula}, we obtain the following explicit representation of the correlation structure of $M$ and $N$. 
\begin{cor} 
\label{cor: intensity}
For any Borel sets $A\subset[0,1]$ and (bounded) $B\subset\R$
we have
\begin{align*}
\E\big[M(A)N(B)\big] & =n^{2}\lambda^{2}\mu|A|\int_{0}^{1}\int_{B}f_{\sigma}(y-x)\,\d y\,\d x+n\lambda\mu\int_{A}\int_{B}f_{\sigma}(y-x)\,\d y\,\d x.
\end{align*}
%For infinitesimal intervals $M(dp)=M((p,p+dp]),p\in(-1,1),$ and $N(dc)=N((c,c+dc]),c\in\R$,
%we obtain
%\[
%\E[N(dc)]=n\lambda\mu(f_{\sigma}\ast\eins_{[0,1]})(c)\d c\quad\text{and}\quad\E[M(dp)N(dc)]=n\lambda\mu\big(n\lambda(f_{\sigma}\ast\eins_{[0,1]})(c)+f_{\sigma}(c-p)\big)\d c\,\d p.
%\]
\end{cor}
Corollary \ref{cor: intensity} reveals the infinitesimal correlation structure
$$\E[M(\d x)N(\d y)]=n\lambda\mu\big(n\lambda(f_{\sigma}\ast p)(y)+f_{\sigma}(y-x)\big)\d y\,\d x,$$
with $p = \eins_{[0,1]}$. Applied to  a well-behaved test function $\varphi\colon [0,1] \times \R \rightarrow \R$ we obtain
\begin{align*}
 & \frac{1}{n\lambda \mu}\E\Big[\int_{[0,1]\times \R}\phi(x,y)M(\d x)N(\d y)\Big] \\
 %&=\int_{0}^{1}\int_{\R}\Big(n\lambda\phi(x,y)(f_{\sigma}\ast\eins_{[0,1]})(y)+\phi(x,y)f_{\sigma}(y-x)\Big)\d y\d x\\
 &\quad=n\lambda\int_{0}^{1}\int_{\R}\phi(x,y)(f_{\sigma}\ast p)(y)\d y\d x+\int_{0}^{1}\int_{\R}\phi(x,x+\sigma z)f(z)\d z\d x.
\end{align*}
In order to obtain information on  $f$ from the first integral, the input function $\varphi(x,y)$ should depend on $y$ solely, while in the second integral, $\phi(x,y)$ should rather depend on
$(y-x)/\sigma$. We therefore pick a test function of the form 
$$\varphi^\star(x,y) = \psi_1(y)\psi_2\big((y-x)/\sigma\big).$$
%where the $\psi_i\colon\R \rightarrow \R$ are well-behaved functions, to be adjusted later. 
In the limit $n \rightarrow \infty$, we expect the empirical mean to be close to its expectation so that the approximation
$$\sum_{i,j}\phi^{\star}(X_{i},Y_{j}) \approx \E\big[\sum_{i,j}\phi^{\star}(X_{i},Y_{j})\big]$$
is valid. Hence, we asymptotically have access to 
\begin{equation}
\frac{1}{n\lambda \mu}\E\Big[\int_{[0,1] \times \R}\psi_1(y)\psi_2\big((y-x)/\sigma\big)M(\d x)N(\d y)\Big]
  =\sigma n\lambda\, \mathcal U_\sigma(f \ast p)+ \mathcal V_\sigma(f),
 \label{eq:exp}
\end{equation}
where
$$\mathcal U_\sigma(f \ast p) = \int_{\R}\psi_1(y)(\psi_2\ast\eins_{[0,1/\sigma]})(y/\sigma)(f_{\sigma}\ast p)(y)\d y$$
and
$$\mathcal V_\sigma(f)=  \int_{\R}(\psi_1\ast\eins_{[-1,0]})(\sigma z)\psi_2(z)f(z)\d z.$$
It is noteworthy that 
$$|\mathcal U_\sigma(f \ast p)| \le \|\psi_1\|_{L^1}\|\psi_2\|_{L^1}\;\;\text{and}\;\;|\mathcal V_\sigma(f)| \le\|f\|_{\infty}\|\psi_1\|_{L^1}\|\psi_2\|_{L^1},$$
showing that {\bf 1)} the functionals $\mathcal U_\sigma$ and $\mathcal V_\sigma$ are not sensitive to the order of magnitude of $\sigma$, and {\bf 2)} the influence of the test functions $\psi_i$ is bounded in $L^1$-norm, hence they can be subsequently chosen as kernels that weakly converge to a Dirac mass as $n \rightarrow \infty$.

%The previous formula reveals two sources of information
%about the dispersal density $f$ across scale in a transparent way: in the functional $\mathcal U_\sigma(f \ast p)$, the function $f$ is convolved with the parent distribution $p$ such that we face deconvolution
%problem. This deconvolution should be implemented via the test function $H$ picked as a suitable kernel,
%that only acts on the children distruntion. In the functional $\mathcal V_\sigma(f)$, the function $f$ is directly accessible via the test function $K$ which depends
%on the interaction of the parents and the children. The different
%weights of these two integrals $\sigma n$ and $1$, respectively,
%indicate that it will not be possible to use both approaches at the
%same time. Depending on the product $\sigma n$ that will serve as a critical scale, one term will dominate
%the other. 

\subsubsection*{The deconvolution approach via $\mathcal U_\sigma(f \ast p)$}

Pick $\psi_2=1$ as a constant function in \eqref{eq:exp} to obtain 
\begin{align*}
\E\big[|\mathcal X|\sum_{j}\psi_1(Y_{j})\big] & =n^2\lambda^2\mu\int_{\R}\psi_1(y) (f_{\sigma}\ast p)(y)\d y+n\lambda \int_0^1\int_{\R} \psi_1(x+\sigma z)f(z)\d z\d x \\
& = n^2\lambda^2\mu\Big(\int_{\R}\psi_1(y) (f_{\sigma}\ast p)(y)\d y +O(n^{-1})\Big),
\end{align*}
where $|\mathcal X| = M([0,1]) = \lambda n +O_{\P}(n^{1/2})$ is the total (random) number of parents. Ignoring remainder terms and using the fact that $|\mathcal Y| = N(\R) = n\lambda \mu + O_{\P}(n^{1/2})$ we also have the approximation
\begin{equation} \label{eq: empirical distrib}
\frac{1}{|\mathcal Y|}\sum_{j}\psi_1(Y_{j}) \approx \int_{\R}\psi_1(y)(f_{\sigma}\ast p)(y)\d y .
\end{equation}
The empirical estimate \eqref{eq: empirical distrib} is transparent: each child with trait $Y_{j}$ has a parent with trait $X_{i_j}$
such that $Y_{j}=X_{i_j}+\sigma D_{j}$, where $D_{j}$ is distributed according to the dispersal density $f$. With $X_{i_j} \sim p$, we obtain
$Y_{j}\sim f_{\sigma}\ast p$.
However, the parent trait distribution $p$ is uniform on $[0,1]$, its Fourier transform oscillates and vanishes on a discrete set, hence a classical deconvolution estimators based
on spectral approaches cannot be readily applied. While there are some general
constructions in the literature to overcome this problem (see {\it e.g.} \citet{meister2007deconvolution, delaigleMeister2011, belomestnyGoldenshluger2021} and the references therein), we take a more explicit route: we elaborate on the  approach of \citet{groeneboomJongbloed2003}, relying on the specific structure of a uniform $p = \eins_{[0,1]}$. (The case of more general parent trait distribution is discussed in Section \ref{sec: other parent trait distrib}.) Denoting by $F$ the cumulative distribution of $D_{ij}$ and writing $g_{\sigma}=f_{\sigma} \ast p$, we have
\begin{align*}
g_{\sigma}(y)= & \int_{\R}\eins_{[0,1]}(y-z)f_{\sigma}(z)\d z
=  \int_{\R}\eins_{[y-1,y]}(z)f_{\sigma}(z)\d z=F\big(\tfrac{y}{\sigma}\big)-F\big(\tfrac{y-1}{\sigma}\big),
\end{align*}
hence the representation
\begin{align}
F\big(\tfrac{y}{\sigma}\big) & =  g_{\sigma}(y)+F\big(\tfrac{y-1}{\sigma}\big) = g_{\sigma}(y)+g_{\sigma}(y-1)+F\big(\tfrac{y-2}{\sigma}\big) =  \dots=\sum_{\ell \ge 0}g_{\sigma}(y-\ell), \label{eq:dutchFormular}
\end{align}
valid for every $y\in\R$. Based on the observation $\mathcal Y$, the density $g_{\sigma}$ can be estimated at $z_0 \in \R$
by a kernel density estimator with kernel $K$:
and bandwidth $h >0$ 
$$\widehat g_{\sigma,h}(z_0) = \frac{1}{|\mathcal Y|}\sum_{j}\frac{1}{h}K\Big(\frac{\sigma z_0-Y_j}{h}\Big).$$
We  then use representation \eqref{eq:dutchFormular}
to obtain an estimator of $F(z_0)$ via
\begin{equation} \label{eq: dutch esti}
\widehat{F}_{h}(z_0)=\frac{1}{|\mathcal Y|}\sum_{j}\sum_{\ell \ge 0}\frac{1}{h}K\Big(\frac{\sigma z_0-\ell-Y_j}{h}\Big).
\end{equation}
Note that for compactly supported kernels
the sum in $\ell$ is finite. For simplicity, we further consider the case
where $f$ is compactly supported and will adjust our assumptions accordingly. With no loss of generality, we assume 
$$\supp f\subset  \Big[-\frac{1}{2},\frac{1}{2}\Big]$$ 
so that only the term for $\ell=0$ in \eqref{eq:dutchFormular} is non-zero and we can omit all terms with $\ell\ge1$ in \eqref{eq: dutch esti}.

\iffalse

An analysis of $\widehat{F}_{h}$ based on a simplified model with independent and identically distributed children traits $(Y_j)_j$, that may have some interest in itself is proposed in Section \ref{sec:naive}.\\

\fi

% However, in our dispersal model, the offspring traits observed via $\mathcal Y$ are not independent (several children may have the same parent).\\

We then take the derivative of $\widehat{F}_{h}(z_0)$ in \eqref{eq: dutch esti} and the choice of bandwidth $h=\sigma h_1$ that scales with $\sigma$ and that will prove technically convenient. We finally obtain a \emph{deconvolution estimator} of $f(z_0)$ by setting
$$\widehat f_{h_1}^{\tt dec}(z_0) = \frac{1}{|\mathcal Y|}\sum_j\frac{1}{\sigma h_1^2}K'\Big(\frac{z_0}{h_{1}}-\frac{Y_j}{\sigma h_{1}}\Big).$$
We recover the representation \eqref{eq: empirical distrib} with
\begin{equation} \label{eq: def H}
\psi_1 = \frac{1}{\sigma h_1^2}K'\Big(\frac{z_0}{h_{1}}-\frac{\cdot}{\sigma h_{1}}\Big).
\end{equation}
Note that the convolution term in (\ref{eq:exp}) is uninformative
if $M$ is a homogeneous Poisson point process on whole real line
as in \citet{goldenshluger2016} or on the torus as in \citet{jiangEtAl2015}.

\subsubsection*{The interaction approach via $\mathcal V_\sigma(f)$}

While the deconvolution approach ignores the infor\-mation of the parents,
an estimator based on the interaction of parents and their offspring can
be constructed via $\psi_2$, taking now $\psi_1=1$ to be constant. From \eqref{eq:exp}, we obtain
\begin{align*}
&\E\Big[\sum_{i,j}\psi_2\Big(\frac{Y_{j}-X_{i}}{\sigma}\Big)\Big] \\
&=n\lambda\mu\int_{\R}\psi_2(z)f(z)\d z+n^2\sigma\lambda^{2}\mu
\int_{\R}(\psi_2\ast\eins_{[0,1/\sigma]})(y/\sigma)(f_{\sigma}\ast p)(y)\d y
\\
& = n\lambda \mu\Big(\int_{\R}\psi_2(y)f(z)\d z+O(n\sigma)\Big).
\end{align*}
 The bias is small only if $n\sigma$ is small, a result which is consistent with the heuristics of Section \ref{sec: dispersal inference setting}. Beyond that scale, as soon as $\sigma \approx n^{-1}$ the situation is a bit more involved. More specifically, when $\sigma \ll n^{-1}$, the offspring traits concentrate around their parents:
we expect roughly to have $n\sigma h$ parental traits in a $\sigma h$-neighbourhood of the trait of a child $Y_{j}=X_{i_j}+\sigma D_{j}$. With overwhelming probability only the true parent trait $X_{i_j}$ of $Y_{j}$ is actually present in this neighbourhood. Then the sum in $i$ over all parents vanishes and we expect the approximation
$$
\sum_{i,j}\psi_2\big((Y_j-X_i)/\sigma\big)\approx\sum_{j}\psi_2(D_{j})
$$
to be valid, while the sum is of order $n\lambda \mu$. More precisely, we expect
\begin{equation} \label{eq: empir dist bis}
\frac{1}{|\mathcal Y|}\sum_{i,j}\psi_2\big((Y_j-X_i)/\sigma\big) \approx \int_{\R}\psi_2(z)f(z)\d z.
\end{equation}
Picking a kernel density estimator with kernel $K$ and bandwidth $h_2>0$, we
obtain an \emph{interaction estimator} of $f(z_0)$ by setting
\begin{equation} \label{eq: esti int}
\widehat{f}_{h_2}^{\tt int}(z_0)=\frac{1}{|\mathcal Y|}\sum_{i,j}\frac{1}{h_{2}}K\Big(\frac{z_0}{h_{2}}-\frac{Y_{j}-X_{i}}{\sigma h_{2}}\Big).
\end{equation}
The representation \eqref{eq: empir dist bis} is recovered with
\begin{equation} \label{eq: def K}
\psi_2 = \frac{1}{h_{2}}K\Big(\frac{z_0}{h_{2}}-\frac{\cdot}{h_{2}}\Big).
\end{equation}
Whenever $n\sigma \gtrsim 1$, the relevance of a procedure like \eqref{eq: esti int} is less obvious. In particular, it is not clear whether the parent traits in the neighbourhood of
an offsping can be used to estimate $f$. For $n\sigma=1$ (and $\lambda=\mu=1$) \citet{brown1970}
constructed an estimator based on an explicit formula that relates
$f$ to the distribution function of the distance of an offspring to its nearest parents. 
\iffalse

We discuss this approach in Section \ref{sec:naive} in our context. 

\fi
An
interaction estimator was also applied by \citet{goldenshluger2016}
in a setting where the intensity measure of $N$ is the Lebesgue measure
on whole real line and which corresponds to $n\sigma=1$. As we will
see below the interaction estimator is still applicable if $\sigma>1/n$
as long as $\sigma$ is not too large. However, the interaction estimator then requires to incorporate a non-trivial kernel $\psi_1$ and a bias correction. %A naive estimator could rely on simply counting the number of parents in a neighbourhood of the trait of a children, as already mentioned in Section \ref{sec: dispersal inference setting}.
\iffalse
, see Section \ref{sec:naive}.
\fi

\subsubsection*{An estimator across scales}

Thanks to the heuristics developed for the construction of $\widehat f_{h_1}^{\tt dec}(z_0)$ and $\widehat{f}_{h_2}^{\tt int}(z_0)$, we are ready to implement an estimator across all scales $\sigma$, when the scale $\sigma$ is known. We consider the case of an unknown $\sigma$ in Section \ref{sec: unknown sigma} below. We first define 
\begin{align*}
\hat{f}_{h_{1},h_{2}}(z_0)=  & \frac{1}{n\lambda\mu\sigma h_{1}}\sum_{i,j}K'\Big(\frac{z_0}{h_{1}}-\frac{Y_{j}}{\sigma h_{1}}\Big)K\Big(\frac{z_0}{h_{2}}-\frac{Y_j-X_{i}}{\sigma h_{2}}\Big),
\end{align*}
where $K$ is a smooth compactly supported kernel with derivative $K'$. We formally retrieve the preceding representation
$\hat f_{h_1,h_2}(z_0)=\frac{h_1h_2}{n\lambda\mu}\sum_{i,j}\psi_1(Y_j)\psi_2\big(\frac{Y_j-X_{i}}{\sigma}\big)$  with $\psi_1$ defined in \eqref{eq: def H} and $\psi_2$ in \eqref{eq: def K}.
Next, we elaborate on the properties we require  for the kernel function $K$:
\begin{assumption} \label{ass: kernel}
The function $K\colon \R\rightarrow \R$ is differentiable, symmetric, bounded and satisfies
$$\mathrm{Supp}(K) \subset [-1,1],\;\;K(z)=1\;\;\text{for}\;|z| \le \tfrac{1}{4},$$
$$\int_{[-1,1]}z^\ell K(z)\d z = \eins_{\{\ell = 0\}}\;\;\text{for}\;\;\ell = 0,\ldots, \ell_K$$
for some $\ell_K \ge 0$ (the order of the kernel $K$).
\end{assumption}
For $\ell_K = 0$ or $1$, which is  generally sufficient in practice, Assumption \ref{ass: kernel} is simply obtained from any suitably dilated and translated compactly supported symmetric (even) density function, see Section \ref{sec: numerical example}. 
%\MT{ [This assumption is not completely standard; shall we say a word/give a reference on building kernels with all sort of properties (Tsybakov?)]}
Finally we define the appropriate normalisations and bias corrections
that need to be tuned depending on the relevant scale. We proceed as follows:
\begin{definition} \label{def: esti final}
Let $K$ satisfy Assumption \ref{ass: kernel}. We define the following estimators across scales:
\begin{itemize}
\item[(i)] (Deconvolution or large scales.) For $h_{1}\in[(\sigma n)^{-1},1]$ and $h_{2}=8/\sigma$ set
\begin{align}
\hat{f}_{h_{1}}^{(1)}(z_0) & =  \frac{1}{n\lambda h_1}\hat{f}_{h_1,8/\sigma}(z_0)\notag \\
&=\frac{1}{\sigma n\lambda\mu h_{1}^{2}}\sum_{j}K'\Big(\frac{z_0}{h_{1}}-\frac{Y_{j}}{\sigma h_{1}}\Big)\Big(\frac{1}{n\lambda}\sum_iK\Big(\frac{\sigma z_0}{8}-\frac{Y_{j}-X_i}{8}\Big)\Big).\label{eq:fHat1}
\end{align}
\item[(ii)] (interaction or small scales.) Let $\sigma<1/8$, $h_{1}=1/(2\sigma)$. We set for $h_{2}\in(0,1]$:
\begin{align}
\hat{f}_{h_{2}}^{(2)}(z_0) & =\frac{1}{h_2}\hat{f}_{1/(2\sigma),h_2}(z_0)- \sigma n\lambda \notag\\
& =\frac{1}{n\lambda\mu h_{2}}\sum_{i,j} 2K'\big(2(\sigma z_0-Y_{j})\big) K\Big(\frac{z_0}{h_{2}}-\frac{Y_{j}-X_{i}}{\sigma h_{2}}\Big)-\sigma n\lambda.\label{eq:fHat2}
\end{align}
\end{itemize}
\end{definition}
For the deconvolution estimator $\hat f_{h_1}^{(1)}(z_0)$ we could also set $\psi_2=1$. In this case the second factor in the right-hand side of \eqref{eq:fHat1} equals $\frac{|\mathcal X|}{n\lambda}\approx 1$ and we recover $\hat f_{h_1}^{\tt dec}(z_0)$ from above. A similar simplication for $\hat f_{h_2}^{(2)}(z_0)$ is not possible across all scales. As soon as $\sigma n\to\infty$, the small scales estimator crucially benefits from the specific structure of $\psi_1$ which excludes all offspring traits $Y_j$ outside an annulus with radius of order $1/\sigma$, see Proposition~\ref{prop:bias2}\emph{(ii)}   for details.

\subsection{Rates of convergence\label{sec:rates}}
Recall that, given a (small) neighbourhood $U_{z_0}$ of $z_0 \in \R$, the function $f\colon \R \rightarrow \R$ belongs to the local H\"older class $\mathcal H^{s}(z_0)$ with $s >0$
if $f$ is $\lfloor s\rfloor$ times continuously differentiable for every $z,z' \in  U_{z_0}$ and
\begin{equation} \label{def holder}
|f^{(\lfloor s \rfloor)}(z)-f^{(\lfloor s \rfloor)}(z')| \le C|z-z'|^{s - \lfloor s \rfloor},
\end{equation}
where $\lfloor s \rfloor$ is the largest integer {\it stricty smaller} than $s$, and $f^{(n)}$ denotes $n$-fold derivation (with $f^{(0)} = f$).
The definition depends on $U_{z_0}$, further omitted in the notation. We obtain a semi-norm 
$|f|_{\mathcal H^s(z_0)}$
by taking the smallest constant $C$ for which \eqref{def holder} holds.
Moreover, as explained in Section \ref{sec:estimator}, we assume for technical convenience that $f$ is bounded and supported in $[-\frac12,\frac12]$ which yields the following nonparametric class of densities\footnote{We further omit a slight ambiguity: the neighbourhood $U_{z_0}$ in the definition of $|f|_{\mathcal H^s(z_0)}$ is implicitly taken independently of $f$.}
\[
  \mathcal G^s(z_0,L):=\Big\{f\colon\R\to [0,\infty):|f|_{\mathcal H^s(z_0)}\le L,\|f\|_\infty\le L,\mathrm{Supp}(f)\subset[-\tfrac12,\tfrac12],\int f(z)\d z=1\Big\}.
\]
We first exhibit rates of convergence for $\hat{f}_{h_{1}}^{(1)}(z_0) $ and $\hat{f}_{h_{2}}^{(2)}(z_0) $ of Definition \ref{def: esti final} built upon $\hat{f}_{h_{1},h_{2}}(z_0)$. 
\begin{prop}
\label{thm:rates} Let $K$ satisfy Assumption \ref{ass: kernel} with $\ell_K\ge \lfloor s\rfloor$.
 
\begin{enumerate}
\item If $h_1\le 1\le \sigma n$, then we have for  any $z_0\in(-\frac{1}{2},\frac{1}{2})$
\begin{equation}
\sup_{f\in\mathcal G^s(z_0,L)}\E\Big[\big(\hat{f}_{h_{1}}^{(1)}(z_0)-f(z_0)\big)^{2}\Big]^{1/2}\lesssim h_{1}^{s}+\big(n\sigma h_{1}^{3}\big)^{-1/2},\label{eq:rateDecon}
\end{equation}
up to a constant that depends on $L$, $s$, $K$ and $z_0$. Choosing $h_{1}=(n\sigma)^{-1/(2s+3)}$, we obtain the optimised rate 
$$\E\Big[\big(\hat{f}_{h_{1}}^{(1)}(z_0)-f(z_0)\big)^{2}\Big]^{1/2}\lesssim(n\sigma)^{-s/(2s+3)}.$$ 
\item Let $\sigma<1/8$. For any $h_{2}\in(0,1]$ and $z_0\in(-\frac{1}{2},\frac{1}{2})$, we have
\begin{equation}
\sup_{f\in\mathcal G^s(z_0,L)}\E\Big[\big(\hat{f}_{h_{2}}^{(2)}(z_0)-f(z_0)\big)^{2}\Big]^{1/2}\lesssim h_{2}^{s}+\max\big((nh_{2})^{-1/2}, \sigma^{1/2} h_{2}^{-1/2}, n^{1/2}\sigma\big),\label{eq:rateLocal}
\end{equation}
up to a constant that depends on $L$, $s$, $K$ and $z_0$. Choosing $h_{2}=(n\wedge\sigma^{-1})^{-1/(2s+1)}$, we obtain the optimised rate 
$$\E\Big[\big(\hat{f}_{h_{2}}^{(2)}(z_0)-f(z_0)\big)^{2}\Big]^{1/2}\lesssim \max\big( (n\wedge\sigma^{-1})^{-s/(2s+1)}, n^{1/2}\sigma \big).$$
\end{enumerate}
\end{prop}
Some remarks are in order: {\bf 1)} The rate $(n\sigma)^{-s/(2s+3)}$ in \emph{(i)} reflects the ill-posedness
of degree one due to the convolution with the indicator function.
Moreover, we see that the rate is determined by $n\sigma$ instead
of $n$ solely: the information about $f$ is concentrated at the boundary $[-\frac{\sigma}{2},\frac{\sigma}{2}]\cup[1-\frac{\sigma}{2},1+\frac{\sigma}{2}]$
of the support of the parent distribution since in the interior we
have $\eins_{[0,1]}\ast f_{\sigma}(y)=1$ for all $y\in(\frac{\sigma}{2},1-\frac{\sigma}{2})$.
Since the number of children in this boundary is of order $n\sigma$,
the latter can be understood as effective sample size. In particular,
the convergence rate deteriorates for $\sigma\to0$ and the deconvolution
estimator is only consistent as long as $n\sigma\to\infty$. {\bf 2)} In \emph{(ii) }we obtain the classical rate of convergence $n^{-s/(2s+1)}$
for nonparametric density estimation as long as $\sigma n \lesssim 1$.
For large scaling factors the bias correction $-\sigma n\lambda$
becomes crucial and the rate gets slower, {\it i.e.} the local interaction
between parents and children becomes less informative. We obtain the
rate $\sigma^{s/(2s+1)}\vee(\sqrt n\sigma)$. In particular, the
interaction approach is only consistent as long as $\sigma=o(n^{-1/2})$.
This limitation is a consequence of the non-negligible correlations between two different offspring traits in a $\sigma$-neighborhood of a parent. %[@Marc: That is the term $J_{4,4}$ in the variance proposition].}  
{\bf 3)} Interestingly, there is an intermediate regime $\sigma\in[n^{-1},n^{-1/2}]$
where both approaches are applicable and we can choose the estimator
with the faster rate.\\ 

We wrap together the results of Proposition \ref{thm:rates} to obtain our main result:
\begin{thm} \label{thm: rates main}
Let $s,L>0$ and let $K$ satisfy Assumption \ref{ass: kernel} with $\ell_K\ge \lfloor s\rfloor$. For any $z_0 \in (-1/2,1/2)$, there exists an estimator
$\widehat{f}(z_0)$ depending on $\sigma,\lambda,\mu$ and $s$, explicitly obtained from Proposition \ref{thm:rates} above such that
\[
\sup_{f\in\mathcal G^s(z_0,L)}\E\Big[\big(\widehat{f}(z_0)-f(z_0)\big)^{2}\Big]^{1/2} \lesssim r_{n},
\]
up to a constant that depends on $L$, $s$, $K$ and $z_0$, and
with rate of convergence from \eqref{eq:rate}.
%the rate of convergence for $r_n$ is given by
%\begin{equation}
%r_{n}=\begin{cases}
%n^{-s/(2s+1)}, & \text{if }\sigma\le n^{-1}\\
%\sigma^{s/(2s+1)}, & \text{if }\sigma\in[n^{-1},n^{-(2s+1)/(2s+2)}),\\
%\sigma\sqrt{n}, & \text{if }\sigma\in\big[n^{-(2s+1)/(2s+2)},n^{-(4s+3)/(6s+6)}\big),\\
%(n\sigma)^{-s/(2s+3)}, & \text{if }\sigma\in[n^{-(4s+3)/(6s+6)},1].
%\end{cases}\label{eq:rate}
%\end{equation}
\end{thm}

Some remarks again: {\bf 1)} The graph of $\log r_n$ as a function of $\tau$ for $\sigma=n^{-\tau}$
is illustrated in Figure~\ref{fig:rates}. Quite surprisingly, the
dependence of the convergence rate on the scaling parameter $\sigma$
is not monotonic which is a consequence of (\ref{eq:exp}): The information
on $f$ in the deconvolution term decreases if $\sigma$ gets smaller,
while the second information based on interaction decreases if
$\sigma$ gets larger. The elbows between the regimes correspond to
the points where
\[
\sigma=n^{-(2s+1)/(2s+2)}\;\;{\it i.e.}\;\;\sqrt{n}\sigma=\sigma^{s/(2s+1)}
\]
and 
\[
\sigma=n^{-(4s+3)/(6s+6)}\;\;{\it i.e.}\;\;\sqrt{n}\sigma=(n\sigma)^{-s/(2s+3)}.
\]
In particular, the best estimator uses the deconvolution approach
if $\sigma>n^{-(4s+3)/(6s+6)}$ and the interaction approach otherwise. 
{\bf 2) }For the construction of the estimator, we need to know $\lambda,\mu$
and $\sigma$. A canonical estimator for $\lambda$ is given by $\hat{\lambda}=n^{-1}|\mathcal X| = n^{-1}M([0,1])\sim n^{-1}\mathrm{Poiss}(\lambda n)$
satisfying $\E[|\hat{\lambda}/\lambda-1|^{2}]=(n\lambda)^{-1}$. 
%Hence,
%it should not be a problem to replace $\lambda$ by $\hat{\lambda}$
%in the weighting of $\hat{f}_{h_{1},h_{2}}$. Similarly, we can estimate
%$\mu$. 
The scaling parameter $\sigma$ is more critical, because 
even the parametric accuracy is not sufficient to construct an
estimator which is adaptive in $\sigma$: We have to decide
whether $\sigma>n^{-(4s+3)/(6s+6)}$ or not and the boundary $n^{-(4s+3)/(6s+6)}$
is $o(n^{-1/2})$ as soon as $s>2$. 
%(additionally the boundary dependence
%on the typically unknown regularity $s$). 
Since usual construction
principles for adaptive estimators rely on a monotonic dependence
of the rate, more precisely of an upper bound for the stochastic error,
the observed dependence on $\sigma$ might complicate the construction
of an adaptive estimator considerably, see Section~\ref{sec: unknown sigma}.

\subsection{Minimax optimality}\label{sec:lowerBound}

For $\sigma\le1/n$ the rate $n^{-s/(2s+1)}$  is optimal: 
a lower bound is obtained by noting that it is more
informative to observe the point cloud of the parent traits $\mathcal Y$ and the dispersal realisations $(D_{j})$ via a Poisson
point process with intensity $n\lambda\mu f$. The offspring point process $N$
 can subsequently be constructed via uniformly distributing the
children around the parents. Observing $(D_{j})$, the classical minimax
rate for estimating $f(z_0)$ is $n^{-s/(2s+1)}$. Less obviously,  for $n\sigma\to\infty$, the rate $r_n$ is optimal too:
% As the following
%£theorem verifies, the obtained rates are optimal,
%too.

\begin{thm}
\label{thm:lowerBound} Let $z_0 \in(-1/2,1/2)$, $s>0$ and $L \ge L_0 >0$, where $L_0$ is given in the proof.
Suppose $\sigma n|\log\sigma|^{-1}\to\infty$
for $\sigma=\sigma_{n}\in(0,1)$ as $n\to\infty$. We have
\[
\liminf_{n\to\infty}\inf_{\hat{\theta}}\sup_{f\in\mathcal G^s(z_0,L)}r_{n}^{-1}\E_{f}\Big[\big(\hat{\theta}-f(z_0)\big)^{2}\Big]^{1/2}>0,
\]
where the infimum is taken over all estimators $\hat{\text{\ensuremath{\theta}}}$ built upon $\mathcal X$ and $\mathcal Y$, with 
$r_{n}$ given by (\ref{eq:rate}).
\end{thm}

Theorem \ref{thm:lowerBound} proves that $r_n$ is the minimax rate of convergence in squared pointwise error for nonparametric dispersal estimation. The main difficulty of the proof consists in establishing that the parent locations
indeed become uninformative if $n\sigma\to\infty$, to the effect that ignoring
the data $X_i$ is indeed the best we can do. Our argument is
based on the following heuristics: Given a number of parents $|\mathcal X| \sim\mathrm{Poiss}(\lambda n)$,
the distribution of a child trait $Y$, 
conditional on the parent traits satisfies: 
\begin{align*}
\P(Y \in \d y\,|X_{1},\dots,X_{|\mathcal X|}\big) & = \frac{1}{{|\mathcal X|}}\sum_{i=1}^{{|\mathcal X|}}f_{\sigma}(y-X_{i})\d y \\
&=\int_{0}^{1}f_{\sigma}(y-x)\d x\d y+\Big(\frac{1}{|\mathcal X|}\sum_{i=1}^{|\mathcal X|}f_{\sigma}(y-X_{j})-\E[f_{\sigma}(y-X_{1})]\Big)\d y,
\end{align*}
where, conditional on $|\mathcal X|$,
\[
\Var\Big(\frac{1}{|\mathcal X|}\sum_{i=1}^{|\mathcal X|}f_{\sigma}(y-X_{i})\Big||\mathcal X|\Big)=\frac{1}{|\mathcal X|}\Var(f_{\sigma}(y-X_{1}))\le\frac{1}{|\mathcal X|\sigma}\|f\|_{L^{2}}^{2}.
\]
Since $\E[|\mathcal X|]=\lambda n$, the influence of the parent traits becomes uninformative if $n \sigma\to\infty$. See Section \ref{sec:ProofsMainResults} for a rigorous proof.

\section{The case of an unknown scale parameter $\sigma$} \label{sec: unknown sigma}

In practice, it may well be the case that the scale $\sigma$ is unknown itself. We address this issue in a two-steps strategy: first, we study the estimation of $\sigma$ as a statistical problem in its own right. In particular, we need to distinguish from the data in which regime we stand ($n\sigma \rightarrow \infty$ versus $n\sigma$ bounded). Also, we need an accurate estimator $\widehat \sigma$ of $\sigma$ with respect to the relative error $\widehat \sigma/\sigma-1$ since $\sigma$ itself may vanish as $n \rightarrow \infty$, see Section \ref{sec: esti unknown sigma} below. Second, we use the estimator $\widehat \sigma$ and the associated decision rules to determine the underlying scale to cook-up a $\sigma$-adaptive and final estimator $\widetilde f(z_0)$ by plug-in that proves to be optimal in all regimes (in probability, for simplicity), see Section \ref{sec: esti f plugin} below.
\subsection{Estimation of $\sigma$} \label{sec: esti unknown sigma}
The first question to settle is to decide whether $n\sigma$ is sufficiently large to apply $n\sigma\to\infty$ asymptotics or not. To quantify $n\sigma$ empirically, we define
\[
  \hat T:=N(\R\setminus[0,1]).
\]
Since the support of the offspring point process $N$ is given by $[-\frac{\sigma}{2},1+\frac{\sigma}{2}]$ and the intensity of $N$ is of the order $n$, we expect that $\hat T$ is indeed of order $n\sigma$. 
\begin{lem}\label{lem:EstNSigma}
  Let $\supp f\subseteq[-\frac12,\frac12]$ and $I_f:=\int_\R |x|f(x)\d x$. For any $\sigma\in(0,1]$ and $n\in\N$ we have
  \begin{equation*}
    \E[\hat T]=n\sigma\lambda\mu I_f\qquad\text{and}\qquad \Var(\hat T)\le n\sigma\lambda (\mu+\mu^2).
  \end{equation*}
\end{lem}
In particular, Chebychev's inequality shows that for any $\kappa>0$, we have
\begin{equation}\label{eq:T}
  \P(\hat T\le \kappa)\to\begin{cases}
                         &0,\qquad \text{if } n\sigma\to\infty,\\
                         &1,\qquad \text{if } n\sigma\le \frac{\kappa}{2\lambda\mu I_f}.
                      \end{cases}
\end{equation}

We cannot use $\hat T$ to estimate $\sigma$ since the quantity $I_f$ is unknown, but we can further exploit the support $[-\frac{\sigma}{2},1+\frac{\sigma}{2}]$ of the offspring location traits $Y_{j}$. Namely, we can construct a boundary type estimator for $\sigma$. We focus on the left boundary,
but the method can be easily modified to the right boundary or a combination
of both. For $l\in\{1,\dots,|\mathcal{Y}|\}$ the (non-decreasing) order statistics
are denoted by $Y_{(l)}$. A naive estimator for $\sigma$ is 
$-2Y_{(1)}=-2\min_{j}Y_{j}$. We actually need to improve this estimator by taking
the parent location traits near the left boundary into account. The
resulting estimator is defined as
\begin{align*}
\hat{\sigma}^{(1)} & :=-2Y_{(1)}+2\bar{X}_{(\hat l)}\qquad\text{with}\qquad\bar{X}_{(l)}:=\frac{1}{l}\sum_{j=1}^{l}X_{(j)},\qquad \hat l:= \kappa_n\sqrt{\widehat T},
\end{align*}
for some sequence $\kappa_n\to\infty$.
%It is well-known that the convergence rates for boundary estimators
%depend on the behavior of the distribution near the boundary.
If $f$
is bounded away from zero on its support $[-\frac{1}{2},\frac{1}{2}]$,
the corresponding c.d.f. $F$ admits at least a linear growth at the boundary.
%Note that even if $f$ is bounded away from zero, the convolution
%$f\ast p$ with the parent trait distribution has a regular
%boundary. In particular, it cannot be bounded away from zero (except
%for the unnatural case where $f$ has a singularity at the boundary).

\begin{prop}
\label{prop:estSigma}Suppose the dispersal density $f$ satisfies
\begin{equation}\label{eq:boundaryCondition}
%F(z-\frac{1}{2})\ge(\gamma z)\vee0\qquad\text{for }z\le1
\inf_{z \in [-\frac{1}{2}, \frac{1}{2}]} f(z) \ge \gamma >0
\end{equation}
for some constant $\gamma>0$. If $\sigma n\to\infty$ and $\hat \sigma^{(1)}$ is specified with $\kappa_{n}\to\infty$ (that can be taken arbitrarily slowly diverging), then
\[
%\hat{\sigma}^{(1)}-\sigma=O_{\P}(\kappa_{n}\sqrt{\frac{\sigma}{n}})\qquad\text{and}\qquad
\frac{\hat{\sigma}^{(1)}}{\sigma}-1=O_{\P}\big(\frac{\kappa_{n}}{\sqrt{\sigma n}}\big).
%=o_{\P}(1).
\]
\end{prop}

Equivalently, we have $\hat\sigma^{(1)}-\sigma=O_{\P}(\kappa_n\sqrt{\sigma/n})$. In other words, for constant $\sigma$ we obtain the typical parametric rate, but the error bound is considerably improved if $\sigma\to0$. Most importantly,
the relative estimation error $\frac{\hat{\sigma}^{(1)}-\sigma}{\sigma}$
is small as soon as $\sigma n\to\infty$.\\

%For large scales, the estimator $\hat{\sigma}^{(1)}$ performs well, while 
In the regime where $\sigma n$ is small, we almost can guess the
relationship $Y_{j}=X_{i_{j}}+\sigma D_{j}$ between an offspring
trait $Y_{j}$ and its parent trait $X_{i_{j}}$ via a nearest neighbour approach. As a consequence, we can
estimate $\sigma$ by a local boundary estimation approach around
the distinct parent traits.  To use this local information, we  pick a kernel function $\psi^\dagger$ with the following properties: for some $C^\dagger >0$, we have
\begin{itemize}
\item[(i)] $\psi^\dagger$ is symmetric with $\mathrm{Supp}(\psi^\dagger) = [-1, 1]$ and $\psi^\dagger(x) = C^\dagger$ for $x \in [-\tfrac{1}{2}, \frac{1}{2}]$,
\item[(ii)] $\psi^\dagger$ is continuous, almost everywhere differentiable and $\tfrac{d}{dx}\psi^\dagger(x) < 0$ on $(\tfrac{1}{2}, 1]$,

\item[(iii)] For some constant $c>0$ and every $\varepsilon \in (0,\tfrac{1}{2})$, we have
$$\int_0^\varepsilon \big(\psi^\dagger(0)-\psi^\dagger(x+\tfrac{1}{2})\big) \d x \ge c\, \frac{\varepsilon}{\log 1/\varepsilon}.$$ 
%\item[(v)] {\color{red}$\int_{-1}^1 |\tfrac{d}{dx} \psi^\dagger(x)|\d x = 2C^\dagger$.}
\end{itemize}
%\[\psi^\dagger(x):=\frac1{C^\dagger}\Big(\eins_{[-1/2,1/2]}(x)-(2|x|-1)_+^\alpha\Big)_+,\qquad \alpha\in(0,1),
%\]
%with normalising constant $C^\dagger>0$ such that $\int_\R\psi^\dagger(x)dx=1$. 
The kernel $\psi^\dagger$ plays the role of a continuous proxy of the indicator function $\eins_{[-1/2,1/2]}$. We specifically may pick
$$\psi^\dagger(x) = C^\dagger \Big(1+\frac{\log 2}{\log(|x|-\tfrac{1}{2})_+}\Big)_+$$
where we set $\tfrac{1}{\log 0}:=0$, with $C^\dagger >0$ such that $\int\psi^\dagger(x) \d x=1$, but other choices (up to a modification of the constants in the estimates) are obviously possible.
%In particular, it is constant on $[-\frac{1}{2},\frac12]$, is supported on $[-1,1]$ and for small $\alpha>0$ it has similar sharp edges at $\pm\frac12$ as the indicator function.}{\color{red} [additional necessary properties for our proofs: decreasing on $[1/2,1]$, $\int|(\psi^\dagger)'(z)|dz<\infty$ and $\int_0^\epsilon(1-\psi^\dagger(z+\frac12))dx\ge \epsilon O(\epsilon)$ with $O(\epsilon)$ as large es possible (For the orignal indicator function the integral was of order $\epsilon$). ]
 Write
\[
\frac{1}{\mu\lambda n}\sum_{i,j}\psi^\dagger\big((Y_{j}-X_{i})/h\big)=:\E[\psi^\dagger(\sigma D_1/h)]+n\lambda h+\xi(h),
\]
where the stochastic noise term $\xi(h)$ defined via the last display satisfies
\begin{align*}
\E[\xi(h)] & =\E\Big[\frac{1}{\mu\lambda n}\sum_{i,j}\psi^\dagger\big((Y_{j}-X_{i})/h\big)\Big]-\E[\psi^\dagger(\sigma D_1/h)]-n\lambda h\\
 & =O(nh(\sigma+h)),
\end{align*}
as shown in Step~1 in the proof of the technical Lemma~\ref{lem:AuxXi} in Section \ref{sec:ProofsMainResults} below. By Proposition~\ref{prop:variance}
we moreover have for $h\le\sigma$:
\begin{align}
\Var\big(\xi(h)\big) & \lesssim\frac{1}{n}\big((n\sigma+n^{2}\sigma^{2})\frac{h^{2}}{\sigma^{2}}+(n\sigma+1)\frac{h}{\sigma}\big)\nonumber \\
 & \lesssim\frac{h^{2}}{\sigma}+nh^{2}+h+\frac{h}{n\sigma} 
\lesssim nh^{2}+\frac{1}{n}.\label{eq:xiVar}
\end{align}
(The last estimate is obvious for $\sigma \le n^{-1}$, whereas for $\sigma \ge n^{-1}$, we write $h=n^{-1/2} n^{1/2}h\le\frac12(n^{-1}+nh^2)$ and conclude with $\frac{h^2}{\sigma}\le nh^2$.)
% and $h=n^{-1/2} n^{1/2}h\le\frac12(n^{-1}+nh^2)$. What do we want to write?]}
Since $h\mapsto\E[\psi^\dagger(\sigma D_1/h)]$ is non-decreasing and equals $\psi^\dagger(0)$ as soon as $h$ reaches $\sigma$ under assumption~\ref{eq:boundaryCondition}, we define for some sequence $\kappa_{n}>0$:
\begin{align*}
\hat{\sigma}^{(2)}:= & \min\Big\{ h>0:\frac{1}{\mu\lambda n}\sum_{i,j}\psi^\dagger\big((Y_{j}-X_{i})/h\big)\ge n\lambda h+\psi^\dagger(0)-\sqrt{nh^{2}+n^{-1}}\kappa_{n}\Big\}.
\end{align*}

\begin{prop}
\label{prop:sigma2}Suppose $n\sigma^{3/2}\to 0$ and \eqref{eq:boundaryCondition}.
Then $\hat\sigma^{(2)}$ specified with $\kappa_n=\sqrt{\log n}$ satisfies
\[
\frac{\hat{\sigma}^{(2)}}{\sigma}-1=
O_{\P}\big((\log n)^2\sqrt{n\sigma^{2}+n^{-1}}\big).
\]
%for arbitrarily small $a>0$.}
\end{prop}
%{\color{red}[Hence, currently, our rate is an arbitrary small polynomial factor worse]}
Note that the condition $n\sigma^{3/2}\to0$ exactly characterises the regime where the rate of $\hat\sigma^{(2)}$ is faster than the rate of $\hat\sigma^{(1)}$. Combining the estimators $\hat\sigma^{(1)}$ and $\hat\sigma^{(2)}$  with the decision rule $\{\hat T>\kappa_n\}$ that enables us to decide whether we are in the regime $n\sigma \rightarrow \infty$ or not, we define our final estimator for $\sigma$ as:
\begin{equation} \label{eq: decision rule esti sigma}
\widehat \sigma :=\left\{
\begin{array}{ll}
\widehat \sigma^{(1)} &  \text{on} \qquad \{\widehat T > \kappa_n\} \cap \{\widehat \sigma^{(1)} > n^{-2/3}/\log n\} \\
\widehat \sigma^{(2)} & \text{otherwise.}
\end{array}
\right.
\end{equation}
Putting together Proposition \ref{prop:estSigma} and \ref{prop:sigma2}, we readily obtain the following rate for estimating $\sigma$ across scales:
%We conclude:
\begin{thm} \label{thm:sigma}
  Under the boundary condition \eqref{eq:boundaryCondition} the estimator $\hat\sigma$ defined above with $\kappa_n=\sqrt{\log n}$ satisfies
  \[
   \frac{\hat{\sigma}}{\sigma}-1=O_{\P}\Big((\log n)^2\big(\sqrt{n\sigma^{2}+n^{-1}}\wedge\frac{1}{\sqrt{n\sigma}}\big)\Big).
  \]
 \end{thm}

The performance of $\widehat \sigma$ in terms of its fluctuations in relative error are shown in Figure~\ref{fig: rate sigma}. They will be sufficient to implement an optimal (up to a logarithm) scale adaptive plug-in strategy for the estimation of $f$, as developed in the next section.

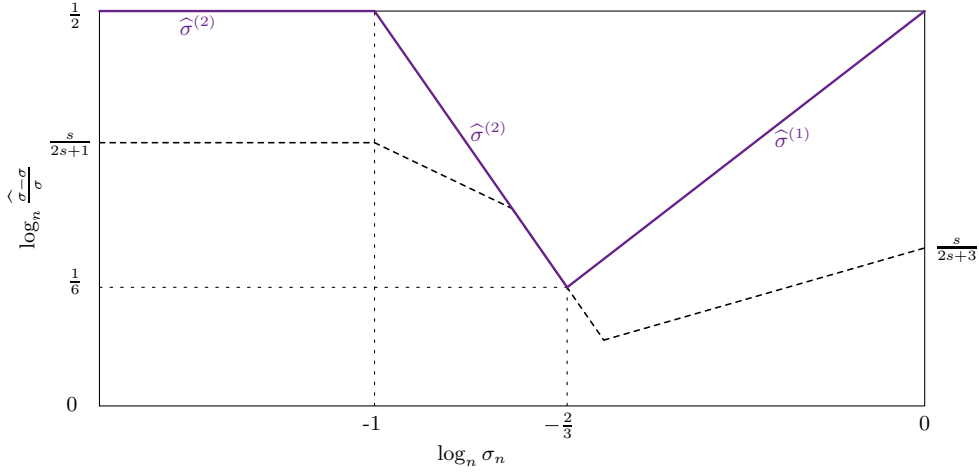
\begin{figure}
\centering{}\input{IMGrates}
\caption{The rates of convergence of $\widehat \sigma = \widehat \sigma^{(1)} $ or $\widehat \sigma^{(2)}$ depending on the decision rule \eqref{eq: decision rule esti sigma}, as a function the dispersal rate $\sigma$ on a log-log plot (in solid purple) together with the minimax rate $r_n$ for dispersal density estimation (in dashed black). The purple curve always dominates the black one.} \label{fig: rate sigma}
\end{figure}

\subsection{Estimation of $f$ when $\sigma$ is unknown} \label{sec: esti f plugin}

Recall that the optimal rate of convergence is achieved by the interaction estimator
or the deconvolution estimator depending whether
\[
\sigma\le n^{-(4s+3)/(6s+6)}\qquad\text{or}\qquad\sigma>n^{-(4s+3)/(6s+6)}
\]
respectively. Theorem~\ref{thm:sigma} implies that 
$
\hat{\sigma}n^{(4s+3)/(6s+6)}
%=\frac{\hat{\sigma}}{\sigma}\sigma n^{-(2s+1)/(2s+2)}
=(1+o_{\P}(1))\sigma n^{(4s+3)/(6s+6)}
$
in all regimes for $\sigma$. In turn, we can decide for the best estimator in a data-driven way by setting
$$ 
\tilde{f_n}(z_0) =
\left\{
\begin{array}{ll}
 \tilde{f}_{\hat \sigma}^{(1)}(z_0) &  \text{on} \qquad \big\{\hat{\sigma}n^{(4s+3)/(6s+6)} \ge 1 \big\} \\ \\
 \tilde{f}_{\hat \sigma}^{(2)}(z_0) & \text{otherwise.}
\end{array}
\right.
$$
where $\hat \sigma$ is specified in Theorem~\ref{thm:sigma}, and we use the plug-in counterparts to the deconvolution estimator from \eqref{eq:fHat1} and the interaction estimator from \eqref{eq:fHat2}, respectively, given by
%The next step is to construct estimators which neither depend on the
%unknown $\sigma$ nor on $\lambda$ and $\mu$. A plug-in approach yields for the deconvolution estimator from \eqref{eq:fHat1}
\begin{align*}
\tilde{f}_{\hat \sigma}^{(1)}(z_0) & :=  \frac{1}{\hat \sigma\hat h_1^2 |\mathcal Y|}\sum_{j}K'\Big(\frac{z_0}{\hat h_{1}}-\frac{Y_{j}}{\hat\sigma \hat h_{1}}\Big)\Big(\frac{1}{|\mathcal X|}\sum_iK\Big(\frac{\hat \sigma z_0}{9}-\frac{Y_{j}-X_i}{9}\Big)\Big),
\end{align*}
specified with $\hat h_1=(n\hat\sigma)^{-1/(2s+3)}$,
%For small $\sigma$ we construct the counterpart to the direct estimator from \eqref{eq:fHat2} via 
and
\begin{align*}
\tilde{f}_{\hat \sigma}^{(2)}(z_0) =\frac{1}{|\mathcal Y| \hat h_{2}}\sum_{i,j} 2K'\big(2(\hat\sigma z_0-Y_{j})\big) K\Big(\frac{z_0}{\hat h_{2}}-\frac{Y_{j}-X_{i}}{\hat \sigma \hat h_{2}}\Big)-\hat \sigma |\mathcal X|,
\end{align*}
specified with $\hat h_2= (n\wedge \hat\sigma^{-1})^{-1/(2s+1)}$.
%\[
%\hat{f}_{h_{1}^{*}}^{\mathrm{dec}}(z_{0})=\frac{1}{|\mathcal{Y}|}\sum_{j}\frac{1}{\sigma(h_{1}^{*})^{2}}K'\Big(\frac{z_{0}}{h_{1}^{*}}-\frac{Y_{j}}{\sigma h_{1}^{*}}\Big)\qquad\text{with}\qquad h_{1}^{*}=(n\sigma)^{-1/(2s+3)}
%\]
%which is
%\[
%\tilde{f}_{\hat{\sigma}}^{\mathrm{dec}}(z_{0})=\frac{1}{|\mathcal{Y}|}\sum_{j}\frac{1}{\hat{\sigma}\hat{h}_{1}^{2}}K'\Big(\frac{z_{0}}{\hat{h}_{1}}-\frac{Y_{j}}{\hat{\sigma}\hat{h}_{1}}\Big)\qquad\text{with}\qquad\hat{h}_{1}=(n\hat{\sigma})^{-1/(2s+3)}.
%\]

\begin{thm}\label{prop:plugIn} 
Let $s,L>0,z_0\in(-\frac12,\frac12)$ and suppose $K$ fulfills Assumption~\ref{ass: kernel} with order $\ell_K \ge \lfloor s \rfloor$. The following holds uniformly for $f\in\mathcal G^s(z_0,L)$ with property \eqref{eq:boundaryCondition}:
\begin{enumerate}
 \item[(i)] If $n\sigma/(\log n)^2\to\infty$, we have 
\[
\tilde{f}_{\hat \sigma}^{(1)}(z_{0})-f(z_{0})=O_{\P}(u_{n}), \quad\text{with}\;\; u_{n}=(\sigma n)^{-s/(2s+3)}.
\]
 \item[(ii)] If $n\sigma^{2}\to0$ and  $s>3/2$, we have 
\[
\tilde{f}_{\hat \sigma}^{(2)}(z_{0})-f(z_{0})=O_{\P}((\log n)^2v_{n}), \quad \text{with}\;\;v_{n}=(n\wedge \sigma^{-1})^{-s/(2s+1)}+\sqrt n\sigma.
\]
\end{enumerate}
In particular, we achieve $\sigma$-adaptation in the following sense:
$$\tilde{f_n}(z_0)- f(z_0) = O_{\P}((\log n)^2r_n),$$
where $r_n$ is the minimax rate for the estimation of $f(z_0)$ in squared error loss, given in \eqref{eq:rate}, according to Theorems \ref{thm: rates main} and \ref{thm:lowerBound}.
\end{thm}

Our final Theorem \ref{prop:plugIn} shows that under our set of assumptions, it is possible to achieve optimality for the pointwise estimation of the dispersal density $f(z_0)$ across scales without any prior knowledge of the scale $\sigma$. The proof is based upon the study of the smoothness of the interaction and convolution estimators as random processes indexed by $\sigma$ together with sharp estimation rates for  $\hat{\sigma}$ provided by Theorem \ref{thm:sigma}. For technical reason, we have the additional restriction $s >3/2$ for the smoothness of $f$ locally around $z_0$ and our bounds are in probability and not expectation.

\section{Discussion} \label{sec: discussion}

\subsection{One-to-one correspondence between parents and children} \label{sec: exactly one}
In a concrete situation like {\it e.g.} Example~1 in Section \ref{sec: applications},  it is desirable to impose a one-to-one correspondence
between parents and their children. Each parent has exactly one
child and in particular $|\mathcal X| = |\mathcal Y|$. The point process $N$ of the offspring generation
should be modified as follows. For $M=\sum_{i}\delta_{X_{i}}$ as
parent generating process, the offspring of a specific parent trait $X_{i}$ is given by
\begin{equation}
Y_{j}=X_{j}+\sigma D_{j}\label{eq:1to1}
\end{equation}
for independent random variables $(D_{j})_{j \ge 1}$ distributed according to the
dispersal density $f$ and with the scaling parameter $\sigma\in(0,1]$.
The offspring point process is then simply given by 
\[
N(\d y)=\sum_{j}\delta_{Y_{j}}(\d y) = \sum_{j}\delta_{X_{j}+\sigma D_{j}}(\d y).
\]
Let us compare this one-to-one model with the original model of Section \ref{sec: dispersal inference setting}: Here, we have an urn model without replacement and a fixed number $|\mathcal X|$ of draws while in Section~\ref{sec: dispersal inference setting} we have an urn model with replacement and random number of draws. In this
modified case we can proceed analogously to \citet[Proposition~1]{goldenshluger2016}
and we obtain the following counterpart to Proposition~\ref{prop:expFormula}
with $\mu=1$:
\begin{prop}
\label{prop:expFormula-1}Let $(A_{i})_{1 \le i \le I}$ and $(B_{j})_{1 \le j \le J}$
be two families of disjoint Borel subsets of $[0,1]$ and $\R$, respectively.
Then for any $(\eta_{1},\dots,\eta_{I})\in\R^{I}$ and $(\xi_{1},\dots,\xi_{J})\in\R^{J}$
we have
\begin{align*}
 & \log\E\Big[\exp\Big(\sum_{i=1}^{I}\eta_{i}M(A_{i})+\sum_{j=1}^{J}\xi_{j}N(B_{j})\Big)\Big]\\
 & \quad=n\lambda\sum_{i=1}^{I}(e^{\eta_{i}}-1)|A_{i}|+n\lambda\sum_{j=1}^{J}(e^{\xi_{k}}-1)Q_{\sigma}([0,1],B_{j})+n\lambda\sum_{i=1}^{I}\sum_{j=1}^{J}(e^{\eta_{i}}-1)(e^{\xi_{j}}-1)Q_{\sigma}(A_{i},B_{j}),
\end{align*}
where $|A|$ denotes the Lebesgue measure of $A$ and 
$
Q_{\sigma}(A,B)=\int_{A}\int_{B}f_{\sigma}(y-x)\,\d y\,\d x.
$
\end{prop}

Note that this modified exponential formula coincides with the result
of Proposition~\ref{prop:expFormula} if we apply a first order approximation
of the exponential functions on the right-hand side of (\ref{eq:campbell})
and set $\mu=1$. As a consequence, differentiation yields the same form of the
intensity measure 
\[
\frac{1}{n\lambda}\E\Big[\frac{M(\d x)N(\d y)}{\d x \d y}\Big]= n\lambda(f_{\sigma}\ast\eins_{[-1,1]})(y)+f_{\sigma}(y-x)
\]
as in Proposition \ref{prop:expFormula} while second order properties differ. In fact, higher
order integrals in the one-to-one setting are a bit simpler, see Remark~\ref{rem:OntoOne} below.
We can thus apply exactly the same estimator as before and Theorem~\ref{thm: rates main}
remains true in the one-to-one setting.
\begin{cor}
Let $z_0 \in (-1/2, 1/2)$, and $s,L>0$. Based on the observations $M=\sum_{i}\delta_{X_{i}}$
and $N=\sum_{j}\delta_{Y_{j}} = \sum_{j}\delta_{X_{j}+\sigma D_j}$ with $D_{j}$ from \eqref{eq:1to1}
there is an estimator $\widehat{f}(z_0)$ depending on $\lambda, \sigma$ and
$s$ such that
\[
\sup_{f\in\mathcal G^s(z_0,L)}\E\Big[\big(\widehat{f}(z_0)-f(z_0)\big)^{2}\Big]^{1/2} \lesssim r_{n},
\]
where the rate of convergence is given by (\ref{eq:rate}).
\end{cor}

\subsection{The multidimensional case with arbitrary parent distributions} \label{sec: other parent trait distrib}
We briefly investigate two essential extensions of our approach and the results of Theorem \ref{thm: rates main}:
\begin{enumerate}
\item[\bf 1)] How will the rate change in the deconvolution regime if we consider
multidimensional observations, {\it i.e.} when $\mathcal X \subset \mathcal O \subset \R^{d}$ and $\mathcal Y \subset \R^{d}$ with $d> 1?$
\item[\bf 2)] How does the parent distribution affect the problem when
$p$ is not uniform over $\mathcal O$?
\end{enumerate}

For the first question, we argue that  Theorem \ref{thm: rates main} and \ref{thm:lowerBound} generalise to a parent point process in $\R^{d}$ with intensity
measure $\lambda\eins_{\mathcal O}$ for a rectangular set $\mathcal O\subset\R^{d}$.
In general, the smoothing properties of a convolution with $p=|\mathcal O|^{-1}\eins_{\mathcal O}$
for a bounded set $\mathcal O\subset\R^{d}$ considerably depends on the
geometry of $\mathcal O$ and its boundary $\partial \mathcal O$ in particular, see {\it e.g.} \citet{Randol1969}. More specifically,
a more regular boundary results in a faster decay of the characteristic
function of the uniform distribution on $\mathcal O$. As a consequence, the statistical
deconvolution problems depends on the geometry, too.
%, which is an interesting
%research problem on its own. 
To investigate the impact of the regularity properties of the parent distribution $p$,
we assume in this section that the characteristic function
of the parent distribution is bounded away from zero. In this case
the classical spectral approach is applicable and allows for a transparent
analysis of statistical estimation, even in the multidimensional case $d>1$.\\ 

Let $M=\sum_{i}\delta_{X_{i}}(\d x)$ be Poisson point process 
with intensity $\lambda n p(x)\d x$ on $\R^{d}$, where 
$p\colon\R^{d}\to [0,\infty)$ is a bounded probability density function. As before the point process $N$ on $\R^d$ of offspring traits has conditional intensity $\mu (M\ast f_\sigma(y))\d y$ with $f_\sigma$ from \eqref{eq: def informal dispersal}. The decomposition \eqref{eq:exp} generalises to
\begin{equation}
\frac{1}{n\lambda \mu}\E\Big[\int_{\mathcal O \times \R^d}\psi_1(y)\psi_2\big((y-x)/\sigma\big)M(\d x)N(\d y)\Big]
  =\sigma^d n\lambda\, \bar{\mathcal U}_\sigma(f \ast p)+ \bar{\mathcal V}_\sigma(f),\label{eq:DecompGen}
\end{equation}
with
$$\bar{\mathcal U}_\sigma(f \ast p) = \int_{\R^d}\psi_1(y)\big(\psi_2\ast p(\sigma\cdot)\big)(y/\sigma)(f_{\sigma}\ast p)(y)\d y$$
and
$$\bar{\mathcal V}_\sigma(f)=  \int_{\R^d}\big(\psi_1\ast p(-\cdot)\big)(\sigma z)\psi_2(z)f(z)\d z.$$

To deconvolve $f_\sigma\ast p$ in $\bar{\mathcal U}_\sigma(f\ast p)$, we denote the characteristic function of $p$ by $\phi_{p}(u)=\F[p](u)=\int_{\R^d} e^{iu^\top x}p(x)\d x$, $u\in\R^d$, and assume that $\phi_p(u)\neq 0$ for all $u\in\R^d$. Then we can choose the spectral deconvolution kernel
\[
  \psi_1=\F^{-1}\Big[\frac{\F K(h_1u)}{\phi_p(u/\sigma)}\Big](z_0-\cdot/\sigma),\qquad z_0\in\R^d,
\]
with inverse Fourier transform $\F^{-1}[h(u)](x)=(2\pi)^{-d}\int_{\R^d}e^{-iu^\top x}h(u)\d u$ for any $h\in L^1(\R^d)$ and where $K\colon\R^d\to \R$ is a band limited kernel with bandwidth $h_1>0$. Plancherel's identity and $\F[(f\ast p)(\sigma\cdot)](u)=\sigma^{-d}\F[f](u)\phi_p(u/\sigma)$ indeed yields
\begin{align}
  \int_{\R^d}\psi_1(z)(f_\sigma\ast p)(z)\d z&=\frac{\sigma^d}{(2\pi)^d}\int_{\R^d}e^{-iu^\top z}\frac{\F K(h_1u)}{\phi_p(u/\sigma)}\F\big[(f\ast p(\sigma\cdot))\big](u)\d u\notag\\
  &=\frac{1}{(2\pi)^d}\int_{\R^d}e^{-iu^\top z}\F[K_{h_1}](u)\F[f](u)\d u
  =(K_{h_1}\ast f)(z).\label{eq:MeanSD}
\end{align}
Using $\psi_1$ from above and $\psi_2=1$, we define the following \emph{spectral deconvolution estimator} on $\R^d$:
$$\widehat f_{h_1}^{\tt sd}(z_0) = \frac{1}{n\lambda\mu}\sum_j\F^{-1}\Big[\frac{\F K(h_1u)}{\phi_p(u/\sigma)}\Big]\Big(z_0-\frac{Y_j}{\sigma}\Big).$$

If the parent distribution is unknown, then we can profit from the observations $\mathcal X$ by replacing $\phi_{p}$ with its empirical counterpart $\hat{\phi}_{p}(u)=|\mathcal X|^{-1}\sum_{j}e^{iX_{j}^{\top}u}$
as demonstrated in the classical (univariate) deconvolution literature,
see {\it e.g.} \citet{Neumann2007,ComteLacour2011,DattnerEtAl2016}. 

The rate of convergence will be determined by the decay of $\phi_{p}(u)$.
Note that $\phi_{p}(u)$  should decay at least as $|u|^{-d}$ in order to allow for
a bounded density $h=\F^{-1}\phi_{p}$. In the multivariate case, the extensions $\mathcal H_d^{s}(z_0)$ and $\mathcal G_d^{s}(z_0,L)$ of the local H\"older classes  $\mathcal H^{s}(z_0)$ and the class of H\"older regular, bounded densities $\mathcal G^s(z_0,L)$, respectively, from the univariate case to $\R^d$ is straightfroward by considering partial derivatives. A kernel $K$ of order $\ell_K\ge\lfloor s\rfloor$ in dimension $d$ can be constructed, for instance by tensorisation of the one dimensional case. We keep-up with the notation $|\cdot|$ to denote the Euclidean norm on $\R^d$.
\begin{thm} \label{thm: d>1 case}
Let $z\in\R^{d}$ and $f\in \mathcal G_{d}^{s}(z_0,L)$ for some $s>0$ and let $p$ be a bounded probability density on $\R^d$ with $\phi_p(u)\neq 0$ for all $u\in\R^d$. If $K$
is a kernel or order $\ell_K\ge\lfloor s\rfloor$ that satisfies $\supp\F K\subset\{u\in\R^{d}:|u|\le1\}$, then we have 
\[
{\sup_{f\in \mathcal G_{d}^{s}(z_0,L)}}\E\Big[\big|\widehat{f}_{h_1}^{\rm sp}(z_0)-f(z_0)\big|^{2}\Big]^{1/2}\lesssim h^{s}+\frac{\sigma^{d}}{n^{1/2}}\Big(\int_{\{u\in\R^{d}:|u|\le1/(\sigma h)\}}|\phi_{p}(u)|^{-2}du\Big)^{1/2}.
\]
In the mildly ill-posed case with $|\phi_{p}(u)|\gtrsim(1+|u|^{2})^{-t/2}$
for some $t \ge d$, we obtain
\[
{\sup_{f\in \mathcal G_{d}^{s}(z_0,L)}}\E\Big[\big|\hat{f}_{h_1}^{\rm sp}(z_0)-f(z_0)\big|^{2}\Big]\lesssim(n\sigma^{2t-d})^{-2s/(2s+2t+d)}.
\]
 for the choice $h_1=(n\sigma^{2t-d})^{1/(2s+2t+d)}$.
 
In the severely ill-posed case $|\phi_{p}(u)|\gtrsim e^{-\gamma|u|^{\beta}}$ 
for some $\gamma,\beta>0$, we obtain
\[
{\sup_{f\in \mathcal G_{d}^{s}(z_0,L)}}\E\Big[\big|\hat{f}_{h_1}^{(d)}(z_0)-f(z)\big|^{2}\Big]^{1/2}\lesssim\sigma^{-s}(\log n)^{-s/\beta}
\]
for the choice $h_1=\sigma^{-1}(\frac{1}{4\gamma}\log n)^{-1/\beta}$.
\end{thm}
Several remarks are in order: {\bf 1)}
for $d=1$, the uniform distribution corresponds to the degree of ill-posedness
$t=1$ for which we indeed recover the rate $(n\sigma)^{2s/(2s+3)}$. {\bf 2)}
For more regular distributions
with $t>1$ the dependence of the deconvolution rate on the scaling
parameter $\sigma$ is even more severe. For $t=\frac32d$ the deconvolution
estimator is only consistent if $n\sigma^{2d}\to\infty.$ Since the analysis of the variance of an interaction approach in the general setting reveals a term of order $n\sigma^{2d}$, {\it cf.} Remark~\ref{rem:VarianceGeneral}, we conjecture that there is a regime where $f$ cannot be estimated consistently if $t>\frac32d$. 

To discuss the behaviour of an interaction estimator similar to \eqref{eq: esti int}, we first note that our variance estimates in Section~\ref{sec:Bias-and-variance}
can be generalised to arbitrary parent distributions
with bounded densities and to higher dimensions, see in particular Remark~\ref{rem:VarianceGeneral} at the end of Section \ref{sec:Bias-and-variance}. A soon as the bias due to $\bar{\mathcal U_\sigma(f\ast p)}$ in the interaction regime can be controlled, one can in principle build an estimator $\widehat f_h^{\tt{int}}(z_0)$ with mean squared-error of order
\[
\E\Big[\big(\widehat f_h^{\tt{int}}(z_0)-f(z_0)\big)^{2}\Big]^{1/2}\lesssim h^{s}+\max\Big(\frac{1}{n^{1/2}h^{d/2}},\frac{\sigma^{d/2}}{h^{d/2}}, n^{1/2}\sigma^{d}\Big)
\]
for $f\in \mathcal G_{d}^{s}(z_0,L)$. An optimised choice of $h=(n\wedge\sigma^{-d})^{-1/(2s+d)}$ then yields
$$\E\Big[\big(\widehat f_h^{\tt{int}}(z_0)-f(z_0)\big)^{2}\Big]^{1/2}\lesssim \max\big((n\wedge\sigma^{-d})^{-s/(2s+d)}, n^{1/2}\sigma^{d}\big).$$
However, the analysis of the bias due to $\bar{\mathcal U_\sigma(f\ast p)}$ is quite delicate and we do not have a clear understanding of its behaviour at the moment.  Note also that the analysis of the interaction estimator is applicable
to a generating parent trait point process with intensity $\lambda n\eins_{\mathcal O}(x)\d x$ for any Borel
set $\mathcal O \subset\R^{d}$ without additional difficulties.

\section{A numerical example} \label{sec: numerical example}
In order to illustrate the main results, we will apply the estimators from Definition~\ref{def: esti final} together with the pure deconvolution estimator and the interaction estimator  from \eqref{eq: def H}  and \eqref{eq: def K}, respectively, on simulated observations. 

We choose $n=1000$, $\lambda=\mu=1$ and consider the $\mathrm{Beta}(2,3)$-distribution (shifted by $-1/2$) for the dispersal, {\it i.e.} 
\[ f(z)= \frac{1}{12}\Big(\frac12+z\Big)\Big(\frac12-z\Big)^2\eins_{[-1/2,1/2]}(z),\qquad z\in\R.\]
For the estimators we pick the kernel
\[
%  K(z):=\begin{cases}
%           \frac{15}{23},\qquad & |z|\le 1/2,\\
%           \frac{15\cdot 16}{23}z^2(1-|z|)^2,\qquad& |z|\in(1/2,1]
%        \end{cases}
    K(z):=\begin{cases}
           1,\qquad & |z|\le \frac14,\\
           \big(\big(\frac{32}{15}(|z|-\frac14)\big)^2-1\big)^2,\qquad& |z|\in(\frac14,\frac{23}{32}],\\
           0,&\text{otherwise},
        \end{cases}
\]
which is continuously differentiable, non-negative and satisfies Assumption~\ref{ass: kernel} with order $\ell_K=1$. The bandwidths are chosen as $h_1=0.7(n\sigma)^{-1/7}$ and $h_2=0.7\min(n,\sigma^{-1})^{-1/5}$ according to Proposition~\ref{thm:rates}. In the numerical experiments we note a considerable sensitivity of the small scale estimator $\hat f_{h_2}^{(2)}$ to the choice of $h_1$. While the estimator achieves the optimal rate with $h_1=1/(2\sigma)$, the proofs reveal that the conditions $h_1+h_2<\sigma^{-1}$ and $h_1\ge 4$ are sufficient for the bias analysis and the variance grows by the factor $(\sigma h_1)^{-1}$. Hence, $h_1$ should be as large as possible and we choose $h_1=\max(4,\sigma^{-1}-1.1\,h_2)$ for $\hat f_{h_2}^{(2)}$.

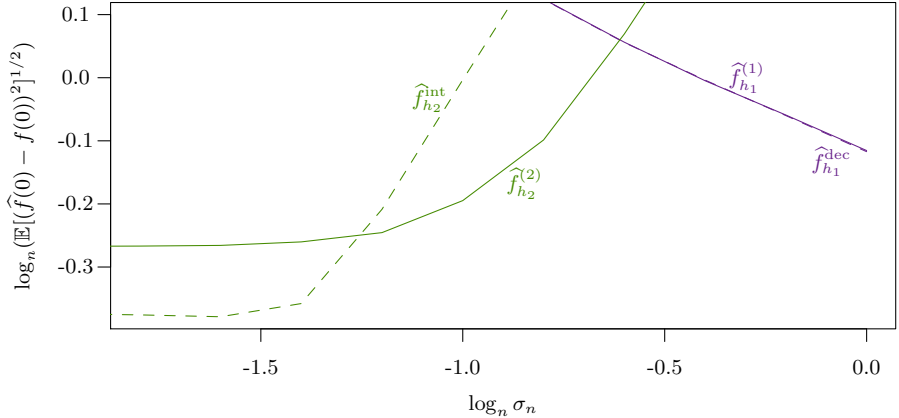
\begin{figure}\centering
  \input{IMGmc}
  \caption{Logarithmic plot of the root mean squared error of $\hat f_{h_1}^{(1)}$ (\emph{purple}), $\hat f_{h_2}^{(2)}$ (\emph{green}) as well as $\hat f_{h_1}^{\rm dec}$ (\emph{purple, dashed}), $\hat f_{h_2}^{\rm int}$ (\emph{green, dashed}) at $z_0=0$ depending on $\sigma_n$ based on a Monte Carlo simulation.}\label{fig:MC}
\end{figure}

A Monte Carlo simulation confirms our theoretical findings. Figure~\ref{fig:MC} shows the root mean squared error at point $z_0=0$ based on a Monte Carlo simulation with 500 repeated samples. In each of these iterations the same random variables are drawn to define the point clouds $\mathcal X$ and $\mathcal Y$ along a grid of scaling parameters $\sigma_n=n^{\tau}$, $\tau\in\{-2,-1.8,\dots,-0.2,0\}$. We see that $\hat f_{h_1}^{(1)}$ is much better for large scales, but its error increases as $\sigma$ decreases. For $\sigma< n^{-0.6}\approx 0.016$ the direct estimator $\hat f_{h_2}^{(2)}$ is better and its error improves when $\sigma$ decreases further. As we can see in this figure $\hat f_{h_1}^{(1)}$ and $\hat f^{\rm dec}_{h_1}$ behave similarly. In contrast, there is a notable difference between $\hat f_{h_2}^{(2)}$ and $\hat f_{h_2}^{\rm int}$. More precisely, our numerical experiments indicate that the stochastic error of $\hat f_{h_2}^{\rm int}$ is smaller across all scales, but even before $\sigma=n^{-1}$ the bias effect drastically kicks in.

\begin{figure}
  \input{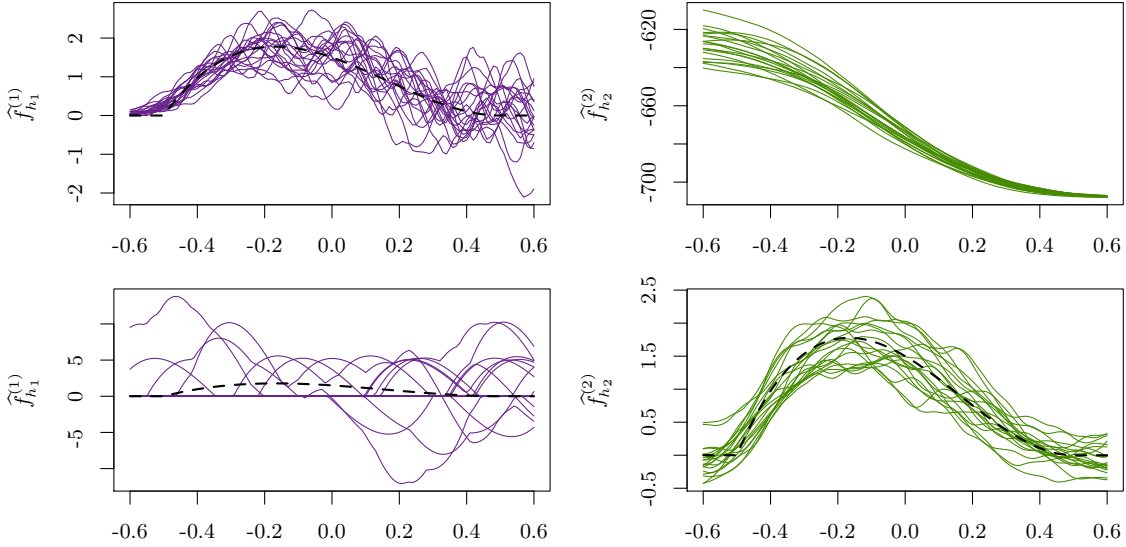}
  \caption{20 realisations of $\hat f_{h_1}^{(1)}$ (\emph{left}) and  $\hat f_{h_2}^{(2)}$ (\emph{right}) and true density (\emph{dashed}) in the point process model with parameters $n=1000$, $\lambda=\mu=1$ and scaling parameters $\sigma=n^{-1/20}\approx0.71$ (\emph{top}) and $\sigma=n^{-19/20}\approx0.0014$ (\emph{bottom}).}\label{fig:Estimates}
\end{figure}

For a more detailed impression on the behaviour of our main estimators Figure~\ref{fig:Estimates} shows 20 realisations of the estimators $\hat f_{h_1}^{(1)}$ and $\hat f_{h_2}^{(2)}$ for two different choices the scaling parameter $\sigma\in\{n^{-1/20},n^{-19/20}\}$. While the deconvolution estimator $\hat f_{h_1}^{(1)}$ fluctuates around the true density for the rather large scale $\sigma=n^{-1/20}\approx0.71$, its result is completely useless at the small scale $\sigma=n^{-19/20}\approx0.0014$. The direct estimator $\hat f_{h_2}^{(2)}$ reveals the opposite behaviour. At the small scale $\hat f_{h_2}^{(2)}$ concentrates around $f$, but its shape is unrelated to the dispersal density if $\sigma$ is large.

\section{Proofs\label{sec:Proofs}} 

\subsection{The covariance structure of $(M,N)$} \label{sec: proof campbell}

\begin{proof}[Proof of Proposition~\ref{prop:expFormula}]
Conditional on $M$ we obtain from Campbell's formula
\begin{align*}
 & \E\Big[\exp\Big(\sum_{j=1}^{J}\eta_{j}M(A_{j})+\sum_{k=1}^{K}\xi_{k}N(B_{k})\Big)\Big|(X_{j})\Big]\\
 & \quad=\exp\Big(\sum_{j=1}^{J}\eta_{j}M(A_{j})\Big)\E\Big[\exp\Big(\sum_{k=1}^{K}\xi_{k}N(B_{k})\Big)\Big|(X_{j})\Big]\\
 & \quad=\exp\Big(\sum_{j=1}^{J}\eta_{j}M(A_{j})\Big)\exp\Big(\int(e^{\sum_{k}\xi_{k}\eins_{B_{k}}(y)}-1)\mu\sum_{l}f_{\sigma}(y-X_{l})\d y\Big)\\
 & \quad=\exp\Big(\sum_{j=1}^{J}\eta_{j}M(A_{j})+\sum_{k=1}^{K}\int_{B_{k}}(e^{\xi_{k}}-1)\mu\sum_{l}f_{\sigma}(y-X_{l})\d y\Big)
 =\exp\Big(\int g\d M\Big)
\end{align*}
for 
\[
g(x):=\sum_{j=1}^{J}\eta_{j}\eins_{A_{j}}(x)+h(x),\qquad h(x):=\mu\sum_{k=1}^{K}(e^{\xi_{k}}-1)\int_{B_{k}}f_{\sigma}(y-x)\d y.
\]
Applying again Campbell's formula yields
\[
\log\E\Big[\exp\Big(\sum_{j=1}^{J}\eta_{j}M(A_{j})+\sum_{k=1}^{K}\xi_{k}N(B_{k})\Big)\Big]=n\lambda\int_{0}^{1}\big(e^{g(x)}-1\big)\d x.
\]
Pluging-in $g(x)$, we obtain
\begin{align*}
&n\lambda\int_{0}^{1}\big(e^{g(x)}-1\big)\d x  =n\lambda\sum_{j=1}^{J}\int_{A_{j}}\big(e^{\eta_{j}+h(x)}-1\big)\d x+n\lambda\int_{(\bigcup_{j}A_{j})^{c}}\big(e^{h(x)}-1)\d x\\
 &\qquad =n\lambda\sum_{j=1}^{J}(e^{\eta_{j}}-1)\int_{A_{j}}(e^{h(x)}-1)\d x+n\lambda\sum_{j=1}^{J}(e^{\eta_{j}}-1)|A_{j}|
 +n\lambda\int_{0}^{1}\big(e^{h(x)}-1)\d x.\qedhere
\end{align*}
%We conclude by differentiating the previous exponential formula.
\end{proof}

\begin{proof}[Proof of Corollary \ref{cor: intensity}]
It suffices to note that
\begin{align*}
\Psi(\eta,\xi):=&\E\Big[e^{\eta M(A)+\xi N(B)}\Big]\\
=&\exp\Big(n\lambda(e^{\eta}-1)|A|+n\lambda\int_{0}^{1}\big(e^{\psi_{B}(\xi,x)}-1)\d x+n\lambda(e^{\eta}-1)\int_{A}(e^{\psi_{B}(\xi,x)}-1)\d x\Big),
\end{align*}
where 
$
\psi_{B}(\xi,x):=\mu(e^{\xi}-1)\int_{B}f_{\sigma}(y-x)\d y
$
satisfies
\[
\partial_{\xi}\Psi(\eta,\xi)=\Psi(\eta,\xi)\Big(n\lambda\int_{0}^{1}e^{\psi_{B}(\xi,x)}\partial_{\xi}\psi_{B}(\xi,x)\d x+n\lambda(e^{\eta}-1)\int_{A}e^{\psi_{B}(\xi,x)}\partial_{\xi}\psi_{B}(\xi,x)\d x\Big)
\]
and
\begin{align*}
\partial_{\eta}\partial_{\xi}\Psi(\eta,\xi) & =\Psi(\eta,\xi)\Big(n\lambda e^{\eta}|A|+n\lambda e^{\eta}\int_{A}(e^{\psi_{B}(\xi,x)}-1)\d x\Big)\\
 & \qquad\times\Big(n\lambda\int_{0}^{1}e^{\psi_{B}(\xi,x)}\partial_{\xi}\psi_{B}(\xi,x)\d x+n\lambda(e^{\eta}-1)\int_{A}e^{\psi_{B}(\xi,x)}\partial_{\xi}\psi_{B}(\xi,x)\d x\Big)\\
 & \quad+\Psi(\eta,\xi)n\lambda e^{\eta}\int_{A}e^{\psi_{B}(\xi,x)}\partial_{\xi}\psi_{B}(\xi,x)\d x.
\end{align*}
Hence, due to $\Psi(0,0)=1$, $\psi_{B}(0,x)=0$ and $\partial_{\xi}\psi_{B}(0,x)=\mu\int_{B}f_{\sigma}(y-x)\d y$,
the claimed formula is given by  $\partial_\eta\partial_\xi\Psi(0,0)$.\end{proof}

The previous proof also shows that
\begin{equation}\label{eq:ExpN}
  \E[N(B)] =n\lambda\mu\int_{0}^{1}\int_{B}f_{\sigma}(y-x)\,\d y\d x,\qquad B\subset \R,
\end{equation}
by calculating $\partial_\xi\Psi(0,0)$. While Corollary~\ref{cor: intensity} determines the mean of linear functionals of $M$ and $N$ the following lemma investigates the covariance structure. 

\begin{lem}
\label{lem:higherOrder}For 
\[
Q_{\sigma}(A,B):=\int_{A}\int_{B}f_{\sigma}(y-x)\,\d y\,\d x,\; Q_{\sigma}^{2}(A,B_{1},B_{2}):=\int_{A}\int_{B_{1}}\int_{B_{2}}f_{\sigma}(y_{1}-x)f_{\sigma}(y_{2}-x)\,\d y_{2}\,\d y_{1}\,\d x
\]
 we have:
\begin{enumerate}
\item If $B_1,B_2\subset\R$ are intervals such that $B_1\cap B_2=\emptyset$, then
\begin{align}
  \E[N(B_1)N(B_2)]&=n^{2}\lambda^{2}Q_\sigma([0,1],B_1)Q_\sigma([0,1],B_2)
  +n\lambda\mu^2 Q_\sigma^2([0,1],B_1,B_2)\notag\\
  &=\E[N(B_1)]\E[N(B_2)]+n\lambda\mu^2 Q_\sigma^2([0,1],B_1,B_2)\label{eq:BB}.
\end{align}
\item If $A_{1},A_{2}\subset[0,1]$ and $B\subset\R$ are intervals such
that $A_{1}\cap A_{2}=\emptyset$, then
\begin{align}
\E[M(A_{1})M(A_{2})N(B)] & =n^{3}\lambda^{3}\mu|A_{1}||A_{2}|Q_{\sigma}([0,1],B)+n^{2}\lambda^{2}\mu|A_{2}|Q_{\sigma}(A_{1},B)\notag\\
&\qquad+n^{2}\lambda^{2}\mu|A_{1}|Q_{\sigma}(A_{2},B)\label{eq:AAB}
\end{align}
\item If $A\subset[0,1]$ and $B_{1},B_{2}\subset\R$ are intervals such
that $B_{1}\cap B_{2}=\emptyset$, then
\begin{align}
&\E[M(A)N(B_{1})N(B_{2})] \label{eq:ABB}\\
&=n^{3}\lambda^{3}\mu^{2}|A|Q_{\sigma}([0,1],B_{1})Q_{\sigma}([0,1],B_{2})\nonumber\\
 & \quad+n^{2}\lambda^{2}\mu^{2}Q_{\sigma}(A,B_{1})Q_{\sigma}([0,1],B_{2})+n^{2}\lambda^{2}\mu^{2}Q_{\sigma}([0,1],B_{1})Q_{\sigma}(A,B_{2})\nonumber \\
 & \quad+n^{2}\lambda^{2}\mu^{2}|A_{1}|Q_{\sigma}^{2}([0,1],B_{1},B_{2})+n\lambda\mu^{2}Q_{\sigma}^{2}(A,B_{1},B_{2}).\nonumber 
\end{align}
\item For $A_{1},A_{2}\subset[0,1]$ and $B_{1},B_{2}\subset\R$ with $A_{1}\cap A_{2}=\emptyset$
and $B_{1}\cap B_{2}=\emptyset$ we have
\begin{align*}
 & \E[M(A_{1})M(A_{2})N(B_{1})N(B_{2})]\\
 & =\E[M(A_{1})N(B_{1})]\E[M(A_{2})N(B_{2})]\\
 & \quad+n^{3}\lambda^{3}\mu^{2}\Big(|A_{1}|Q_{\sigma}(A_{2},B_{1})Q_{\sigma}([0,1],B_{2})+|A_{2}|Q_{\sigma}([0,1],B_{1})Q_{\sigma}(A_{1},B_{2})\\
 &\qquad\qquad\qquad+|A_{1}||A_{2}|Q_{\sigma}^{2}([0,1],B_{1},B_{2})\Big)\\
 & \quad+n^{3}\lambda^{2}\mu^{2}\Big(Q_{\sigma}(A_{2},B_{1})Q_{\sigma}(A_{1},B_{2})+|A_{1}|Q_{\sigma}^{2}(A_{2},B_{1},B_{2})+|A_{2}|Q_{\sigma}^{2}(A_{1},B_{1},B_{2})\Big).
\end{align*}
\end{enumerate}
\end{lem}
The proof is similiar to the proof of Corollary~\ref{cor: intensity}, but using fourth order partial derivatives \eqref{eq:campbell}. We postpone the details to the appendix.
\begin{rem}
\label{rem:OntoOne}A modified proof shows that in the one-to-one
case from Section~\ref{sec: exactly one} the formulas in Lemma~\ref{lem:higherOrder} remain true if we set $Q_{\sigma}^{2}(A,B_1,B_2)=0$ for all Borel sets $A,B_1,B_2\subset\R$.
\end{rem}

\subsection{Bias and variance bounds\label{sec:Bias-and-variance}}

The two integrals $\mathcal U_\sigma(f\ast p)$ and $\mathcal V_\sigma(f)$
%\begin{align*}
%  \mathcal U_\sigma(f\ast p)&=\sigma^{-1}\int_{\R}\psi_1(y)(\psi_2(\cdot/\sigma)\ast\eins_{[0,1]})(y)(f_{\sigma}\ast\eins_{[0,1]})(y)\d y\qquad\text{and}\\
%  \mathcal V_\sigma(f)&=\int_\R(\psi_1\ast\eins_{[-1,0]})(\sigma z)\psi_2(z)f(z)\d z
%\end{align*}
from \eqref{eq:exp} are analysed in the following two propositions.
\begin{prop}
\label{prop:bias1} For $s,L>0$ let $f\in\mathcal G^s(z_0,L)$. Suppose that $\psi_1$ and $\psi_2$ are given by \eqref{eq: def H} and \eqref{eq: def K}, respectively, with $K$ satisfying Assumption~\ref{ass: kernel} with $\ell_K \ge \lfloor s \rfloor$.
\begin{enumerate}
\item If $\sigma h_{2}\ge8$
for $\sigma\in(0,1]$, then we have for all $z_0 \in(-1/2,1/2)$, and
$h_{1}\in(0,\frac{h_{2}}{8}]$
\[
\mathcal U_\sigma(f\ast p)=\frac{1}{\sigma h_{2}}f(z_0)+O\big(\frac{h_{1}^{s}}{\sigma h_{2}}\big).
\]
\item If $h_{1}\in [4,\sigma^{-1})$
for $\sigma\in(0,1/4)$, then we have for all $z_0\in(-1/2,1/2)$, and
$h_{2}\in(0,\min\{1,\sigma^{-1}-h_1\}]$
\[
\mathcal U_\sigma(f\ast p)=\frac{1}{h_{1}}.
\]
\end{enumerate}
\end{prop}

\begin{proof}
\emph{(i)} Noting that $(f_{\sigma}\ast\eins_{[0,1]})(y)=F\big(\frac{y}{\sigma}\big)-F\big(\frac{y-1}{\sigma}\big)$
for the cumulative distribution function $F$ of $f$, we plug-in
the choice of $\psi_{1}$ and substitute $z=z_0-\frac{x}{\sigma}$ to obtain
\begin{align}
 \mathcal U_\sigma(f\ast p)
 & =\frac{1}{\sigma h_{1}^{2}}\int_{\R}(\psi_{2}\ast\eins_{[0,1/\sigma]})(x/\sigma)K'\big(\frac{z_0}{h_{1}}-\frac{x}{\sigma h_{1}}\big)\big(F\big(\frac{x}{\sigma}\big)-F\big(\frac{x-1}{\sigma}\big)\big)\d x\nonumber \\
 & =\frac{1}{h_{1}^{2}}\int_{\R}(\psi_2\ast\eins_{[0,1/\sigma]})\big(z_0-z\big)K'\big(\frac{z}{h_{1}}\big)\big(F\big(z_0-z\big)-F\big(z_0-z-\frac{1}{\sigma}\big)\big)\d z.\label{eq:firstbias}
\end{align}
Moreover, we have
\begin{align*}
(\psi_2\ast\eins_{[0,\frac1\sigma]})\big(z_0-z\big) & =\frac{1}{h_{2}}\int_{\R}K\big(\frac{z_0-x}{h_{2}}\big)\eins_{[0,\frac1\sigma]}\big(z_0-z-x\big)\d x
  =\frac{1}{h_{2}}\int_{\R}K\big(\frac{x}{h_{2}}\big)\eins_{[0,\frac1\sigma]}\big(x-z\big)\d x.
\end{align*}
Denoting the anti-derivative of $K$ by $K^{(-1)}(z):=\int_{-\infty}^{z}K(x)\d x$,
we obtain
\begin{equation}
(\psi_2\ast\eins_{[0,1/\sigma]})\big(z_0-z\big)=K^{(-1)}\big(\frac{z}{h_{2}}+\frac{1}{\sigma h_{2}}\big)-K^{(-1)}\big(\frac{z}{h_{2}}\big).\label{eq:K1conv}
\end{equation}
On the assumptions that $K(x)=1$ for $|x|\le\frac{1}{4}$,
$h_{1}\le\frac{h_{2}}{8}$ and $\sigma h_{2}\ge8$, we have $\frac{z+1/\sigma}{h_{2}},\frac{z}{h_{2}}\in[-\frac{1}{4},\frac{1}{4}]$
for any $|z|\le h_{1}$ and thus
$
(\psi_2\ast\eins_{[0,1/\sigma]})\big(z_0-z\big)=\frac{1}{\sigma h_{2}}.
$
Since the boundary terms vanish by the compact support of $K$,
we conclude from (\ref{eq:firstbias}) together with integration by
parts that
\begin{align*}
\mathcal U_\sigma(f\ast p) & =\frac{1}{\sigma h_{2}h_{1}^{2}}\int_{\R}K'\big(\frac{z}{h_{1}}\big)\big(F\big(z_0-z\big)-F\big(z_0-z-\sigma^{-1}\big)\big)\d z\\
 & =\frac{1}{\sigma h_{1}h_{2}}\int_{\R}K\big(\frac{z}{h_{1}}\big)\big(f\big(z_0-z\big)-f\big(z_0-z-\sigma^{-1}\big)\big)\d z\\
 & =\frac{1}{\sigma h_{2}}\Big(\big(\frac{1}{h_{1}}K\big(\frac{\cdot}{h_{1}}\big)\ast f\big)(z_0)-\big(\frac{1}{h_{1}}K\big(\frac{\cdot}{h_{1}}\big)\ast f\big)\big(z_0-\sigma^{-1}\big)\Big)\\
 & =\frac{1}{\sigma h_{2}}\big(f(z_0)+f(z_0-\sigma^{-1})\big)+O\big(\frac{h_{1}^{s}}{\sigma h_{2}}\big),
\end{align*}
where the last bound is due to the usual bias estimate based on the local
H\"older regularity of $f$. Note that $f(z_0-\sigma^{-1})=0$ since
$\sigma\le1$ and $|z_0|<1/2$. {Especially, $f\in \mathcal H^s(z_0-\sigma)$ such that the bias estimate applies for convolution at $z_0-\sigma^{-1}$, too}.

\emph{(ii)} If $h_{1}+h_{2}<\frac{1}{\sigma}$, then $\frac{z}{h_{2}}+\frac{1}{\sigma h_{2}}>1$
for any $|z|\le h_{1}$ and thus (\ref{eq:K1conv}) reads as
\[
(\psi_2\ast\eins_{[0,1/\sigma]})\big(z_0-z\big)=1-K^{(-1)}\big(\frac{z}{h_{2}}\big).
\]
If $\sigma<1/2$, we moreover have $F(z_0-z-\sigma^{-1})=0$ for
all $z_0\in(-1/2,1/2)$ and {any} $z\in[-h_{1},h_{1}]\subset[-(2\sigma)^{-1},(2\sigma)^{-1}]$.
Using that $K(x)=1$ for $|x|\le1/4$ and using $F(x)=0$
for $x<-1/2$ and $F(x)=1$ for $x\ge1/2$, we obtain from (\ref{eq:firstbias})
for $h_{1}\ge4$
\begin{align*}
\mathcal U_\sigma(f\ast p) & =\frac{1}{h_{1}^{2}}\int_{|z|>h_{1}/4}\big(1-K^{(-1)}(z/h_{2})\big)K'\big(\frac{z}{h_{1}}\big)F\big(z_0-z\big)\d z\\
 & =\frac{1}{h_{1}^{2}}\int_{-\infty}^{-h_{1}/4}\big(1-K^{(-1)}(z/h_{2})\big)K'\big(\frac{z}{h_{1}}\big)\d z\\
 & =\frac{1}{h_{1}^{2}}\int_{-\infty}^{0}\big(1-K^{(-1)}(z/h_{2})\big)K'\big(\frac{z}{h_{1}}\big)\d z\\
 & =\frac{1}{h_{1}}\Big(\big(1-K^{(-1)}(0)\big)K\big(0\big)+\frac{1}{h_{2}}\int_{-\infty}^{0}K\big(\frac{z}{h_{2}}\big)K\big(\frac{z}{h_{1}}\big)\d z\Big)\\
 & =\frac{1}{2h_{1}}+\frac{1}{h_{1}h_{2}}\int_{-\infty}^{0}K\big(\frac{z}{h_{1}}\big)K\big(\frac{z}{h_{2}}\big)\d z
\end{align*}
where we have used integration by parts and symmetry of $K$.
Since $K(\cdot/h_{1})$ is constant one on $[-\frac{h_{1}}{4},\frac{h_{1}}{4}]$,
the last line simplifies for $h_{2}{\le1}\le\frac{h_{1}}{4}$ to
\[
\mathcal U_\sigma(f\ast p)=\frac{1}{2h_{1}}+\frac{1}{h_{1}h_{2}}\int_{-\infty}^{0}K\big(\frac{z}{h_{2}}\big)\d z=\frac{1}{h_{1}}.\qedhere
\]
\end{proof}
\begin{prop}
\label{prop:bias2} For $s,L>0$ let $f\in\mathcal G^s(z_0,L)$. Suppose that $\psi_1$ and $\psi_2$ are given by \eqref{eq: def H} and \eqref{eq: def K}, respectively, with $K$ satisfying Assumption~\ref{ass: kernel} with $\ell_K \ge \lfloor s \rfloor$. Let $\sigma\in(0,1]$.
\begin{enumerate}
\item If $h_{2}\le\frac{h_{1}}{4}$ and $h_{1}+h_{2}<1/\sigma$, then
\[
\mathcal V_\sigma(f)=h_{1}^{-1}f(z_0)+O\big(h_{1}^{-1}h_{2}^{s}\big).
\]
\item We have for all $\sigma,h_{1},h_{2}>0$ 
\[
\big|\mathcal V_\sigma(f)\big|\le h_{1}^{-1}\|K\|_{L^{1}}\|K'\|_{L^{1}}\|f\|_{\infty}.
\]
\end{enumerate}
\end{prop}

\begin{proof}
(i) We use
\begin{align*}
(\psi_1\ast\eins_{[-1,0]})(\sigma z) & =\frac{1}{\sigma h_{1}^{2}}\int_{-1}^{0}K'\big(\frac{z_0-z}{h_{1}}+\frac{t}{\sigma h_{1}}\big)\d t\\
 & =\frac{1}{h_{1}}\int_{-1/(\sigma h_{1})}^{0}K'\big(\frac{z_0-z}{h_{1}}+t\big)\d t=\frac{1}{h_{1}}\Big(K\big(\frac{z_0-z}{h_{1}}\big)-K\big(\frac{z_0-z}{h_{1}}-\frac{1}{\sigma h_{1}}\big)\Big).
\end{align*}
Noting that $z\in[z_0-h_{2},z_0+h_{2}]$ by the support of $\psi_2$ and
using $h_{1}+h_{2}<1/\sigma$, we have $|z_0-z-\frac{1}{\sigma}|\ge\frac{1}{\sigma}-h_{2}>h_{1}$
and thus $K\big(\frac{z_0-z}{h_{1}}-\frac{1}{\sigma h_{1}}\big)=0$.
Since $K\big(\frac{z_0-z}{h_{1}}\big)=1$ for $|\frac{z_0-z}{h_{1}}|\le\frac{h_{2}}{h_{1}}\le\frac{1}{4}$,
we obtain
\begin{align}
\int(\psi_1\ast\eins_{[-1,0]})(\sigma z)\psi_2(z)f(z)\d z & =\frac{1}{h_{1}h_{2}}\int_{\R}K\big(\frac{z_0-z}{h_{1}}\big)K\big(\frac{z_0-z}{h_{2}}\big)f(z)\d z.\nonumber \\
 & =\frac{1}{h_{1}h_{2}}\int_{\R}K\big(\frac{z_0-z}{h_{2}}\big)f(z)\d z.\label{eq:secondBias}
\end{align}
Applying again the usual bias estimates on $(h_{2}^{-1}K(\cdot/h_{2})\ast f)(z_0)$,
we conclude
\[
\int(\psi_1\ast\eins_{[-1,0]})(\sigma z)\psi_2(z)f(z)\d z=h_{1}^{-1}f(z_0)+O(h_{1}^{-1}h_{2}^{s}).
\]

(ii) The second bound easily follows from Young's inequality:
\begin{align*}
\Big|\int(\psi_1\ast\eins_{[-1,0]})(\sigma z)\psi_2(z)f(z)\d z\Big| & \le\|\psi_1\ast\eins_{[-1,0]}\|_{\infty}\|\psi_2\|_{L^{1}}\|f\|_{\infty}\\
 & \le\|\psi_1\|_{L^{1}}\|\psi_2\|_{L^{1}}\|f\|_{\infty}\le h_{1}^{-1}\|K\|_{L^{1}}\|K'\|_{L^{1}}\|f\|_{\infty}.\qedhere
\end{align*}
\end{proof}
The next step is to investigate the variance based on Lemma~\ref{lem:higherOrder}.
\begin{prop}
\label{prop:variance}If $f$ is bounded and $\phi^{\star}(x,y):=\psi_1(y)\psi_2\big(\frac{y-x}{\sigma}\big)$
for some kernels $\psi_1\in L^{1}\cap L^{2},\psi_2\in L^{2}$, then
there is some $C>0$ such that
\begin{align*}
\Var\Big(\sum_{j,k}\phi^{\star}(X_{j},Y_{k})\Big)\le C & n\lambda\mu(1\vee\|f\|_{\infty})\\
&\times\Big((\mu+1)\big(n\lambda\sigma+n^{2}\lambda^{2}\sigma^{2}\big)\|\psi_2\|_{L^{1}}^{2}+(n\lambda\sigma+\mu+1)\|\psi_2\|_{L^{2}}^{2}\Big)\|\psi_1\|_{L^{2}}^{2}.
\end{align*}
\end{prop}

\begin{proof}
We decompose
\begin{align*}
\Var\Big(\sum_{j,k}\phi^{\star}(X_{j},Y_{k})\Big) & =\E\Big[\Big(\sum_{j,k}\phi^{\star}(X_{j},Y_{k})\Big)^{2}\Big]-\E\Big[\sum_{j,k}\phi^{\star}(X_{j},Y_{k})\Big]^{2}\\
 & =\E\Big[\sum_{j,k}\phi^{\star}(X_{j},Y_{k})^{2}\Big]+\E\Big[\sum_{j_{1}\neq j_{2},k}\phi^{\star}(X_{j_{1}},Y_{k})\phi^{\star}(X_{j_{2}},Y_{k})\Big]\\
 & \quad+\E\Big[\sum_{j,k_{1}\neq k_{2}}\phi^{\star}(X_{j},Y_{k_{1}})\phi^{\star}(X_{j},Y_{k_{2}})\Big]\\
 & \quad+\Big(\E\Big[\sum_{j_{1}\neq j_{2},k_{1}\neq k_{2}}\phi^{\star}(X_{j_{1}},Y_{k_{1}})\phi^{\star}(X_{j_{2}},Y_{k_{2}})\Big]-\E\Big[\sum_{j,k}\phi^{\star}(X_{j},Y_{k})\Big]^{2}\Big)\\
 & =:J_{1}+J_{2}+J_{3}+J_{4}.
\end{align*}
Due to (\ref{eq:exp}) and Young's inequality we have
\begin{align*}
J_{1} & =n^{2}\lambda^{2}\mu\int_{0}^{1}\int_{\R}\psi_2^{2}\big(\frac{y-x}{\sigma}\big)\psi_1^{2}(y)(f_{\sigma}\ast\eins_{[0,1]})(y)\d y\d x+n\lambda\mu\int_{0}^{1}\int_{\R}\psi_2^{2}\big(z\big)\psi_1^{2}(x+\sigma z)f(z)\d z\d x\\
 & =n^{2}\lambda^{2}\mu\sigma\int_{\R}(\psi_2^{2}\ast\eins_{[0,1/\sigma]})(y/\sigma)\psi_1^{2}(y)(f_{\sigma}\ast\eins_{[0,1]})(y)\d y\\
 & \qquad+n\lambda\mu\int(\psi_1^{2}\ast\eins_{[-1,0]})(\sigma z)\psi_2^{2}(z)f(z)\d z.\\
 & \le n^{2}\lambda^{2}\mu\sigma\|\psi_2^{2}\ast\eins_{[0,1/\sigma]}\|_{\infty}\|\psi_1\|_{L^{2}}^{2}\|f_{\sigma}\ast\eins_{[0,1]}\|_{\infty}+n\lambda\mu\|\psi_1^{2}\ast\eins_{[-1,0]}\|_{\infty}\|\psi_2\|_{L^{2}}^{2}\|f\|_{\infty}\\
 & \le\mu\big(n^{2}\lambda^{2}\sigma+n\lambda\|f\|_{\infty}\big)\|\psi_1\|_{L^{2}}^{2}\|\psi_2\|_{L^{2}}^{2}.
\end{align*}
From Lemma~\ref{lem:higherOrder}(ii) and (iii), we deduce for $x_{1}\neq x_{2}$
in $[0,1]$ and $y\in\R$
\[
\E\big[\d M(x_{1})\d M(x_{2})\d N(y)\big]=n^{2}\lambda^{2}\mu\Big(n\lambda(f_{\sigma}\ast\eins_{[0,1]})(y)+f_{\sigma}(y-x_{1})+f_{\sigma}(y-x_{2})\Big)\d y\,\d x_{1}\d x_{2}
\]
as well as for $x\in[0,1]$ and $y_{1}\neq y_{2}$
\begin{align*}
\E\big[\d M(x)\d N(y_{1})\d N(y_{2})\big] & =n^{2}\lambda^{2}\mu^{2}\Big(n\lambda(f_{\sigma}\ast\eins_{[0,1]})(y_{1})(f_{\sigma}\ast\eins_{[0,1]})(y_{2})\\
 & \qquad+f_{\sigma}(y_{1}-x)(f_{\sigma}\ast\eins_{[0,1]})(y_{2})+f_{\sigma}(y_{2}-x)(f_{\sigma}\ast\eins_{[0,1]})(y_{1})\\
 & \qquad+\int_{0}^{1}f_{\sigma}(y_{1}-t)f_{\sigma}(y_{2}-t)\d t+\frac1{n\lambda} f_{\sigma}(y_{1}-x)f_{\sigma}(y_{2}-x)\Big)\d y_{1}\d y_{2}\d x.
\end{align*}
Therefore,
\begin{align*}
J_{2} & =\E\Big[\int_{0}^{1}\int_{0}^{1}\int_{\R}\phi^{\star}(x_{1},y)\phi^{\star}(x_{2},y)\eins_{\{x_{1}\neq x_{2}\}}M(\d x_{1})M(\d x_{2})N(\d y)\Big]=\mu(J_{2,1}+2J_{2,2})
\end{align*}
with
\begin{align*}
J_{2,1} & :=n^{3}\lambda^{3}\int_{0}^{1}\int_{0}^{1}\int_{\R}\psi_2\big(\frac{y-x_{1}}{\sigma}\big)\psi_2\big(\frac{y-x_{2}}{\sigma}\big)\psi_1^{2}(y)(f_{\sigma}\ast\eins_{[0,1]})(y)\d y\,\d x_{1}\d x_{2}\\
 & =n^{3}\lambda^{3}\sigma^{2}\int_{\R}(\psi_2\ast\eins_{[0,1/\sigma]})^{2}(y/\sigma)\psi_1^{2}(y)(f_{\sigma}\ast\eins_{[0,1]})(y)\d y\\
 & \le n^{3}\lambda^{3}\sigma^{2}\|\psi_1\|_{L^{2}}^{2}\|\psi_2\|_{L^{1}}^{2}\qquad\text{and}\\
J_{2,2} & :=n^{2}\lambda^{2}\int_{0}^{1}\int_{0}^{1}\int_{\R}\psi_2\big(\frac{y-x_{1}}{\sigma}\big)\psi_2\big(\frac{y-x_{2}}{\sigma}\big)\psi_1^{2}(y)f_{\sigma}(y-x_{1})\d y\,\d x_{1}\d x_{2}\\
 %& =n^{2}\lambda^{2}\int_{0}^{1}\int_{0}^{1}\int_{\R}\psi_2\big(z\big)\psi_2\big(z+\frac{x_{1}-x_{2}}{\sigma}\big)\psi_1^{2}(x_{1}+\sigma z)f(z)\d z\,\d x_{1}\d x_{2}\\
 & =n^{2}\lambda^{2}\sigma\int_{0}^{1}\int_{\R}(\psi_2\ast\eins_{[0,1/\sigma]})\big(z+\frac{x_{1}}{\sigma}\big)\psi_1^{2}(x_{1}+\sigma z)\psi_2\big(z\big)f(z)\d z\,\d x_{1}\\
 & \le n^{2}\lambda^{2}\sigma\|\psi_2\|_{L^{1}}\|\psi_1\|_{L^{2}}^{2}\int_{\R}|\psi_2\big(z\big)f(z)|\d z
  =n^{2}\lambda^{2}\sigma\|\psi_1\|_{L^{2}}^{2}\|\psi_2\|_{L^{1}}^{2}\|f\|_{\infty}.
\end{align*}
For the third term we have
\[
J_{3}=\E\Big[\int_{0}^{1}\int_{\R}\int_{\R}\phi^{\star}(x,y_{1})\phi^{\star}(x,y_{2})\eins_{\{y_{1}\neq y_{2}\}}M(\d x)N(\d y_{1})N(\d y_{2})\Big]=\mu^{2}(J_{3,1}+2J_{3,2}+J_{3,3}+J_{3,4}),
\]
where
\begin{align*}
J_{3,1} & :=n^{3}\lambda^{3}\int_{[0,1]\times\R^2}\psi_2\big(\frac{y_{1}-x}{\sigma}\big)\psi_2\big(\frac{y_{2}-x}{\sigma}\big)\psi_1(y_{1})\psi_1(y_{2})(f_{\sigma}\ast\eins_{[0,1]})(y_{1})(f_{\sigma}\ast\eins_{[0,1]})(y_{2})\d y_{1}\d y_{2}\d x\\
 & =n^{3}\lambda^{3}\int_{0}^{1}\big(\psi_2(-\cdot/\sigma)\ast\big(\psi_1(f_{\sigma}\ast\eins_{[0,1]})\big)\big)(x)^{2}\d x\\
 & \le n^{3}\lambda^{3}\|\psi_2(\cdot/\sigma)\|_{L^{1}}^{2}\|\psi_1\|_{L^{2}}^{2}
  =n^{3}\lambda^{3}\sigma^{2}\|\psi_1\|_{L^{2}}^{2}\|\psi_2\|_{L^{1}}^{2},\\
J_{3,2} & :=n^{2}\lambda^{2}\int_{[0,1]\times\R^2}\psi_2\big(\frac{y_{1}-x}{\sigma}\big)\psi_2\big(\frac{y_{2}-x}{\sigma}\big)\psi_1(y_{1})\psi_1(y_{2})f_{\sigma}(y_{1}-x)(f_{\sigma}\ast\eins_{[0,1]})(y_{2})\d y_{1}\d y_{2}\d x\\
 & =n^{2}\lambda^{2}\int_{0}^{1}\big(\big(\psi_2(-\cdot/\sigma)f_{\sigma}(-\cdot)\big)\ast \psi_1\big)(x)\big(\psi_2(-\cdot/\sigma)\ast\big(\psi_1(f_{\sigma}\ast\eins_{[0,1]})\big)\big)(x)\d x\\
 & \le n^{2}\lambda^{2}\|\psi_2(-\cdot/\sigma)f_{\sigma}\|_{L^{1}}\|\psi_1\|_{L^{2}}^{2}\|\psi_2(-\cdot/\sigma)\|_{L^{1}}\\
 & =n^{2}\lambda^{2}\sigma\|\psi_2(-\cdot)f\|_{L^{1}}\|\psi_2\|_{L^{1}}\|\psi_1\|_{L^{2}}^{2}
  \le n^{2}\lambda^{2}\sigma\|f\|_{\infty}\|\psi_1\|_{L^{2}}^{2}\|\psi_2\|_{L^{1}}^{2},\\
J_{3,3} & :=n^{2}\lambda^{2}\int_{[0,1]^2\times\R^2}\psi_2\big(\frac{y_{1}-x}{\sigma}\big)\psi_2\big(\frac{y_{2}-x}{\sigma}\big)\psi_1(y_{1})\psi_1(y_{2})f_{\sigma}(y_{1}-t)f_{\sigma}(y_{2}-t)\d y_{1}\d y_{2}\,\d t\d x\\
 & =n^{2}\lambda^{2}\int_{[0,1]^2\times\R^2}\psi_2\big(z_{1}+\frac{t-x}{\sigma}\big)\psi_2\big(z_{2}+\frac{t-x}{\sigma}\big)\psi_1(\sigma z_{1}+t)\psi_1(\sigma z_{2}+t)f(z_{1})f(z_{2})\d z_{1}\d z_{2}\,\d t\d x\\
 & \le n^{2}\lambda^{2}\sigma\int_{\R^{4}}\big|\psi_2\big(z_{1}+x\big)\psi_2\big(z_{2}+x\big)\psi_1(\sigma z_{1}+t)\psi_1(\sigma z_{2}+t)\big|f(z_{1})f(z_{2})\d z_{1}\d z_{2}\,\d x\d t\\
 & =n^{2}\lambda^{2}\sigma\int_{\R^{2}}(|\psi_2|\ast|\psi_2|)(z_{1}-z_{2})(|\psi_1|\ast|\psi_1|)\big(\sigma(z_{1}-z_{2})\big)f(z_{1})f(z_{2})\d z_{1}\d z_{2}\\
 & \le n^{2}\lambda^{2}\sigma\|\psi_1\|_{L^{2}}^{2}\int_{\R^{2}}(|\psi_2|\ast|\psi_2|)(z_{1})f(z_{1}+z_{2})f(z_{2})\d z_{1}\d z_{2}
  \le n^{2}\lambda^{2}\sigma\|\psi_1\|_{L^{2}}^{2}\|\psi_2\|_{L^{1}}^{2}\|f\|_{\infty},\\
J_{3,4} & :=n\lambda\int_{[0,1]\times\R^2}\psi_2\big(\frac{y_{1}-x}{\sigma}\big)\psi_2\big(\frac{y_{2}-x}{\sigma}\big)\psi_1(y_{1})\psi_1(y_{2})f_{\sigma}(y_{1}-x)f_{\sigma}(y_{2}-x)\d y_{1}\d y_{2}\,\d x\\
 & =n\lambda\int_{0}^{1}\big((f_{\sigma}(-\cdot)\psi_2(-\cdot/\sigma))\ast \psi_1\big)(x)^{2}\d x\\
 & \le n\lambda\|\psi_1\|_{L^{2}}^{2}\|f_{\sigma}\psi_2(\cdot/\sigma)\|_{L^{1}}^{2}\\
 & %\le n\lambda\|\psi_1\|_{L^{2}}^{2}\|f\|_{\infty}\|f\|_{L^{1}}\|\psi_2\|_{L^{2}}^{2}
 \le n\lambda\|f\|_{\infty}\|\psi_1\|_{L^{2}}^{2}\|\psi_2\|_{L^{2}}^{2}.
\end{align*}
Finally, we have due to Lemma~\ref{lem:higherOrder}(iv) for $x_{1}\neq x_{2}$
and $y_{1}\neq y_{2}$ that
\begin{align*}
 & \E\big[\d M(x_{1})\d M(x_{2})\d N(y_{1})\d N(y_{2})\big]-\E\big[\d M(x_{1})\d N(y_{1})\big]\E\big[\d M(x_{2})\d N(y_{2})\big]\\
 & \qquad=n^{2}\lambda^{2}\mu^{2}\Big(n\lambda f_{\sigma}(y_{1}-x_{2})(f_{\sigma}\ast\eins_{[0,1]})(y_{2})+n\lambda f_{\sigma}(y_{2}-x_{1})(f_{\sigma}\ast\eins_{[0,1]})(y_{1})\\
 & \qquad\qquad+n\lambda\int_{0}^{1}f_{\sigma}(y_{1}-t)f_{\sigma}(y_{2}-t)\d t+f_{\sigma}(y_{1}-x_{2})f_{\sigma}(y_{2}-x_{1})\\
 & \qquad\qquad+f_{\sigma}(y_{1}-x_{1})f_{\sigma}(y_{2}-x_{1})+f_{\sigma}(y_{1}-x_{2})f_{\sigma}(y_{2}-x_{2})\Big)\d y_{1}\d y_{2}\d x_{1}\d x_{2}
\end{align*}
and thus $J_{4}=\mu^{2}(2J_{4,1}+J_{4,2}+J_{4,3}+2J_{4,4})$ with
\begin{align*}
J_{4,1} & :=n^{3}\lambda^{3}\int_{[0,1]^2\times\R^2}\psi_2\big(\frac{y_{1}-x_{1}}{\sigma}\big)\psi_2\big(\frac{y_{2}-x_{2}}{\sigma}\big)\psi_1(y_{1})\psi_1(y_{2})\\
 & \qquad\qquad\times f_{\sigma}(y_{1}-x_{2})(f_{\sigma}\ast\eins_{[0,1]})(y_{2})\d y_{1}\d y_{2}\d x_{1}\d x_{2}\\
 & =n^{3}\lambda^{3}\int_{[0,1]\times\R}\big(\psi_2(\frac{\cdot}\sigma)\ast\eins_{[0,1]}\big)(y_{1})\big(\psi_2(-\frac{\cdot}\sigma)\ast\big((f_{\sigma}\ast\eins_{[0,1]})\psi_1\big)\big)(x_{2})\psi_1(y_{1})f_{\sigma}(y_{1}-x_{2})\d y_{1}\d x_{2}\\
 & =n^{3}\lambda^{3}\int_{0}^{1}\big(f_{\sigma}\ast\big((\psi_2(\cdot/\sigma)\ast\eins_{[0,1]})\psi_1\big)\big)(x_{2})\big(\psi_2(\cdot/\sigma)\ast\big((f_{\sigma}\ast\eins_{[0,1]})\psi_1\big)\big)(x_{2})\d x_{2}\\
 & \le n^{3}\lambda^{3}\|\psi_1\|_{L^{2}}^{2}\|\psi_2(\cdot/\sigma)\|_{L^{1}}^{2}\\
 & =n^{3}\lambda^{3}\sigma^{2}\|\psi_1\|_{L^{2}}^{2}\|\psi_2\|_{L^{1}}^{2},\\
J_{4,2} & :=n^{3}\lambda^{3}\int_{[0,1]^3\times\R^2}\psi_2\big(\frac{y_{1}-x_{1}}{\sigma}\big)\psi_2\big(\frac{y_{2}-x_{2}}{\sigma}\big)\psi_1(y_{1})\psi_1(y_{2}) f_{\sigma}(y_{1}-t)f_{\sigma}(y_{2}-t)\d y_{1}\d y_{2}\d t\d x_{1}\d x_{2}\\
% & =n^{3}\lambda^{3}\int_{[0,1]^3\times\R^2}\psi_2\big(z_{1}+\frac{t-x_{1}}{\sigma}\big)\psi_2\big(z_{2}+\frac{t-x_{2}}{\sigma}\big)\psi_1(\sigma z_{1}+t)\psi_1(\sigma z_{2}+t)\\
% & \qquad\qquad\times f(z_{1})f(z_{2})\d z_{1}\d z_{2}\d t\d x_{1}\d x_{2}\\
 & =n^{3}\lambda^{3}\sigma^{2}\int_{[0,1]\times\R^2}(\psi_2\ast\eins_{[0,1/\sigma]})\big(z_{1}+t/\sigma\big)(\psi_2\ast\eins_{[0,1/\sigma]})\big(z_{2}+t/\sigma\big)\psi_1(\sigma z_{1}+t)\psi_1(\sigma z_{2}+t)\\
 & \qquad\qquad\times f(z_{1})f(z_{2})\d z_{1}\d z_{2}\d t\\
 & =n^{3}\lambda^{3}\sigma^{2}\int_{0}^{1}\big(\big((\psi_2\ast\eins_{[0,1/\sigma]})\psi_1(\sigma\cdot)\big)\ast f(-\cdot)\big)(t/\sigma)^{2}\d t\\
 & \le n^{3}\lambda^{3}\sigma^{3}\|\big((\psi_2\ast\eins_{[0,1/\sigma]})\psi_1(\sigma\cdot)\big)\|_{L^{2}}^{2}\\
 & \le n^{3}\lambda^{3}\sigma^{2}\|\psi_1\|_{L^{2}}^{2}\|\psi_2\|_{L^{1}}^{2},\\
J_{4,3}&:=  n^{2}\lambda^{2}\int_{[0,1]^2\times\R^2}\psi_2\big(\frac{y_{1}-x_{1}}{\sigma}\big)\psi_2\big(\frac{y_{2}-x_{2}}{\sigma}\big)\psi_1(y_{1})\psi_1(y_{2})\\
 & \hspace{3cm}\times f_{\sigma}(y_{1}-x_{2})f_{\sigma}(y_{2}-x_{1})\d y_{1}\d y_{2}\d x_{1}\d x_{2}\\
%&=  n^{2}\lambda^{2}\int_{[0,1]^2\times\R^2}\psi_2\big(z_{1}+\frac{x_{2}-x_{1}}{\sigma}\big)\psi_2\big(z_{2}+\frac{x_{1}-x_{2}}{\sigma}\big)\\
% & \hspace{3cm}\times \psi_1(x_{2}+\sigma z_{1})\psi_1(x_{1}+\sigma z_{2})f(z_{1})f(z_{2})\d z_{1}\d z_{2}\d x_{1}\d x_{2}\\
&\le  n^{2}\lambda^{2}\int_{\R^4}\big|\psi_2\big(z_{1}-\frac{x_{1}}{\sigma}\big)\psi_2\big(z_{2}+\frac{x_{1}}{\sigma}\big)\\
 & \hspace{3cm}\times \psi_1(x_{2})\psi_1(x_{1}+x_{2}+\sigma(z_{2}-z_{1}))\big|f(z_{1})f(z_{2})\d z_{1}\d z_{2}\d x_{1}\d x_{2}\\
%&=  \lambda^{2}\int_{\R^3}\big|\psi_2\big(z_{1}+z_{2}-\frac{x_{1}}{\sigma}\big)\psi_2\big(\frac{x_{1}}{\sigma}\big)\big|(|\psi_1|\ast|\psi_1|)(\sigma z_{1}-x_{1})f(z_{1})f(z_{2})\d z_{1}\d z_{2}\d x_{1}\\
&=  n^{2}\lambda^{2}\int_{\R^2}(f\ast|\psi_2|)\big(\frac{x_{1}}{\sigma}-z_{1}\big)|\psi_2|\big(\frac{x_{1}}{\sigma}\big)(|\psi_1|\ast|\psi_1|)(\sigma z_{1}-x_{1})f(z_{1})\d z_{1}\d x_{1}\\
&=  n^{2}\lambda^{2}\sigma\int_{\R}\big(f\ast\big((f\ast|\psi_2|)(|\psi_1|\ast|\psi_1|)(\sigma\cdot)\big)\big)(x_{1})|\psi_2|(x_{1})\d x_{1}\\
&\le  n^{2}\lambda^{2}\sigma\|\psi_1\|_{L^{2}}^{2}\|\psi_2\|_{L^{1}}^{2}\|f\|_{\infty},\\
J_{4,4} & :=n^{2}\lambda^{2}\int_{[0,1]^2\times\R^2}\psi_2\big(\frac{y_{1}-x_{1}}{\sigma}\big)\psi_2\big(\frac{y_{2}-x_{2}}{\sigma}\big)\psi_1(y_{1})\psi_1(y_{2}) f_{\sigma}(y_{1}-x_{1})f_{\sigma}(y_{2}-x_{1})\d y_{1}\d y_{2}\d x_{1}\d x_{2}\\
 %& =n^{2}\lambda^{2}\sigma\int_{[0,1]\times\R^2}\psi_2\big(z_{1}\big)(\psi_2\ast\eins_{[0,1/\sigma]})\big(z_{2}+\frac{x_{1}}{\sigma}\big)\psi_1(x_{1}+\sigma z_{1})\psi_1(x_{1}+\sigma z_{2})f(z_{1})f(z_{2})\d z_{1}\d z_{2}\d x_{1}\\
 & =n^{2}\lambda^{2}\sigma\int_{0}^{1}\big((\psi_2f)\ast \psi_1(\sigma\cdot)\big)(-x_{1}/\sigma)\big(f\ast((\psi_2\ast\eins_{[0,1/\sigma]})\psi_1(\sigma\cdot))\big)(-x_{1}/\sigma)\d x_{1}\\
 & \le n^{2}\lambda^{2}\sigma^{2}\|(\psi_2f)\ast \psi_1(\sigma\cdot)\|_{L^{2}}\|f\ast((\psi_2\ast\eins_{[0,1/\sigma]})\psi_1(\sigma\cdot))\|_{L^{2}}\\
 & \le n^{2}\lambda^{2}\sigma\|\psi_1\|_{L^{2}}^{2}\|\psi_2\|_{L^{1}}^{2}\|f\|_{\infty}.
\end{align*}
Combining all estimates yields for some $C>0$
\begin{align*}
\Var\Big(\sum_{j,k\in\Z}\phi^{\star}(X_{j},Y_{k})\Big)\le & Cn\lambda\mu(1+\|f\|_{\infty})\|\psi_1\|_{L^{2}}^{2}\\
&\times\Big((\mu+1)\big(n\lambda\sigma+n^{2}\lambda^{2}\sigma^{2}\big)\|\psi_2\|_{L^{1}}^{2}+(n\lambda\sigma+1+\mu)\|\psi_2\|_{L^{2}}^{2}\Big).\qedhere
\end{align*}
\end{proof}
If the test function only depends on $Y_k$, we obtain the following simplified and refined version of Proposition \ref{prop:variance}:
\begin{lem}\label{lem:VarY}
  We have $\Var\Big(\sum_{j}\psi_1(Y_j)\Big)\le n\lambda (\mu+\mu^2) \|\psi_1\|_{L^2}^2$.
\end{lem}
\begin{proof}
By \eqref{eq:BB} of Lemma~\ref{lem:higherOrder}(i), we have:
$$\E\Big[\sum_{j_1\neq j_2}\psi_1(Y_{j_1})\psi_1(Y_{j_2})\Big] = \E\Big[\sum_{j}\psi_1(Y_j)\Big]^2+n\lambda \mu^2\int_{\R \times \R}\psi_1(y_1)\psi_1(y_2)Q_\sigma^2([0,1], \d y_1,\d y_2),$$  
where  $Q_\sigma^2(A,\d y_1,\d y_2) = \big(\int_{A}f_\sigma(y_1-x)f_\sigma(y_2-x)\d x \big)\d y_1 \d y_2$. It follows that
  \begin{align*}
  \Var\Big(\sum_{j}\psi_1(Y_j)\Big)%&=\E\Big[\Big(\sum_{j}\psi_1(Y_j)\Big)^2\Big]-\E\Big[\sum_{j}\psi_1(Y_j)\Big]^2\\
  &=\E\Big[\sum_{j}\psi_1(Y_j)^2\Big]+\E\Big[\sum_{j_1\neq j_2}\psi_1(Y_{j_1})\psi_1(Y_{j_2})\Big]-\E\Big[\sum_{j}\psi_1(Y_j)\Big]^2\\
  &= n\lambda\mu\int_{\R}\psi_1(y)^2(f_\sigma\ast \eins_{[0,1]})(y)\d y\\
  &\quad+n\lambda\mu^2\int_{0}^1\int_{\R \times \R} \psi_1(y_1)\psi_1(y_2)f_\sigma(y_1-x)f_\sigma(y_2-x)\d y_1\d y_2\d x\\
  &=:J_1+J_2.
\end{align*}
These terms are bounded above by 
\begin{align*}
  J_1&\le n\lambda \mu \|\psi_1\|_{L^2}^2\|f_\sigma \ast \eins_{[0,1]}\|_\infty\le n\lambda \mu \|\psi_1\|_{L^2}^2\qquad\text{and}\\
  J_2&\le n\lambda\mu^2 \|(f_\sigma \ast \psi_1)^2\eins_{[0,1]}\|_{L^1}
  \le n\lambda\mu^2 \|f_\sigma\ast \psi_1\|_{L^2}^2
  \le n\lambda\mu^2 \|\psi_1\|_{L^2}^2.\qedhere
\end{align*}
\end{proof}

\begin{rem}\label{rem:VarianceGeneral}
  With only minor modifications the same techniques apply to point processes $M$ and $N$ in $\R^d$ where $M$ has intensity $n\lambda p$ with bounded probability density function $p$ on $\R^d$ and where $N$ has conditional intensity $\mu (M\ast f_\sigma)$ as before. In this case all indicator functions $\eins_{[0,1]}$ have to be replaced by $p$, all integrals of the type $\int_0^1\dots \d x$ have to be replaced by $\int_{\R^d}\dots p(x)\d x$ and the factors $\sigma$ in the above estimates have to be replaced by $\sigma^d$. We then obtain the following variance bounds: 
  \begin{align}
    \Var\Big(\sum_{j,k}\phi^{\star}(X_{j},Y_{k})\Big)\le & Cn\lambda\mu(1+\|f\|_{\infty})(\|p\|_\infty+\|p\|_\infty^3)\notag\\
    &\times\Big((\mu+1)\big(n\lambda\sigma^d+n^{2}\lambda^{2}\sigma^{2d}\big)\|\psi_2\|_{L^{1}}^{2}+(n\lambda\sigma^d+1+\mu)\|\psi_2\|_{L^{2}}^{2}\Big)\|\psi_1\|_{L^{2}}^{2},\label{eq:VarBoundGeneral}\\
    \Var\Big(\sum_{k}\phi^{\star}(Y_{k})\Big)\le &n\lambda (\mu+\mu^2)\|p\|_\infty \|\psi_1\|_{L^2}^2.\notag
  \end{align}
\end{rem}

\subsection{Proof of  upper and lower bounds}\label{sec:ProofsMainResults}
Based on the previous bounds, we can prove our main results.
\subsubsection*{Proof of Theorem~\ref{thm: rates main}}
The theorem is an immediate consequence of Proposition~\ref{thm:rates}.
\begin{proof}[Proof of Proposition~\ref{thm:rates}]
 From \eqref{eq:exp} and Propositions~\ref{prop:bias1} and \ref{prop:bias2}, we conclude
\begin{align}
\E\Big[\frac{1}{n\lambda h_1}\hat{f}_{h_{1},h_{2}}(z_0)\Big] & =\sigma h_2\,\mathcal U_\sigma(f\ast p)+\frac{h_2}{n\lambda}\mathcal V_\sigma(f)\nonumber\\
&=f(z_0)+O\big(h_{1}^{s}\big)+O\Big(\frac{h_{2}}{nh_1}\Big),\quad\text{for }h_{1}\le \frac{h_2}{8},\sigma h_{2}\ge8,\sigma\le1,\label{eq:biasf1}\\
\E\Big[\frac{1}{h_2}\hat{f}_{h_{1},h_{2}}(z_0)\Big] & =h_1\mathcal V_\sigma(f)+\sigma n\lambda h_1\,\mathcal U_\sigma(f\ast p)\nonumber\\
&=f(z_0)+O\big(h_{2}^{s}\big)+\sigma n\lambda,\quad \text{for }h_{2}\le\min(1,\frac1\sigma-h_1),h_{1}\in[4,\frac{1}{\sigma}],\sigma<\frac14\label{eq:biasf2}
\end{align}
Since the kernels from \eqref{eq: def H} and \eqref{eq: def K} satisfy
\[
\|\psi_1\|_{L^{2}}^{2}=O(\sigma^{-1}h_{1}^{-3}),\qquad\|\psi_2\|_{L^{1}}^{2}=O(1),\quad\|\psi_2\|_{L^{2}}^{2}=O(h_{2}^{-1}),
\]
Proposition~\ref{prop:variance} yields
\begin{align*}\Var\Big(\frac{1}{n\lambda h_1}\hat{f}_{h_{1},h_{2}}(z_0)\Big) & \lesssim\frac{h_{2}^{2}\sigma^{2}}{n\sigma h_{1}^{3}}\Big(1+\frac{1}{n\sigma}+\frac{1}{h_{2}n\sigma}+\frac{1}{h_{2}(n\sigma)^{2}}\Big),\\
\Var\Big(\frac{1}{h_2}\hat{f}_{h_{1},h_{2}}(z_0)\Big) & \lesssim\frac{1}{nh_{2}}\Big(1+n\sigma+h_{2}n\sigma+h_{2}(n\sigma)^{2}\Big)\frac{1}{\sigma h_{1}}
\end{align*}
with the constants depending on $\|f\|_\infty,\lambda$ and $\mu$.
Combining these bounds, we conclude:
\begin{enumerate}
\item[\emph{(i)}] If $n\sigma\ge1\ge h_1$ and $h_{2}=8/\sigma$ we obtain
\begin{align*}
\E\Big[\big(\hat{f}_{h_{1}}^{(1)}(z_0)-f(z_0)\big)^{2}\Big] & =\E\Big[\big(\frac{1}{n\lambda h_1}\hat{f}_{h_{1},8/\sigma}(z_0)-f(z_0)\big)^{2}\Big]\\
 & \lesssim h_{1}^{2s}+\frac{(h_{2}\sigma)^{2}}{(n\sigma h_{1})^{2}}+\frac{(h_{2}\sigma)^{2}}{n\sigma h_{1}^{3}}\Big(1+\frac{1}{h_{2}n\sigma}\Big)\lesssim h_{1}^{2s}+\frac{1}{n\sigma h_{1}^{3}}.
\end{align*}
\item[\emph{(ii)}] We have for $h_{1}=\frac{1}{2\sigma}$ and $h_{2}\in(0,1]$ and $\sigma\le1/8$
\begin{align*}
\E\Big[\big(\hat{f}_{h_{2}}^{(2)}(z_0)-f(z_0)\big)^{2}\Big] & =\E\Big[\big(\frac{1}{h_2}\hat{f}_{h_{1},h_{2}}(z_0)-f(z_0)-\sigma n\lambda\big)^{2}\Big]
  \lesssim h_{2}^{2s}+\frac{1}{nh_{2}}\vee\frac{\sigma}{h_{2}}\vee n\sigma^{2}.\quad\qedhere
\end{align*}
\end{enumerate}
\end{proof}

\begin{proof}[Proof of Theorem~\ref{thm:lowerBound}]
Without loss of generality let $z_0=0$. Consider the density $f_{0}(z)=6(\frac{1}{4}-z^{2})\eins_{[-1/2,1/2]}(z)\in\mathcal{G}^{s}(0,L)$  for some $L>0$.
Let $K$ be a {$\max\{2,s+1\}$-times continuously differentiable} function with $\supp K\subset[-1/2,1/2]$
and $K'(0)>0$. Set for some $\epsilon,h>0$
\[
f_{1}(z):=f_{0}(z)+\varepsilon h^{s}K'(z/h),\qquad z\in\R,
\]
where $K'$ denotes the derivative of $K$. Since the compact support
of $K$ implies $\int K'(z)\d z=0$ and because $K'$ is uniformly bounded,
the function $f_{1}$ is a density supported on $[-1/2,1/2]$ if $h$
is small enough. Due to 
\begin{align*}
|f_{1}|_{\mathcal H^{s}(z_0)} & \le |f_{0} |_{\mathcal H^{s}(z_0)}+\varepsilon h^{s}|K'(x/h)|_{\mathcal H^{s}(z_0)} \\
& =|f_{0}|_{\mathcal H^{s}(z_0)}+\epsilon|K'|_{\mathcal H^{s}(z_0)} \\
&=|f_{0}|_{\mathcal H^{s}(z_0)}+\epsilon |K |_{\mathcal H^{s+1}(z_0)},
\end{align*}
 we also conclude $|f_{1}|_{\mathcal H^s(0)}\le L$ up to inflating the value of $L$. The maximum of the two radii defines our $L_0$. Therefore,
we have constructed two alternatives $f_{0},f_{1}\in\mathcal{G}^s(0,L)$ satisfying
\[
|f_{0}(0)-f_{1}(0)|=\epsilon h^{s}|K'(0)|\gtrsim h^{s}.
\]
The lower bounds for the pointwise loss follow from \citet[Thm. 2.2]{tsybakov2009},
if the total variation distance of the corresponding observations
laws remains bounded for the choices
\begin{equation}
h=\begin{cases}
\sigma^{1/(2s+1)}, & \text{if }{n\sigma\ge 1\text{ and }}n\sigma^{(2s+2)/(2s+1)}<1,\\
(\sigma\sqrt{n})^{1/s}, & \text{if }n\sigma^{(2s+2)/(2s+1)}\ge1\text{ and }\sigma<n^{-(4s+3)/(6s+6)},\\
(n\sigma)^{-1/(2s+3)} & \text{if }\sigma\ge n^{-(4s+3)/(6s+6)}.
\end{cases}\label{eq:hChoice-1}
\end{equation}
{(Recall that a lower bound for $\sigma\le n^{-1}$ follows from standard results for nonparametric density estimation.)}
We denote by $\P_{0,n}$ and $\P_{1,n}$ the joint distribution of the point processes $M$ and $N$ where the coniditional intensity measure of $N$ is given by $M\ast f_{0,\sigma}$ and $M\ast f_{1,\sigma}$, respectively.
By conditioning on the number $n_x=M([0,1])\sim\mathrm{Poiss}(\lambda n)$
of parents, the number $n_y=N(\R)\sim\mathrm{Poiss}(\mu n_x)|n_x$
of children and the location of the parents traits $(X_{i})_{i=1,\dots,n_x}\overset{i.i.d}{\sim}\mathrm{U}([0,1])|n_x$
we obtain
\begin{align}
\|\P_{1,n}-\P_{0,n}\|_{TV}= & \sum_{n_x\ge1}\frac{(n\lambda)^{n_x}}{n_x!}e^{-\lambda n}\sum_{n_y\ge0}\frac{(\mu n_x)^{n_y}}{n_y!}e^{-\mu n_x}d(\P_{1,n},\P_{0,n}|n_x,n_y)\qquad\text{where}\label{eq:TV}\\
d(\P_{1,n},\P_{0,n}|n_x,n_y):= & \int_{[0,1]^{n_x}}\Big\|\P_{1,n}^{Y_{1},\dots,Y_{n_y}|X_{1}=x_{1},\dots,X_{n_x}=x_{n_x}}-\P_{0,n}^{Y_{1},\dots,Y_{n_y}|X_{1}=x_{1},\dots,X_{n_x}=x_{n_x}}\Big\|_{TV}\d x,\nonumber 
\end{align}
taking into account that on the event $\{M([0,1])=0\}$ the conditional
distributions coincide. Conditional on $X_{j}=x_{j},j=1,\dots,n_x$,
the offspring $(Y_{i})$ are i.i.d. Estimating the total variation
distance by the $\chi^{2}$-distance { and exploiting the behaviour of the latter under product measures \citep[Chapter 2.4]{tsybakov2009}}, we thus have
\begin{align*}
&d(\P_{1,n},\P_{0,n}|n_x,n_y)^{2}  \le\int_{[0,1]^{n_x}}\big\|\P_{1,n}^{Y_{1},\dots,Y_{n_y}|X_{1}=x_{1},\dots,X_{n_x}=x_{n_x}}-\P_{0,n}^{Y_{1},\dots,Y_{n_y}|X_{1}=x_{1},\dots,X_{n_x}=x_{n_x}}\big\|_{TV}^{2}\d x\\
 &\quad \le\int_{[0,1]^{n_x}}\chi^{2}\big(\P_{1,n}^{Y_{1},\dots,Y_{n_y}|X_{1}=x_{1},\dots,X_{n_x}=x_{n_x}},\P_{0,n}^{Y_{1},\dots,Y_{n_y}|X_{1}=x_{1},\dots,X_{n_x}=x_{n_x}}\big)\wedge1\,\d x\\
 &\quad =\int_{[0,1]^{n_x}}\Big(\Big(1+\chi^{2}(\P_{1,n}^{Y_{1}|X_{1}=x_{1},\dots,X_{n_x}=x_{n_x}},\P_{0,n}^{Y_{1}|X_{1}=x_{1},\dots,X_{n_x}=x_{n_x}})\Big)^{n_y}-1\Big)\wedge1\,\d x\\
 &\quad \le e\int_{[0,1]^{n_x}}\Big(n_y\chi^{2}(\P_{1,n}^{Y_{1}|X_{1}=x_{1},\dots,X_{n_x}=x_{n_x}},\P_{0,n}^{Y_{1}|X_{1}=x_{1},\dots,X_{n_x}=x_{n_x}})\Big)\wedge1\,\d x.
\end{align*}
Under $\P_{k,n}^{Y_{1}|X_{1}=x_{1},\dots,X_{n_x}=x_{n_x}}$, $k\in\{0,1\}$,
$Y_{1}$ is distributed according to the density
\[
g_{n,\sigma}^{(k)}(y|x)=\frac{1}{n_x}\sum_{j=1}^{n_x}f_{k,\sigma}(y-x_{j}).
\]
Therefore,
\begin{align*}
&\chi^2(\P_{1,n}^{Y_{1}|X_{1}=x_{1},\dots,X_{n_x}=x_{n_x}},\P_{0,n}^{Y_{1}|X_{1}=x_{1},\dots,X_{n_x}=x_{n_x}} \big)=\int_{g_{n,\sigma}^{(0)}>0}\frac{(g_{n,\sigma}^{(1)}(y|x)-g_{n,\sigma}^{(0)}(y|x))^{2}}{g_{n,\sigma}^{(0)}(y|x)}\d y\\
 &\qquad  =\frac{\varepsilon^{2}h^{2s}}{\sigma^{2}}\int_{g_{n,\sigma}^{(0)}>0}\Big(\frac{1}{n_x}\sum_{j=1}^{n_x}K'\Big(\frac{y}{\sigma h}-\frac{x_{j}}{\sigma h}\Big)\Big)^{2}\frac{\d y}{g_{n,\sigma}^{(0)}(y|x)}.
\end{align*}
In order to estimate the previous integral, we need a lower bound
for the denominator $g_{n,\sigma}^{(0)}$ on the support 
\begin{align*}
\supp\sum_{j=1}^{n_x}K'\Big(\frac{\cdot}{\sigma h}-\frac{x_{j}}{\sigma h}\Big) & \subset\Big[-\sigma h/2,1+\sigma h/2\Big].
\end{align*}
Defining the event 
$
A:=\{\forall y\in[-\sigma h/2,1+\sigma h/2]:g_{n,\sigma}^{(0)}(y|X)\ge c\},
$
we obtain
\begin{align*}
d(\P_{1,\lambda},\P_{0,\lambda}|n_x,n_y)^{2} & \le e\E_{n}\Big[\Big(n_y\chi^{2}(\P_{1,\lambda}^{Y_{1}|X_{1},\dots,X_{n_x}},\P_{0,\lambda}^{Y_{1}|X_{1},\dots,X_{n_x}})\Big)\wedge1\Big]\\
 & \le\frac{e}{c}\E_{n}\Big[\frac{n_y\varepsilon^{2}h^{2s}}{\sigma^{2}}\int\Big(\frac{1}{n_x}\sum_{j=1}^{n_x}K'\Big(\frac{y}{\sigma h}-\frac{X_{j}}{\sigma h}\Big)\Big)^{2}\d y\Big]+e\P_{n}\big(A^{c}\big).
\end{align*}
where $X_{1},\dots,X_{n_x}\overset{i.i.d.}{\sim}\mathrm{U}([0,1])$
under $\P_{n}$ and $\E_{n}$ denotes the expectation with respect
to $\P_{n}$. Applying Lemma~\ref{lem:auxLowerBound} from below, we have $e\P_{n}\big(A^{c}\big)\le C\sqrt{\frac{\log\sigma^{-1}}{n_x\sigma}}=:r_{n}$.
Hence,
\begin{equation}
d\big(\P_{1,\lambda},\P_{0,\lambda}|n_x,n_y\big)^{2}\le\frac{e}{c}\epsilon^2\frac{n_yh^{2s+1}}{\sigma}\int\E_{n}\Big[\Big(\frac{1}{n_x}\sum_{j=1}^{n_x}K'\Big(y-\frac{X_{j}}{\sigma h}\Big)\Big)^{2}\Big]\d y+r_{n}.\label{eq:chi2Step2}
\end{equation}
 To obtain a sharp upper bound, we will use two different approaches
to estimate the previous display. While the first one will use a stochastic
integral approximation of $\int_{0}^{1/(\sigma h)}K'(y-x)\d x$, the
second approach relies on a numerical approximation. 

In the first case we represent (\ref{eq:chi2Step2}) via
\begin{align}
d\big(\P_{1,n},\P_{0,n}|n_x,n_y\big)^{2}\le&\frac{e}{c}\varepsilon^{2}\frac{n_yh^{2s+1}}{\sigma}\int\E_{n}\Big[\frac{1}{n_x}\sum_{j=1}^{n_x}K'\Big(y-\frac{X_{j}}{\sigma h}\Big)\Big]^{2}\notag\\
&\quad+\Var_{n}\Big(\frac{1}{n_x}\sum_{j=1}^{n_x}K'\Big(y-\frac{X_{j}}{\sigma h}\Big)\Big)\,\d y+r_{n}.\label{eq:chi2Step3}
\end{align}
Owing to
\[
\E\Big[\frac{1}{n_x}\sum_{j=1}^{n_x}K'\Big(y-\frac{X_{j}}{\sigma h}\Big)\Big]=\E\Big[K'\Big(y-\frac{X_{1}}{\sigma h}\Big)\Big]=\sigma h\int_{0}^{1/(\sigma h)}K'(y-x)\d x,
\]
the first term is bounded by
\begin{align}
\frac{n_yh^{2s+1}}{\sigma}\int\E_{n}\Big[\frac{1}{n_x}\sum_{j=1}^{n_x}K'\Big(y-\frac{X_{j}}{\sigma h}\Big)\Big]^{2}\d y & =n_y\sigma h^{2s+3}\int\Big(\int_{0}^{1/(\sigma h)}K'\Big(y-x\Big)\d x\Big)^{2}\d y\nonumber \\
 & =n_y\sigma h^{2s+3}\int\big(K\big(y-\frac{1}{\sigma h}\big)-K(y)\big)^{2}\d y\nonumber \\
 & \le2n_y\sigma h^{2s+3}\|K\|_{L^{2}}^{2}.\label{eq:lowerBoundDecon}
\end{align}
The variance term in (\ref{eq:chi2Step3}) can be estimated by
\begin{align*}
\Var_{n}\Big(\frac{1}{n_x}\sum_{j=1}^{n_x}K'\Big(y-\frac{X_{j}}{\sigma h}\Big)\Big)=\frac{1}{n_x}\Var_{n}\Big(K'\Big(y-\frac{X_{1}}{\sigma h}\Big)\Big) & \le\frac{1}{n_x}\E\Big[K'\Big(y-\frac{X_{1}}{\sigma h}\Big)^{2}\Big].
\end{align*}
Therefore,
\[
\frac{n_yh^{2s+1}}{\sigma}\int\Var_{n}\Big(\frac{1}{n_x}\sum_{j=1}^{n_x}K'\Big(y-\frac{X_{j}}{\sigma h}\Big)\Big)\,\d y\le\frac{n_y}{n_x}\frac{h^{2s+1}}{\sigma}\E_{n}\Big[\int K'\Big(y-\frac{X_{1}}{\sigma h}\Big)^{2}\d y\Big]=\frac{n_y}{n_x}\frac{h^{2s+1}}{\sigma}\|K'\|_{L^{2}}^{2}.
\]
Together with (\ref{eq:chi2Step3}) and (\ref{eq:lowerBoundDecon})
we conclude for some constant $C>0$
\begin{equation}
d^{2}\big(\P_{1,n},\P_{0,n}|n_y,n_x\big)\le C\varepsilon^{2}\Big(n_y\sigma h^{2s+3}+\frac{n_y}{n_x}\frac{h^{2s+1}}{\sigma}\Big)+r_{n}.\label{eq:lowerBoundRegime1}
\end{equation}

In the second and third regime we need a different bound. Applying
a Riemann sum motivated approximation, we decompose
\begin{align*}
\frac{1}{n_x}\sum_{j=1}^{n_x}K'\Big(y-\frac{X_{j}}{\sigma h}\Big) & =\sum_{j=1}^{n_x}\sum_{k=1}^{n_x}\eins_{[(k-1)/n_x,k/n_x)}(X_{j})\int_{(k-1)/n_x}^{k/n_x}K'\Big(y-\frac{X_{j}}{\sigma h}\Big)\d x\\
 & =\sum_{j=1}^{n_x}\sum_{k=1}^{n_x}\eins_{[(k-1)/n_x,k/n_x)}(X_{j})\int_{(k-1)/n_x}^{k/n_x}K'\Big(y-\frac{x}{\sigma h}\Big)\d x\\
 & \quad+\sum_{j=1}^{n_x}\sum_{k=1}^{n_x}\eins_{[(k-1)/n_x,k/n_x)}(X_{j})\int_{(k-1)/n_x}^{k/n_x}\Big(K'\Big(y-\frac{X_{j}}{\sigma h}\Big)-K'\Big(y-\frac{x}{\sigma h}\Big)\Big)\d x\\
 & =:I_{1}(y)+I_{2}(y).
\end{align*}
Therefore, we obtain an alternative bound for (\ref{eq:chi2Step2}):
\begin{align}
d\big(\P_{1,n},\P_{0,n}|n_y,n_x\big)^{2}\le\frac{e\varepsilon^{2}}{c}\frac{n_yh^{2s+1}}{\sigma}\int\E_{n}\Big[I_{1}(y)^{2}\Big]\,\d y+\frac{e\varepsilon^{2}}{c}\frac{n_yh^{2s+1}}{\sigma}\int\E_{n}\Big[I_{2}(y)^{2}\Big]\,\d y+r_{n}.\label{eq:chi2Step4}
\end{align}
For the first term, we calculate
\begin{align*}
\E_{n}\Big[I_{1}(y)^{2}\Big] & =\sum_{j,j'}\sum_{k,k'}\underbrace{\P_{n}\Big(X_{j}\in\Big[\frac{k-1}{n_x},\frac{k}{n_x}\Big),X_{j'}\in\Big[\frac{k'-1}{n_x},\frac{k'}{n_x}\Big)\Big)}_{=n_x^{-2}\text{ if }j\neq j'\text{ and }=n_x^{-1}\text{ if }j=j',k=k'}\\
 & \qquad\qquad\times\int_{(k-1)/n_x}^{k/n_x}\int_{(k'-1)/n_x}^{k'/n_x}K'\Big(y-\frac{x}{\sigma h}\Big)K'\Big(y-\frac{x'}{\sigma h}\Big)\d x'\d x\\
 & =\sum_{k,k'}\Big(\frac{n_x^{2}-n_x}{n_x^{2}}+\frac{n_x}{n_x}\eins_{\{k=k'\}}\Big)\int_{(k-1)/n_x}^{k/n_x}\int_{(k'-1)/n_x}^{k'/n_x}K'\Big(y-\frac{x}{\sigma h}\Big)K'\Big(y-\frac{x'}{\sigma h}\Big)\d x'\d x\\
 & \le \Big(2-\frac{1}{n_x}\Big)\Big(\int_{0}^{1}K'\Big(y-\frac{x}{\sigma h}\Big)\d x\Big)^{2}\\
 & \le2(\sigma h)^{2}\Big(\int_{0}^{1/\sigma h}K'(y-x)\d x\Big)^{2}
  =2(\sigma h)^{2}\big(K(y)-K\big(y-1/(\sigma h)\big)\big)^{2}.
\end{align*}
Hence,
\[
\frac{n_yh^{2s+1}}{\sigma}\int\E_{n}\Big[I_{1}(y)^{2}\Big]\d y\le2n_y\sigma h^{2s+3}\int\big(K(y)-K\big(y-1/(\sigma h)\big)\big)^{2}\d y\le4\|K\|_{L^{2}}^{2}n_y\sigma h^{2s+3}.
\]
The second term in (\ref{eq:chi2Step4}) can be bounded as follows:
\begin{align*}
&\E_{n}\Big[I_{2}(y)^{2}\Big]  =\E_{n}\Big[\Big(\sum_{j=1}^{n_x}\sum_{k=1}^{n_x}\eins_{[(k-1)/n_x,k/n_x)}(X_{j})\int_{(k-1)/n_x}^{k/n_x}\Big(K'\Big(y-\frac{X_{j}}{\sigma h}\Big)-K'\Big(y-\frac{x}{\sigma h}\Big)\Big)\d x\Big)^{2}\Big]\\
 & \le\E_{n}\Big[\Big(\sum_{j=1}^{n_x}\sum_{k=1}^{n_x}\eins_{[(k-1)/n_x,k/n_x)}(X_{j})\int_{(k-1)/n_x}^{k/n_x}\int_{[(x\wedge X_{j})/(\sigma h),(x\vee X_{j})/(\sigma h)]}|K''(y-z)|\d z\d x\Big)^{2}\Big]\\
 & \le\E_{n}\Big[\Big(\sum_{j=1}^{n_x}\sum_{k=1}^{n_x}\eins_{[(k-1)/n_x,k/n_x)}(X_{j})\int_{(k-1)/n_x}^{k/n_x}\int_{(k-1)/(n_x\sigma h)}^{k/(n_x\sigma h)}|K''(y-z)|\d z\d x\Big)^{2}\Big]\\
 & =\E_{n}\Big[\Big(\frac{1}{n_x}\sum_{j=1}^{n_x}\sum_{k=1}^{n_x}\eins_{[(k-1)/n_x,k/n_x)}(X_{j})\int_{(k-1)/(n_x\sigma h)}^{k/(n_x\sigma h)}|K''(y-z)|\d z\Big)^{2}\Big].
\end{align*}
With an analogous calculation as for $\E_n[I_{1}(x)^{2}]$ we obtain
\begin{align*}
\frac{n_yh^{2s+1}}{\sigma}\int\E_{n}\Big[I_{2}(y)^{2}\Big]\,\d y & \le2\frac{n_yh^{2s+1}}{n_x^{2}\sigma}\int\Big(\int_{0}^{1/(\sigma h)}|K''(y-z)|\d z\Big)^{2}\d y\\
 & \le2\frac{n_yh^{2s+1}}{n_x^{2}\sigma}\|K''\|_{L^{1}}\int_{0}^{1/(\sigma h)}\int|K''(y-z)|\,\d y\d z\le2\|K''\|_{L^{1}}^{2}\frac{n_yh^{2s}}{n_x^{2}\sigma^{2}}.
\end{align*}
Therefore, we conclude from (\ref{eq:chi2Step4}) that for some constant
$C'>0$ 
\[
d\big(\P_{1,n},\P_{0,n}|n_y,n_x\big)^{2}\le C'\epsilon^2\big(n_y\sigma h^{2s+3}+\frac{n_yh^{2s}}{n_x^{2}\sigma^{2}}\big)+r_{n}.
\]
In combination with (\ref{eq:lowerBoundRegime1}) we obtain for some
constant $C''>0$
\[
d\big(\P_{1,n},\P_{0,n}|n_y,n_x\big)^{2}\le C''\epsilon^2\min\Big(n_y\sigma h^{2s+3}+\frac{n_y}{n_x}\frac{h^{2s+1}}{\sigma},n_y\sigma h^{2s+3}+\frac{n_yh^{2s}}{n_x^{2}\sigma^{2}}\Big)+C\Big(\frac{\log\sigma^{-1}}{n_x\sigma}\Big)^{1/2}.
\]
If we plug this estimate into \eqref{eq:TV}, we deduce
\begin{align*}
\|\P_{1,n}-\P_{0,n}\|_{TV}^{2} & \le C''\epsilon^2\min\Big(\E[N(\R)]\sigma h^{2s+3}+\E\Big[\frac{N(\R)}{M([0,1])}\Big]\frac{h^{2s+1}}{\sigma},\\
 & \qquad\qquad\qquad\E[N(\R)]\sigma h^{2s+3}+\E\Big[\frac{N(\R)}{M([0,1])^{2}}\eins_{\{M([0,1])>0\}}\Big]\frac{h^{2s}}{\sigma^{2}}\Big)\\
 & \qquad+C\Big(\frac{\log\sigma^{-1}}{\sigma}\E\Big[\frac{1}{M([0,1])}\eins_{\{M([0,1])>0\}}\Big]\Big)^{1/2}.
\end{align*}
Using that $\E[N(\R)|M([0,1])]=\mu M([0,1])$, the remaining expectations
are given by
\begin{gather*}
\E[N(\R)]=\mu\lambda n,\qquad\E\Big[\frac{N(\R)}{M([0,1])}\Big]=\mu,\\
\E\Big[\frac{N(\R)}{M([0,1])^{2}}\eins_{\{M([0,1])>0\}}\Big]=\mu\E\Big[\frac{\eins_{\{M([0,1])>0\}}}{M([0,1])}\Big]=\mu e^{-\lambda}\sum_{n\ge1}\frac{\lambda^{n}}{n\cdot n!}\le\frac{2\mu}{\lambda}e^{-\lambda}\sum_{n\ge1}\frac{\lambda^{n+1}}{(n+1)!}\le\frac{2\mu}{\lambda n}.
\end{gather*}
Therefore,
\[
\|\P_{1,n}-\P_{0,n}\|_{TV}^{2}\le C'''\epsilon^{2}\min\Big(\sigma n\lambda\mu h^{2s+3}+\frac{\mu h^{2s+1}}{\sigma},\sigma n\lambda\mu h^{2s+3}+\frac{\mu h^{2s}}{\lambda n\sigma^{2}}\Big)+C'''\Big(\frac{\log\sigma^{-1}}{\lambda n\sigma}\Big)^{1/2}
\]
where the last term is $o(1)$ by assumption. Based on this estimate,
the theorem follows by verifying that this upper bound remains bounded
for $h$ from (\ref{eq:hChoice-1}).
\end{proof}

\begin{lem}
\label{lem:auxLowerBound}For $g_{n,\sigma}^{(0)}(y|x)=\frac{1}{n}\sum_{j=1}^{n}f_{0,\sigma}(y-x_{j})$
and $f_{0}(z)=6(\frac{1}{4}-z^{2})\eins_{[-1/2,1/2]}(z)$ there is
some $C>0$ such that the event $A:=\Big\{\forall y\in[-\sigma h/2,1+\sigma h/2]:g_{n,\sigma}^{(0)}(y|X)\ge1/14\Big\}$
for $X_{1},\dots,X_{n}\overset{i.i.d.}{\sim}\mathrm{U}([0,1])$ satisfies
$
\P(A)\ge1-C\sqrt{\frac{\log\sigma^{-1}}{n\sigma}}.
$
\end{lem}

\begin{proof}
We first bound the expectation
\[
\E[g_{n,\sigma}^{(0)}(y|X)])=\int_{0}^{1}f_{0,\sigma}(y-x)\d x=\int_{0}^{1/\sigma}f_{0}\big(\frac{y}{\sigma}-x\big)\d x=f_{0}\ast\eins_{[0,1/\sigma]}\big(\frac{y}{\sigma}\big)
\]
uniformly from below: For any $h\in(0,1/2)$ we have 
\begin{align*}
\inf_{y\in[-\sigma h/2,1+\sigma h/2]}f_{0}\ast\eins_{[0,1/\sigma]}\big(\frac{y}{\sigma}\big) & =\inf_{y\in[-\sigma h/2,1+\sigma h/2]}6\int_{-1/2}^{1/2}\big(\frac{1}{4}-z^{2}\big)\eins_{[0,1/\sigma]}\big(\frac{y}{\sigma}-z\big)\d z\\
% & =\min_{y\in\{-\sigma h/2,1+\sigma h/2\}}6\int_{-1/2}^{1/2}\big(\frac{1}{4}-z^{2}\big)\eins_{[(y-1)/\sigma,y/\sigma]}(z)\d z\\
 & =6\int_{h/2}^{1/2}\big(\frac{1}{4}-z^{2}\big)\d z\ge\frac{1}{7}.
\end{align*}
By continuity of $f_{0}$ we deduce
\begin{align*}
\P(A^{c}) & =\P\Big(\inf_{y\in[-\sigma h/2,1+\sigma h/2]\cap\mathbb{Q}}g_{n,\sigma}^{(0)}(y|X)<\frac{1}{14}\Big)\\
% & \le\P\Big(\inf_{y\in[-\sigma h/2,1+\sigma h/2]\cap\mathbb{Q}}\big(\frac{1}{n}\sum_{j=1}^{n}f_{0,\sigma}(y-X_{j})-\E[f_{0,\sigma}(y-X_{j})]\big)<-\frac{1}{14}\Big)\\
 & \le\P\Big(\sup_{y\in[-\sigma h/2,1+\sigma h/2]\cap\mathbb{Q}}\Big|\frac{1}{n\sigma}\sum_{j=1}^{n}\Big(f_{0}\big((y-X_{j})/\sigma\big)-\E\big[f_{0}\big((y-X_{j})/\sigma\big)\big]\Big)\Big|>\frac{1}{14}\Big)\\
 & \le\frac{14}{\sigma\sqrt{n}}\E\Big[\max_{y\in[-\sigma h/2,1+\sigma h/2]\cap\mathbb{Q}}\Big|\frac{1}{\sqrt{n}}\sum_{j=1}^{n}\Big(f_{0}\big((y-X_{j})/\sigma\big)-\E\big[f_{0}\big((y-X_{j})/\sigma\big)\big]\Big)\Big|\Big].
\end{align*}
To bound the previous expectation, we will apply an entropy bound:
Since $f_{0}$ is of bounded variation, the transition class 
\[
\mathcal{F}=\big\{[0,1]\ni x\mapsto f_{0}((y-x)/\sigma)\big|y\in[-\sigma h/2,1+\sigma h/2]\cap\mathbb{Q}\big\}
\]
 is of Vapnik-Cervonenkis type satisfying the covering number bound
$N(\mathcal{F},L^{2}(\mathbb{W}),\epsilon)\le(A/\varepsilon)^{2w}$
for any probability measure $\mathbb{W}$, any $w>3$ and some constant
$A$ which does not depend on the dilation parameter $\sigma$ \citep[Proposition 3.6.12]{GineNickl2016}.
Moreover, $\mathcal{F}$ admits the envelope $F_{\sigma}:=\frac{3}{2}\eins_{[-1,2]}(\cdot/\sigma)$
since $\sup_{z}f_{0}(z/\sigma)=\frac{3}{2}$ and $\supp f_{0}((y-x)/\sigma)\subset\supp F_{\sigma}$
for any $y\in[-\sigma h/2,1+\sigma h/2]$. Theorem~3.5.4 and Remark~3.5.5
by \citet{GineNickl2016} thus yield for some $C>0$
\begin{align*}
 & \E\Big[\max_{y\in[-\sigma h/2,1+\sigma h/2]\cap\mathbb{Q}}\Big|\frac{1}{\sqrt{n}}\sum_{j=1}^{n}f_{0}\big((y-X_{j})/\sigma\big)-\E\big[f_{0}\big((y-X_{j})/\sigma\big)\big]\big)\Big|\Big]\\
 & \quad\le8\sqrt{2}\|F_{\sigma}\|_{L^{2}(\P^{X_{1}})}\int_{0}^{1}\sup_{\mathbb{W}}\sqrt{\log2N(\mathcal{F},L^{2}(\mathbb{W}),\tau\|F_{\sigma}\|_{L^{2}(\P^{X_{1}})})}d\tau\\
 & \quad\le C\|F_{\sigma}\|_{L^{2}(\P^{X_{1}})}|\log(\|F_{\sigma}\|_{L^{2}(\P^{X_{1}})})|^{1/2}.
\end{align*}
Since $\|F_{\sigma}\|_{L^{2}(\P^{X_{1}})}^{2}=\frac{9}{4}\int_{0}^{1}\eins_{[-1,2]}(x/\sigma)\d x\le\frac{27}{4}\sigma$,
we conclude
$
\P(A^{c})\le C\sqrt{\frac{\log\sigma^{-1}}{n\sigma}}.
$
\end{proof}

\begin{proof}[Proof of Theorem \ref{thm: d>1 case}] 

Due to \eqref{eq:ExpN} and \eqref{eq:MeanSD}, we have
$\E[\widehat f_{h_1}^{\rm sd}(z_0)]=K_{h_1}\ast f(z_0)$.
Together with the bias-variance decomposition and a standard bias estimate, we obtain
\[
  \E\big[|\hat f_{h_1}^{\rm sd}(z_0)-f(z_0)|^2\big]\lesssim h_1^{2s}+\Var\big(\hat f_{h_1}^{\rm sd}(z_0)\big).
\]
A bound for the variance of $\hat f_{h_1}^{\rm sp}(z_0)=(n\lambda\mu)^{-1}\sum_{j}\psi_1(Y_j)$ is given in Remark~\ref{rem:VarianceGeneral}.
Applying Plancherel's identity we have
\begin{align*}
  \Var\big(\hat f^{\rm sd}_{h_1}(z_0)\big)\lesssim \frac{1}{n}\|\psi_1\|_{L^2}^2
  &=\frac1{n}\|\F\psi_2\|_{L^2}^2\\
  &\lesssim \frac{\sigma^d}{n}\int_{\R^d}\Big|\frac{e^{-iu^\top y}\F K(hu)}{\phi_p(u/\sigma)}\Big|^2\d u
  \lesssim \frac{\sigma^{2d}}{n}\|K\|_{L^1}^2\int_{|u|\le1/(\sigma h) }|\phi_p(u)|^{-2}\d u
\end{align*}
We conlcude the first inequality in Theorem~\ref{thm: d>1 case}.

In the mildly ill-posed case $|\phi_{p}(u)|\gtrsim(1+|u|^{2})^{-t/2}$
we have $\int_{\{u\in\R^{d}:|u|\le1/(\sigma h)\}}|\phi_{p}(u)|^{-2}\d u\lesssim(\sigma h)^{-2t-d}$.
Therefore,
\[
\E\big[|\hat{f}_{1,h}^{(d)}(z_0)-f(z_0)|^{2}\big]\lesssim h^{2s}+\frac{\sigma^{2d}}{n(h\sigma)^{2t+d}}=h^{2s}+\frac{1}{n\sigma^{2t-d}h^{2t+d}}.
\]
For $h=(n\sigma^{2t+d-2})^{1/(2s+2t+d)}$ we thus obtain the asserted
rate of convergence. In the severely ill-posed case $|\phi_{p}(u)|\gtrsim e^{-\gamma|u|^{\beta}}$
we obtain 
\[
\E\big[|\hat{f}_{h}(z_0)-f(z_0)|^{2}\big]\lesssim h^{2s}+\frac{\sigma^{2d}}{nh\sigma}e^{2\gamma(h\sigma)^{-\beta}}
\]
which yields the claimed rate for $h=\sigma^{-1}(\frac{1}{4\gamma}\log n)^{-1/\beta}.$
\end{proof}

\subsection{Proofs for the scaling estimators}
\begin{proof}[Proof of Lemma~\ref{lem:EstNSigma}]
%  We denote the c.d.f. corresponding to $f_\sigma$ by $F_\sigma$. 
Using \eqref{eq:ExpN}, we obtain
% and integration by parts, 
 \begin{align*}
    \E[\hat T]=\E[N(\R\setminus[0,1])]&=n\lambda \mu\Big(\int_0^1\int_{-\infty}^0f_\sigma(y-x)\d y\d x+\int_0^1\int_1^{\infty}f_\sigma(y-x)\d y\d x\Big)\\
    %&=n\lambda \mu\Big(\int_0^1F_\sigma(-x)\d x+1-\int_0^1 F_\sigma(1-x)\d x\Big)\\
    &=n\sigma\lambda \mu\Big(\int_0^{1/\sigma}F(-x)\d x+\int_0^{1/\sigma}\big(1- F(x)\big)\d x\Big)\\
 %   &=n\sigma\lambda \mu\int_0^{1/\sigma}\big(\P(D\le -x)+\P(D>x)\big)\d x\\
    &=n\sigma\lambda\mu\int_0^{1/\sigma}\P(|D|>x)\d x\\
    &=n\sigma\lambda\mu\E[|D|]
  \end{align*}
where $D$ has density $f$, exploiting in particular the fact that $\P(|D|>x)=0$ for any $x>1/2$ and $1/\sigma>1/2$. The bound for the variance of $\hat T$ follows from Lemma~\ref{lem:VarY} with $\psi_1:=\eins_{[-\sigma/2,0]}+\eins_{[1,1+\sigma/2]}$.
\end{proof}

\begin{proof}[Proof of Proposition~\ref{prop:estSigma}]
First we note, that $l\mapsto \bar X_{(l)}$ is increasing. Indeed, we have
\begin{align*}
  \bar X_{(l+1)}\ge \bar X_{(l)} \Longleftrightarrow \sum_{j=1}^{l+1}X_{(j)}\ge \frac{l+1}{l}\sum_{j=1}^lX_{(j)}
 \Longleftrightarrow \sum_{j=1}^lX_{(j)}+X_{(l+1)} \ge \sum_{j=1}^lX_{(j)}+ \bar X_{(l)},
%  X_{(l+1)}\ge \frac{1}{l}\sum_{j=1}^lX_{(j)}
\end{align*}
and $X_{(l+1)} \ge  \bar X_{(l)}$ holds true since the $X_{(j)}$ are ordered increasingly. Furthermore, Lemma~\ref{lem:EstNSigma} yields
\begin{align*}
  \P\Big(\Big|\frac{\hat l}{\kappa _n \E[\hat T]^{1/2}}-1\Big|>\frac12\Big)
  &=\P\Big(\Big|\frac{\hat T}{\E[\hat T]}-1\Big|>\frac12\Big|\frac{\hat T^{1/2}}{\E[\hat T]^{1/2}}+1\Big|\Big)\\
  &\le\P\Big(\Big|\frac{\hat T}{\E[\hat T]}-1\Big|>\frac12\Big)\le 2\frac{\Var(\hat T)}{\E[\hat T]^2}\lesssim \frac{1}{n\sigma}\to 0. 
\end{align*}

Consequently, with $l:=\frac{\kappa_n}2\E[\hat T]^{1/2}$ of order $\kappa _n\sqrt{\sigma n}$, the event 
$$\Lambda:=\{l\le \hat l\le 3 l\} = \Big\{\Big|\frac{\hat l}{\kappa _n \E[\hat T]^{1/2}}-1\Big|>\frac12\Big\}$$
satisfies $\P(\Lambda)\to 1$. The statement of Proposition \ref{prop:estSigma} being equivalent to  
$\hat{\sigma}^{(1)}-\sigma = O_{\mathbb P}(\kappa_n \sqrt{\sigma/n})$, setting $\varepsilon:=C\frac{l}{n}$, 
%\[
%\varepsilon:=C\frac{l}{n}\qquad\text{and}\qquad R:=X_{(l+1)}-\bar{X}_{(l)}\in[0,1].
%\]
for some $C>0$, it is enough to show that $\P\big(|\hat{\sigma}^{(1)}-\sigma|>2\varepsilon\big)$ can be made arbitrarily small by taking $C$ (and $n$) sufficiently large.
Let $R:=X_{(l+1)}-\bar{X}_{(l)}\in[0,1]$. We have
\begin{align*}
\P\big(|\hat{\sigma}^{(1)}-\sigma|>2\varepsilon\big) 
&=\P\big(\hat{\sigma}^{(1)}-\sigma>2\epsilon\big)+\P\big(\sigma-\hat{\sigma}^{(1)}>2\epsilon\big)\\
 & \le\P\big(\hat{\sigma}^{(1)}-\sigma>2\epsilon\big)+\P\big(\sigma-\hat{\sigma}^{(1)}>2R\big)+\P(R>\varepsilon)\\
 & =:T_{1}+T_{2}+T_{3}.
\end{align*}
We will consider all three terms separately. For $T_{1}$,  using that the support of the offspring location trait is $[-\tfrac{\sigma}{2}, \tfrac{\sigma}{2}]$, we have, on the event $\Lambda$:
\[
\hat{\sigma}^{(1)}=-2(\min_{j}Y_{j}-X_{(1)})-2X_{(1)}+2\bar{X}_{(\hat l)}\le\sigma-2X_{(1)}+2{\bar{X}_{(3 l)}}.
\]
Since conditional on $|\mathcal{X}|=n_{x}$ it holds $X_{(l)}\sim\mathrm{Beta}(l,n_{x}+1-l)$,
we obtain 
\begin{align*}
T_{1}&=\P\big(\hat{\sigma}^{(1)}-\sigma>2\epsilon\big)  \le\P\big(-2X_{(1)}+2\bar{X}_{(3 l)}>2C\frac{l}{n}\big) +\P(\Lambda^c)\\
 & \le\frac{2n}{Cl}\E\big[\bar{X}_{(3 l)}\big]+o(1)
 =\frac{2n}{Cl}\frac{1}{3 l}\sum_{i=1}^{3 l}\E\Big[\frac{i}{|\mathcal{X}|+1}\Big]+o(1)
 \le\frac{2n}{Cl}\E\Big[\frac{3 l}{|\mathcal{X}|+1}\Big]+o(1).
\end{align*}
Due to $|\mathcal{X}|\sim\mathrm{Poiss}(\lambda n)$,  we have $n\E\big[\frac{1}{|\mathcal{X}|+1}\big] = \tfrac{1}{\lambda}(1-\exp(-\lambda n)) = O(1)$ hence $T_{1}$ can be made
arbitrarily small for sufficiently large $C$.

To bound $T_2$, we note that $R$ only depends on $\mathcal{X}$. Also, conditional on $\mathcal{X}$, the offspring trait $Y_{1}$ has 
distribution function $F_{\sigma|\mathcal{X}}(z):=\frac{1}{|\mathcal{X}|}\sum_{i}F_{\sigma}(z-X_{i})$
where $F_{\sigma}=F(\cdot/\sigma)$ is the cumulative distribution function corresponding to
the scaled dispersal density $f_{\sigma}$. Therefore, we bound $T_{2}$
as follows:
\begin{align*}
T_{2}=\P(\hat{\sigma}^{(1)}<\sigma-2R) & \le\P\big(\min_{j}Y_{j}>R-\frac{\sigma}{2}+{\bar{X}_{(l)}}\big)+\P(\Lambda^c)\\
 & =\P\big(\forall j:Y_{j}>X_{(l+1)}-\frac{\sigma}{2}\big)+o(1)\\
 & =\E\Big[\P\big(\forall j:Y_{j}>X_{(l+1)}-\frac{\sigma}{2}\big|\mathcal{X},|\mathcal{Y}|\big)\Big]+o(1)\\
 & =\E\Big[\P\big(Y_{1}>X_{(l+1)}-\frac{\sigma}{2}\big|\mathcal{X}\big)^{|\mathcal{Y}|}\Big]+o(1)\\
 & =\E\Big[\big(1-F_{\sigma|\mathcal{X}}\big(X_{(l+1)}-\frac{\sigma}{2}\big)\big)^{|\mathcal{Y}|}\Big]+o(1).
\end{align*}
The boundary assumption \eqref{eq:boundaryCondition} on $f = F'$ yields $F(z-\frac{1}{2}) \ge (\gamma z) \vee 0$ for $z \le 1$. It follows that
\begin{align*}
F_{\sigma|\mathcal{X}}\big(X_{(l+1)}-\frac{\sigma}{2}\big) & =\frac{1}{|\mathcal{X}|}\sum_{i}F_{\sigma}\big(X_{(l+1)}-\frac{\sigma}{2}-X_{i}\big)\\
 & \ge\frac{\gamma}{\sigma|\mathcal{X}|}\sum_{i}\big(0\vee\big(X_{(l+1)}-X_{i}\big)\big)\\
 & =\frac{\gamma}{\sigma|\mathcal{X}|}\sum_{i:X_{i}<X_{(l+1)}}\big(X_{(l+1)}-X_{i}\big)\\
 & =\frac{\gamma}{\sigma|\mathcal{X}|}\sum_{i=1}^{l}\big(X_{(l+1)}-X_{(i)}\big)=\frac{\gamma lR}{\sigma|\mathcal{X}|}.
\end{align*}
Pick  $C_{1}>0$. We infer from the previous bound
\begin{align*}
T_{2} & \le\E\Big[\big(1-\frac{\gamma lR}{\sigma|\mathcal{X}|}\big)^{|\mathcal{Y}|}\Big]\\
 & \le\E\Big[\big(1-\frac{\gamma C_{1}}{|\mathcal{X}|}\big)^{|\mathcal{Y}|}\Big]+\P(lR<C_{1}\sigma)\\
 & \le\exp(-\gamma C_{1}^{1/3})+\P(|\mathcal{Y}|<C_{1}^{-1/3}n)+\P(|\mathcal{X}|>C_{1}^{1/3}n)+\P(lR<C_{1}\sigma)
\end{align*}
using $(1-\kappa)^{n_{y}}\le\exp(n_{y}\log(1-\kappa))\le\exp(-n_{y}\kappa)$
for $\kappa=\gamma C_1^{2/3}/n\to0$. Since $|\mathcal{X}|\sim\mathrm{Poiss}(\lambda n)$,
conditional on $|\mathcal{X}|$, we have $|\mathcal{Y}|\sim\mathrm{Poiss}(\mu|\mathcal{X}|)$,
and for arbitrary small $\delta$ and sufficiently large
$C_{1}=C_{1}(\delta,\lambda,\mu,\gamma)$, we conclude 
\[
T_{2}\le\delta+\P(lR<C_{1}\sigma).
\]

We bound $T_{3}+\P(lR<C_{1}\sigma)$ in the same line of arguments, having now
\begin{align*}
T_{3}+\P(lR<C_{1}\sigma) & =\P(R>\varepsilon)+\P(R<C_{1}\frac{\sigma}{l}).\\
 & \le\varepsilon^{-1}\E[R]+\P(R-\E[R|\,|\mathcal{X}|]<C_{1}\frac{\sigma}{l}-\E[R|\,|\mathcal{X}|])\\
 & \le\varepsilon^{-1}\E[R]+\E\Big[\frac{\Var(R|\,|\mathcal{X}|)}{(\E[R|\,|\mathcal{X}|]-C_{1}\frac{\sigma}{l})^{2}}\Big].
\end{align*}
Indeed, the computation below shows that $\E[R|\,|\mathcal{X}|]$ is of order $l/|\mathcal X| \approx \kappa_n \sqrt{\sigma/n}$ with $|\mathcal X|$ of order $n$ while $C_{1}\frac{\sigma}{l}$ is of order $C_1 \kappa_n^{-1}\sqrt{\sigma/n}$. Having $\kappa_n$ slowly diverging guarantees that the event has vanishing probability, even for arbitrarily big (but fixed) $C_1$. More precisely,
note first that $\P(|\mathcal{X}|=0)\to0$ as $n\to\infty$. For $n_{x}\ge1$
Then, we explicitly compute 
\begin{align*}
\E[R|\,|\mathcal{X}|=n_{x}] & =\E\Big[X_{(l+1)}-\bar{X}_{(l)}|\,|\mathcal{X}|=n_{x}\Big]\\
 & =\frac{l+1}{n_{x}+1}-\frac{1}{l}\sum_{j=1}^{l}\frac{j}{n_{x}+1}\\
 & =\frac{l+1}{n_{x}+1}-\frac{l+1}{2(n_{x}+1)}\\
 & =\frac{l+1}{2(n_{x}+1)}\in\big(\frac{l}{4n_{x}},\frac{l}{n_{x}}\big).
\end{align*}
The properties of order statistics under the uniform distribution yield $\Cov(X_{(j)},X_{(k)}|\,|\mathcal{X}|=n_{x})=\frac{j(n_{x}-k+1)}{(n_{x}+1)^{2}(n_{x}+2)}\le\frac{l+1}{n_{x}^{2}}$
for $1\le j\le k\le l+1$. We infer
\begin{align*}
\Var(R|\,|\mathcal{X}|=n_{x}) & =\Var(X_{(l+1)}|\,|\mathcal{X}|=n_{x})  -\frac{2}{l}\sum_{j=1}^{l}\Cov(X_{(j)},X_{(l+1)}|\,|\mathcal{X}|=n_{x})\\
 & +\frac{1}{l^{2}}\sum_{j_{1},j_{2}=1}^{l}\Cov(X_{(j_{1})},X_{(j_{2})}|\,|\mathcal{X}|=n_{x})\lesssim \frac{l+1}{n_{x}^{2}}.
\end{align*}
Therefore,
\begin{align*}
T_{3}+\P(lR<C_{1}\sigma) & \lesssim\frac{l}{\varepsilon n}+\frac{l}{n^{2}}\Big(\frac{l}{8n}-C_{1}\frac{\sigma}{l}\Big)^{-2}\\
 & =\frac{1}{C}+\frac{1}{l}\Big(\frac{1}{8}-C_{1}\frac{n\sigma}{l^2}\Big)^{-2}
 \lesssim\frac{1}{C}+O\Big(\frac{1}{\kappa_{n}\sqrt{\sigma n}}\Big).
\end{align*}
This upper bound is arbitrary small for sufficiently large $C$ and
$n$. 
%{\color{red} [explain better why $\kappa_n$ needs to diverge here!?]}
\end{proof}

\begin{proof}[Proof of Proposition~\ref{prop:sigma2}]
We decompose for some arbitrary $c\in(0,1)$ and $\epsilon\in(0,1/2)$
\begin{align*}
\P\big(\big|\frac{\hat{\sigma}^{(2)}}{\sigma}-1\big|\ge\epsilon\big)
&=\P\big(\big|\hat{\sigma}^{(2)}-\sigma\big|\ge\sigma\epsilon\big)\\
&\le  \P\big(\hat{\sigma}^{(2)}-\sigma\ge\sigma\epsilon\big)+\P\big(\sigma-\hat{\sigma}^{(2)}\ge\sigma\epsilon\big)\\
&\le  \P\big(\hat{\sigma}^{(2)}\ge\sigma(1+\epsilon)\big)+\P\big(\sigma(1-\epsilon)\ge\hat{\sigma}^{(2)}\ge c\sigma\big)+\P\big(\hat{\sigma}^{(2)}< c\sigma\big)\\
&=:P_{1}+P_{2}+P_3.
\end{align*}
In the following we will prove that all three probabilities tend to zero. To this end, note that we can write $\hat{\sigma}^{(2)}$ as
\[
\hat{\sigma}^{(2)}=\min\Big\{ h>0:{\E[\psi^\dagger(\sigma D_1/h)]}+\xi(h)\ge{\psi^\dagger(0)}-\sqrt{nh^{2}+n^{-1}}\kappa_{n}\Big\}.
\]

For the first term $P_{1}$, we set $h^{\circ}=\sigma$. {Since $\psi^\dagger$ is constant on the support of $D_1$, we have $\E[\psi^\dagger(\sigma D_1/h^\circ)]=\E[\psi^\dagger(D_1)]=\psi^\dagger(0)$.} Therefore,
\begin{align*}
P_{1}\le\P(\hat{\sigma}^{(2)}>h^{\circ}) 
& \le\P\big(\xi(h^{\circ})+{\E[\psi^\dagger(\sigma D_1/h^\circ)]}<\psi^\dagger(0)-\sqrt{n(h^{\circ})^{2}+n^{-1}}\kappa_{n}\big)\\
 & =\P\big(-\xi(h^{\circ})>\sqrt{n(h^{\circ})^{2}+n^{-1}}\kappa_{n}\big)\\
 & \le\P(\Xi^{c})\to0
\end{align*}
with the good event $\Xi$ from Lemma~\ref{lem:AuxXi}.

To bound the second probability $P_{2}$, we set $h^{*}=\sigma(1-\varepsilon)$. 
 Since 
$h \mapsto \E[\psi^\dagger(\sigma D_1/h)]+\sqrt{n(h)^{2}+n^{-1}}\kappa_{n}$ is non-decreasing, on the event  $\Xi\cap\{h^{*}\ge\hat{\sigma}^{(2)}>c\sigma\}$ for some arbitrary $c \in (0,1)$, we have
\begin{align*}
{\E[\psi^\dagger(\sigma D_1/h^*)]}+\sqrt{n(h^{*})^{2}+n^{-1}}\kappa_{n} & \ge{\E[\psi^\dagger(\sigma D_1/\hat \sigma^{(2)})]}+\sqrt{n(\hat{\sigma}^{(2)})^{2}+n^{-1}}\kappa_{n}\\
 & \ge{\E[\psi^\dagger(\sigma D_1/\hat\sigma^{(2)})]}+\xi(\hat{\sigma}^{(2)})\\
 & ={\psi^\dagger(0)}-\sqrt{n(\hat{\sigma}^{(2)})^{2}+n^{-1}}\kappa_{n}\\
 & \ge{\psi^\dagger(0)}-\sqrt{n(h^{*})^{2}+n^{-1}}\kappa_{n},
\end{align*}
 where we used Lemma \ref{lem:AuxXi} for the second line with the constant $c$ chosen accordingly and using the fact that $\kappa_n \rightarrow \infty$.
Since $\supp f\subset[-1/2,1/2]$ and $f$ is bounded from below by \eqref{eq:boundaryCondition}, using the specific choice of $\psi^\dagger$, 
we conclude, on the event $\Xi\cap\{h^{*}\ge\hat{\sigma}^{(2)}\ge c\sigma\}$:
\begin{align*}
2\sqrt{n(h^{*})^{2}+n^{-1}}\kappa_{n}
&\ge\E\big[\psi^\dagger(0)-\psi^\dagger\big(\sigma D_1/\sigma(1-\varepsilon)\big)\big]\;\;\text{since}\;\;h^*=\sigma(1-\varepsilon)\\
&\ge\min_{|x|\le1/2}f(x)\int_{-1/2}^{1/2}\Big(\psi^\dagger(0)-\psi^\dagger\big(\frac{x}{1-\epsilon}\big)\Big)dx\\
&=(1-\epsilon)\min_{|x|\le1/2}f(x)\int_{0}^{\epsilon/2(1-\epsilon)}\Big(\psi^\dagger(0)-\psi^\dagger(x+1/2)\Big)dx\\
%&=\min_{|x|\le1/2}f(x)\frac{2^\alpha}{C^\dagger}\int_{0}^{\epsilon/2}x^\alpha dx\\
%&{=\frac{\epsilon^{1+a}}{2(1+\alpha)}}\,\min_{|x|\le1/2}f(x).
&\ge \frac{C^\dagger \log 2}{4} \min_{|x|\le1/2}f(x) \frac{\varepsilon}{\log (4\varepsilon^{-1})}\\
&\ge \frac{C^\dagger\log2}{12}\min_{|x|\le1/2}f(x) \frac{\varepsilon}{\log (\varepsilon^{-1})}
\end{align*}
Hence, for 
\begin{align*}
\varepsilon/\log( \varepsilon^{-1}) & =\frac{12}{C^\dagger \log 2}(\min_{|x|\le1/2}f(x))^{-1}\sqrt{n\sigma^{2}+n^{-1}}\kappa_{n} \\
& >\frac{12}{C^\dagger \log 2}(\min_{|x|\le1/2}f(x))^{-1}\sqrt{n(h^{*})^{2}+n^{-1}}\kappa_{n}
\end{align*}
we have
\begin{align*}
P_{2} & \le\P(\Xi^{c})\to0.
\end{align*}
Setting $\alpha(\varepsilon) := \varepsilon /\log (\varepsilon^{-1})$, the rate of convergence is given by $\alpha^{-1}(\sqrt{n\sigma^{2}+n^{-1}}\kappa_{n})$. From the fact that $\alpha$ is non-decreasing for $\varepsilon\in(0,1)$, we deduce $\alpha^{-1}(\varepsilon) \lesssim \varepsilon (\log (1/\varepsilon))^{1+a}$ for arbitrarily small $a$ as $\varepsilon\downarrow0$. To deduce the claimed rate, it suffices to note that
\[
  \kappa_{n}\sqrt{n\sigma^{2}+n^{-1}}\big|\log (\kappa_{n}\sqrt{n\sigma^{2}+n^{-1}})\big|^{1+a}\lesssim \kappa_n(\log n)^{1+a}\sqrt{n\sigma^{2}+n^{-1}}\le (\log n)^{2}\sqrt{n\sigma^{2}+n^{-1}}
\]
for $\kappa_n=\sqrt{\log n}$ and $n\sigma^{3/2}\to0$.\\ 
%hence $ which is smaller than any $v_n$ for which $\sqrt{n\sigma^{2}+n^{-1}}\kappa_{n} \ge \alpha(v_n)$.}\\

It remains to prove $P_3\to0$ { where we choose $c=1/2$}. 
Since $h \mapsto \psi^\dagger(x/h)$ is non-decreasing, so is $h\mapsto\frac{1}{\mu\lambda n}\sum_{i,j}\psi^\dagger\big((Y_{j}-X_{i})/h\big)$. Therefore, on the event $\{\bar h\ge\hat{\sigma}^{(2)}>\underline h\}$ for any $0\le\underline h<\bar h<\sigma$ we have that
\begin{align*}
 \frac{1}{\mu\lambda n}\sum_{i,j}\psi^\dagger\big((Y_{j}-X_{i})/\bar h\big)
 &\ge \frac{1}{\mu\lambda n}\sum_{i,j}\psi^\dagger\big((Y_{j}-X_{i})/\hat\sigma^{(2)}\big)\\
 &= n\lambda \hat{\sigma}^{(2)}+{\psi^\dagger(0)}-\sqrt{n(\hat{\sigma}^{(2)})^{2}+n^{-1}}\kappa_{n}\\
 &\ge n\lambda \underline h+{\psi^\dagger(0)}-\sqrt{n\sigma^2+n^{-1}}\kappa_n.
\end{align*}
Therefore, assuming $\kappa_n$ is such that $\sqrt{n\sigma^{2}+n^{-1}}\kappa_{n}=o(1)$, a choice which is always possible, and under the condition $n\lambda(\bar h-\underline h)<{\psi^\dagger(0)-\E[\psi^\dagger(\sigma D_1/\bar h)]}$, Markov's inequality yields 
\begin{align}
\P\big(\bar h\ge\hat{\sigma}^{(2)}>\underline h\big)
& \le \P\Big(\frac{1}{\mu\lambda n}\sum_{i,j}\psi^\dagger\big((Y_{j}-X_{i})/\bar h\big)\ge n\lambda \underline h+{\psi^\dagger(0)}-o(1)\Big)\notag\\
& = \P\Big(\xi(\bar h)\ge n\lambda (\underline h-\bar h)+{\psi^\dagger(0)}-{\E[\psi^\dagger(\sigma D_1/\bar h)]}-o(1)\Big)\notag\\
& \le \P\Big(\xi(\bar h)\ge {\psi^\dagger(0)-\E[\psi^\dagger(\sigma D_1/\bar h)]}-n\lambda (\bar h-\underline h)-o(1)\Big)\notag\\
 & \le\frac{\Var(\xi(\bar h))}{\big({\psi^\dagger(0)-\E[\psi^\dagger(\sigma D_1/\bar h)]}-n\lambda(\bar h -\underline h)-\E[\xi(\bar h)]-o(1)\big)^{2}}\notag\\
 %& \le\frac{n\bar h^2+n^{-1}}{\big({1-\E[\psi^\dagger(\sigma D_1/\bar h)]}-n\lambda(\bar h -\underline h)+O(n\sigma^2+\sigma)+o(1)\big)^{2}}\notag\\ 
 & \le\frac{n\bar h^2+n^{-1}}{\big({\psi^\dagger(0)-\E[\psi^\dagger(\sigma D_1/\bar h)]}-n\lambda(\bar h -\underline h)+o(1)\big)^{2}},\label{eq:boundSigmaHat} 
\end{align}
where we used \eqref{eq:xiVar} for the last estimate.
In the case $n\sigma\le c_1:= {\E[\psi^\dagger(0)-\psi^\dagger(2D_1)]/\lambda}$, we can choose $\underline h=0,\bar h=\sigma/2$ and conclude
\[
  \P\Big(\hat \sigma^{(2)}\le \frac{\sigma}{2}\Big)
  \lesssim  \frac{n\sigma^2+n^{-1}}{\big({\E[\psi^\dagger(0)-\psi^\dagger(2D_1)]}-n\sigma\lambda/2+o(1)\big)^{2}}\to 0.
\]
If $n\sigma> c_1$, we first note that \eqref{eq:boundSigmaHat} with $\underline h=0$, {$\bar h=\frac{c_1}{2n}$} yields 
\begin{align*}
  \P\Big(\hat \sigma^{(2)}\le \frac{c_1}{2n}\Big)
  &\lesssim  \frac{n^{-1}}{\big({\E[\psi^\dagger(0)-\psi^\dagger(2n\sigma D_1/c_1)]}-c_1\lambda/2+o(1)\big)^{2}}\\
  &\le  \frac{n^{-1}}{\big({\frac12\E[\psi^\dagger(0)-\psi^\dagger(4D_1)]+o(1)}\big)^{2}}\to 0.
\end{align*}
To improve this bound in the case $\sigma>\frac1n$, we choose $h_i:=\frac{c_1}{2n}+\frac{c_2}{2}\frac in\in[\frac{c_1}{2n},\frac{c_2}{2n}+\frac{\sigma}{2}]$ for $c_2:=\min\{\frac12,\E[\psi^\dagger(0)-\psi^\dagger(4 D_1/3)]/\lambda\}$ and $0\le i\le I:= \lceil \frac{n}{c_2}\big(\sigma-\frac{c_1}{n}\big)\rceil= O(\sigma n)$ and estimate using \eqref{eq:boundSigmaHat}
\begin{align*}
  \P\Big(\hat \sigma^{(2)} \le \frac\sigma2\Big)&\le\P\Big(\hat \sigma^{(2)}\le\frac{c_1}{2n}\Big)+\sum_{i=1}^I\P\big(h_{i-1}<\hat \sigma^{(2)}\le h_i\big)\\
  &\le \sum_{i=1}^I \frac{n h_i^2+n^{-1}}{\big({\E[\psi^\dagger(0)-\psi^\dagger(\sigma D_1/h_i)]}-n\lambda(h_i - h_{i-1})+o(1)\big)^{2}}+ o(1)\\
  &\le \sum_{i=1}^I \frac{n h_i^2+n^{-1}}{\big({\E[\psi^\dagger(0)-\psi^\dagger(\sigma D_1/(\frac{c_2}{2n}+\frac{\sigma}{2}))]-c_2}\lambda/2+o(1)\big)^{2}}+ o(1)\\
  &\le \sum_{i=1}^I \frac{n h_i^2+n^{-1}}{\big({\E[\psi^\dagger(0)-\psi^\dagger(4 D_1/3)]}-c_2\lambda/2+o(1)\big)^{2}}+ o(1)\\
  &\lesssim \frac{I}{n}+\frac{1}{n}\sum_{i=1}^Ii^2+ o(1)\\
  &\lesssim \frac{I}{n}+\frac{I^3}{n}+ o(1)
  \lesssim \sigma +n^2\sigma^3.\qedhere
\end{align*}
\end{proof}

\begin{lem}
\label{lem:AuxXi}Let $c\in(0,1)$ and $\sqrt{n}\sigma=O(1)$. For every $\varepsilon>0$ there exists $\kappa>0$ such that 
\begin{align}
\P(\Xi^{c})\le\varepsilon\qquad\text{with }\quad\Xi:=\Big\{\sup_{h\in[c\sigma,\sigma]}\frac{|\xi(h)|}{\sqrt{nh^{2}+n^{-1}}}\le\kappa\Big\}.\label{eq:Xi}
\end{align}
\end{lem}

\begin{proof}
\emph{Step 1: We first bound $\E[\xi(h)]$.} For $D_1\sim f$ and
$Y\sim f_{\sigma}\ast p$ Corollary~\ref{cor: intensity} yields
\begin{align*}
\E\Big[\frac{1}{\mu\lambda n}\sum_{i,j}\psi^\dagger((Y_{j}-X_{i})/h)\Big] & =\int\psi^\dagger(z/h)f_{\sigma}(z)\d z+n\lambda\int_{0}^{1}\int\psi^\dagger(\frac{y-x}{h})(f_{\sigma}\ast p)(y)\d y\d x\\
 & =\E[\psi^\dagger(\sigma D_1/h)]+n\lambda h\int_{0}^{1}\E[\psi^\dagger_{h}(Y-x)]\d x\\
 & =\E[\psi^\dagger(\sigma D_1/h)]+n\lambda h-n\lambda h\int_{\R\setminus[0,1]}\E[\psi^\dagger_{h}(Y-x)]\d x,
\end{align*}
using $\int_{\R}\E[\psi^\dagger_{h}(Y-x)]\d x=\int\psi^\dagger_{h}(x)\d x=1$. Since
and $\psi_h^\dagger \le h^{-1}\|\psi^\dagger\|_\infty \eins_{[-1, 1]}$, $\supp(f_\sigma\ast p)=[-\frac{\sigma}{2},1+\frac{\sigma}{2}]$ and  $\|f_{\sigma}\ast p\|_{\infty}\le1$, we successively have 
\begin{align*}
\int_{\R\setminus[0,1]}\E[\psi^\dagger_{h}(Y-x)]\d x & =\E\Big[\int_{(-\infty,Y-1]\cup[Y,\infty)}\psi^\dagger_{h}(x)\d x\Big]\\
 & \le \frac1h\|\psi^\dagger\|_{\infty}\E\big[\eins_{\{-h<Y-1\}}\big(h\wedge(Y-1)+h\big)+\eins_{\{Y<h\}}\big(h-(-h\vee Y)\big)\big]\\
 & \le2\|\psi^\dagger\|_{\infty}\P(\{Y>1-h\}\cup\{Y<h\})\\
 & \le2\|\psi^\dagger\|_{\infty}(2h+\sigma).
\end{align*}
%Let $\psi=\eins_{[-\frac{1}{2},\frac{1}{2}]}\ast K$ for 
Therefore,
\[
\E[\xi(h)]=\E\Big[\frac{1}{\mu\lambda n}\sum_{i,j}\psi^\dagger(Y_{j}-X_{i})\Big]{-\E\big[\psi^\dagger(\sigma D_1/h)\big]}-n\lambda h=O\big(nh(h+\sigma)\big).
\]
In particular for all $h\in(0,\sigma)$
\[
\frac{|\E[\xi(h)]|}{\sqrt{nh^{2}+n^{-1}}}\lesssim \frac{nh(\sigma+h)}{\sqrt{n}h} = O(\sqrt{n}\sigma)
\]
which is uniformly bounded. 
\emph{Step 2: It remains to prove the tightness of 
\[
\sup_{h\in[c\sigma,\sigma]}\frac{|\xi(h)-\E[\xi(h)]|}{\sqrt{nh^{2}+n^{-1}}}.
\]
}To this end, we apply the Kolmogorov-Chentsov criterion to the process
\[
V_{t}:=\frac{\xi(h_{t})-\E[\xi(h_{t})]}{\sqrt{nh_{t}^{2}+n^{-1}}},\qquad h_{t}:=t\sigma,t\in(c,1].
\]
First, due to (\ref{eq:xiVar}), we have $\E[V_{c}^2] \lesssim 1$. Next, for $0<s<t\le1$
we decompose increments into
\begin{align*}
V_{t}-V_{s} & =\frac{1}{n\lambda\mu}\sum_{i,j}\big(\Delta_{s,t}(Y_{j}-X_{i})-\E[\Delta_{s,t}(Y_{j}-X_{i})]\big),\\
\Delta_{s,t}(z) & :=\frac{1}{\sqrt{nh_{t}^{2}+n^{-1}}}\psi^\dagger(z/h_{t})-\frac{1}{\sqrt{nh_{s}^{2}+n^{-1}}}\psi^\dagger(z/h_{s}).
\end{align*}
Proposition~\ref{prop:variance} with $\psi_{2}=\Delta_{s,t}(\sigma\cdot)$ and $\psi_{1}=\eins_{[-2,2]}$
yields 
\begin{align}
\Var\Big(\frac{1}{n\lambda\mu}\sum_{i,j}\big(\Delta_{s,t}(Y_{j}-X_{i})\Big) & \lesssim\frac{1}{n}\big((n\sigma+n^{2}\sigma^{2})\sigma^{-2}\|\Delta_{s,t}\|_{L^{1}}^{2}+(n\sigma+1)\sigma^{-1}\|\Delta_{s,t}\|_{L^{2}}^{2}\big)\nonumber \\
 & =(\sigma^{-1}+n)\|\Delta_{s,t}\|_{L^{1}}^{2}+(1+(n\sigma)^{-1})\|\Delta_{s,t}\|_{L^{2}}^{2}.\label{eq:VarIncr}
\end{align}
We can bound the above $L^{1}$-norm by
\begin{align*}
\|\Delta_{s,t}\|_{L^{1}}^{2} & \le2\Big(\frac{1}{\sqrt{nh_{t}^{2}+n^{-1}}}-\frac{1}{\sqrt{nh_{s}^{2}+n^{-1}}}\Big)^{2}\Big(\int|\psi^\dagger(z/h_{t})|\d z\Big)^{2}\\
 & \qquad+\frac{2}{nh_{s}^{2}+n^{-1}}\Big(\int|\psi^\dagger(z/h_{t})-\psi^\dagger(z/h_{s})|\d z\Big)^{2}
  =2T_{1}^{2}+2T_{2}^{2}.
\end{align*}
Using $\frac{1}{\sqrt{a}}-\frac{1}{\sqrt{b}}=\frac{b-a}{\sqrt{ab}(\sqrt{a}+\sqrt{b})}\le\frac{b-a}{\sqrt{a}b}$, $a+b\ge 2\sqrt{ab}$
and $s\le t$, we estimate
\begin{align*}
(\sigma^{-1}+n)T_{1}^{2} & \le(\sigma^{-1}+n)\frac{n^2h_{t}^{2}(h_{t}^{2}-h_{s}^{2})^{2}}{(nh_{t}^{2}+n^{-1})^{2}(nh_{s}^{2}+n^{-1})}\Big(\int|\psi^\dagger(z)|\d z\Big)^{2}\\
 & \lesssim(\sigma^{-1}+n)\frac{n^{2}\sigma^{6}(t^{2}-s^{2})^{2}}{(n\sigma^{2}t^{2}+n^{-1})^{2}(n\sigma^{2}s^{2}+n^{-1})}\\
 & \lesssim \frac{n^{2}\sigma^{5}(t-s)^{2}(t+s)^{2}}{n^{2}\sigma^{5}t^{3}s^{2}}
 +\frac{n^{3}\sigma^{6}(t-s)^{2}(t+s)^{2}}{n^{3}\sigma^{6}t^{4}s^{2}}
  \lesssim c^{-4}(t-s)^{2}
\end{align*}
 For $T_{2}$ the mean value theorem yields
\begin{align*}
(\sigma^{-1}+n)T_{2}^{2} & \le\frac{\sigma^{-1}+n}{nh_{s}^{2}+n^{-1}}\sup_{r\in[s,t]}\Big(\int\Big|\frac{z}{h_{t}}-\frac{z}{h_{s}}\Big|\big|(\psi^\dagger)'\big(\frac{z}{h_{r}}\big)\big|\d z\Big)^{2}\\
 & =\frac{\sigma^{-1}+n}{nh_{s}^{2}+n^{-1}}\sup_{r\in[s,t]}h_{r}^{2}\Big(\frac{h_{r}}{h_{s}}-\frac{h_{r}}{h_{t}}\Big)^{2}\Big(\int|z|\big|(\psi^\dagger)'\big(z\big)\big|\d z\Big)^{2}\\
 & \lesssim\frac{\sigma+n\sigma^{2}}{n\sigma^{2}s^2+n^{-1}}\sup_{r\in[s,t]}r^{2}\Big(\frac{r}{s}-\frac{r}{t}\Big)^{2}\\
 & \lesssim\frac{t^{4}}{s^2}\Big(\frac{t-s}{ts}\Big)^{2}
  \le c^{-4}(t-s)^{2}.
\end{align*}
Similarly we proceed with the $L^{2}$-norm in (\ref{eq:VarIncr}):
\begin{align*}
\|\Delta_{s,t}\|_{L^{2}}^{2} & \le2\Big(\frac{1}{\sqrt{nh_{t}^{2}+n^{-1}}}-\frac{1}{\sqrt{nh_{s}^{2}+n^{-1}}}\Big)^{2}\int|\psi^\dagger(z/h_{t})|^{2}\d z\\
 & \qquad+\frac{2}{nh_{s}^{2}+n^{-1}}\int|\psi^\dagger(z/h_{t})-\psi^\dagger(z/h_{s})|^{2}\d z
  =:2S_{1}^{2}+2S_{2}^{2}.
\end{align*}
with
\begin{align*}
(1+(n\sigma)^{-1})S_{1}^{2} & \lesssim\Big(1+\frac{1}{n\sigma}\Big)h_{t}\Big(\frac{1}{\sqrt{nh_{t}^{2}+n^{-1}}}-\frac{1}{\sqrt{nh_{s}^{2}+n^{-1}}}\Big)^{2}\\
 & \le\Big(1+\frac{1}{n\sigma}\Big)\frac{n^{2}h_t(h_{t}^{2}-h_{s}^{2})^{2}}{(nh_{t}^{2}+n^{-1})^2(nh_{s}^{2}+n^{-1})}\\
 & \lesssim\Big(1+\frac{1}{n\sigma}\Big)\frac{n^{2}\sigma^4(t+s)^{2}(t-s)^{2}}{(n\sigma^2{t}^{2}+n^{-1})(n\sigma^2{s}^{2}+n^{-1})}\\
 & \lesssim\frac{n^{2}\sigma^{4}(t+s)^{2}}{n^{2}\sigma^{4}t^{2}s^{2}}(t-s)^{2}+\frac{n\sigma^{3}(t+s)^{2}}{n\sigma^{3}t^{2}s^2}(t-s)^{2}
  \lesssim c^{-2}(t-s)^{2}
\end{align*}
and
\begin{align*}
(1+(n\sigma)^{-1})S_{2}^{2} & \lesssim\frac{1+(n\sigma)^{-1}}{nh_{s}^{2}+n^{-1}}\sup_{r\in[s,t]}\int\Big|\frac{z}{h_{t}}-\frac{z}{h_{s}}\Big|^{2}\big|(\psi^\dagger)'\big(\frac{z}{h_{r}}\big)\big|^{2}\d z\\
 & \lesssim\frac{1+(n\sigma)^{-1}}{nh_{s}^{2}+n^{-1}}\sup_{r\in[s,t]}h_{r}\Big(\frac{h_{r}}{h_{s}}-\frac{h_{r}}{h_{t}}\Big)^{2}
  \lesssim\frac{\sigma+n^{-1}}{n\sigma^{2}s^{2}+n^{-1}}\frac{t^3}{t^2s^2}(t-s)^{2}\lesssim c^{-3}(t-s)^{2}.
\end{align*}
These calculations verify
\[
\E\big[(V_{t}-V_{s})^{2}\big]=\Var\Big(\frac{1}{n\lambda\mu}\sum_{i,j}\big(\Delta_{s,t}(Y_{j}-X_{i})\Big)\lesssim(t-s)^{2}.
\]
Therefore, we may apply the Kolmogorov-Chentsov criterion and $(V_{t})$ has an $\alpha$-H\"older regular modification for
any $\alpha\in(0,1/2)$ implying tightness. 
\end{proof}

%{\color{purple}
\begin{proof}[Proof of Theorem \ref{thm:sigma}]
  We distinguish different regimes for $\sigma$:
  \begin{enumerate}
   \item If $\sigma\ge2n^{-2/3}/\log n$, then for $C>0$
   \begin{align*}
      \P\Big(\Big|\frac{\hat\sigma}{\sigma}-1\Big|\ge C\frac{(\log n)^2}{\sqrt{n\sigma}}\Big)
      \le&\P\Big(\Big|\frac{\hat\sigma^{(1)}}{\sigma}-1\Big|\ge C\frac{(\log n)^2}{\sqrt{n\sigma}}\Big)
      +\P(\hat T<\kappa_n)+\P\Big(\hat\sigma^{(1)}\le \frac{n^{-2/3}}{\log n}\Big)\\
      \le&\P\Big(\Big|\frac{\hat\sigma^{(1)}}{\sigma}-1\Big|\ge C\frac{(\log n)^2}{\sqrt{n\sigma}}\Big)+\P(\hat T<\kappa_n)+\P\big(\hat\sigma^{(1)}\le \frac\sigma 2\big)\\
      \le&\P\Big(\Big|\frac{\hat\sigma^{(1)}}{\sigma}-1\Big|\ge C\frac{(\log n)^2}{\sqrt{n\sigma}}\Big)+\P(\hat T<\kappa_n)+\P\Big(\frac{\sigma-\hat\sigma^{(1)}}\sigma \ge \frac1 2\Big),
   \end{align*}
   where the first and the last term are bounded by Proposition~\ref{prop:estSigma} while the middle term is $o(1)$ due to \eqref{eq:T}.
   \item If $\frac12n^{-2/3}/\log n \le \sigma< 2n^{-2/3}/\log n$, then $\sqrt{n\sigma^2+n^{-1}}$ is of order $(\log n)^{-1}n^{-1/6}$ and $(n\sigma)^{-1/2}$ is of order $(\log n)^{1/2}n^{-1/6}$. Therefore, 
      \begin{align*}
        &\P\Big(\Big|\frac{\hat\sigma}{\sigma}-1\Big|\ge C(\log n)^2\big(\sqrt{n\sigma^{2}+n^{-1}}\wedge\frac{1}{\sqrt{n\sigma}}\big)\Big)\\
        &\qquad\le\P\Big(\Big|\frac{\hat\sigma^{(1)}}{\sigma}-1\Big|\ge C'\frac{(\log n)^{1/2}}{\sqrt{n\sigma}}\Big)+\P\Big(\Big|\frac{\hat\sigma^{(2)}}{\sigma}-1\Big|\ge C''(\log n)^2\sqrt{n\sigma^2+n^{-1}}\Big).
      \end{align*}
      These terms can be bounded by Propositions~\ref{prop:estSigma} and \ref{prop:sigma2} noting that $n\sigma^{3/2}\to 0$ in this case.
   \item If $\sigma\le\frac12n^{-2/3}/\log n$, then 
   \begin{align}
      \P\Big(\Big|\frac{\hat\sigma}{\sigma}-1\Big|\ge C(\log n)^2\sqrt{n\sigma^{2}+n^{-1}}\Big)
      \le&\P\Big(\Big|\frac{\hat\sigma^{(2)}}{\sigma}-1\Big|\ge C(\log n)^2\sqrt{n\sigma^{2}+n^{-1}}\Big)\notag\\
      &\qquad +\P(\hat T\ge \kappa_n, \hat\sigma^{(1)}> \frac{n^{-2/3}}{\log n}\Big).\label{eq:finalproof}
    \end{align}
    The first term can again be treated with Proposition~\ref{prop:sigma2}. For $\sigma >\frac{\kappa_n^{1/2}}n$ the second term is bounded by
    \[
      \P\Big(\hat \sigma^{(1)}>\frac{n^{-2/3}}{\log n}\Big)
      \le\P\big(\hat \sigma^{(1)}>2\sigma \big)
      \le \P\Big(\frac{\hat\sigma^{(1)}}{\sigma}-1>1\Big),
    \]
    which converges to zero by Proposition~\ref{prop:estSigma} due to $\sigma n>\kappa_n^{1/2}\to\infty$. If on the other hand $\sigma\le\frac{\kappa _n^{1/2}}{n}$, then the probability in \eqref{eq:finalproof} can be bounded by
    \[
      \P(\hat T\ge \kappa_n)\le\frac{\Var(\hat T)}{(\kappa_n-\E[\hat T])^2}\lesssim \frac{n\sigma}{(\kappa_n-n\sigma\lambda \mu I_f)^2}\lesssim \kappa_n^{-3/2}\to \infty
    \]
    thanks to Lemma~\ref{lem:EstNSigma}.\qedhere
  \end{enumerate}
\end{proof}
%}

\subsection{Proofs for the plug-in estimators}

\begin{proof}[Proof of Theorem \ref{prop:plugIn}]
\emph{(i)} We analyse the deconvolution estimator in four steps.

\emph{Step 1: Prefactor. }Defining
\[
  \bar f_{\hat \sigma}^{(1)}(x_0):=\frac{1}{\hat \sigma\hat h_1^2 \lambda \mu n}\sum_{j}K'\Big(\frac{z_0}{\hat h_{1}}-\frac{Y_{j}}{\hat\sigma \hat h_{1}}\Big)\Big(\frac{1}{n\lambda}\sum_iK\Big(\frac{\hat \sigma z_0}{9}-\frac{Y_{j}-X_i}{9}\Big)\Big),
\]
we have
\begin{align*}
\tilde{f}_{\hat\sigma}^{(1)}(z_{0})-f(z_{0}) & =\frac{\lambda\mu n}{|\mathcal{Y}|}\frac{\lambda n}{|\mathcal X|}\big(\bar f_{\hat \sigma}^{(1)}(z_0)-f(z_{0})\big)+\Big(\frac{\lambda\mu n}{|\mathcal{Y}|}\frac{\lambda n}{|\mathcal X|}-1\Big)f(z_{0}).
\end{align*}
For $\tau_n\to\infty$ the event
\begin{equation*}
\Lambda:=\Big\{\Big|\frac{|\mathcal{Y}|}{\lambda\mu n}-1\Big|\le\frac{\tau_n}{\sqrt n}\Big\}\cup\Big\{\Big|\frac{|\mathcal{X}|}{\lambda n}-1\Big|\le\frac{\tau_n}{\sqrt n}\Big\}
\end{equation*}
satisfies
\begin{align*}
\P(\Lambda^{c}) & \le\P\Big(\Big|\frac{|\mathcal{Y}|}{\lambda\mu n}-1\Big|>\frac{\tau_n}{\sqrt{n}}\Big)+\P\Big(\Big|\frac{|\mathcal{X}|}{\lambda n}-1\Big|>\frac{\tau_n}{\sqrt{n}}\Big)\\
 & \le\frac{n}{\tau_n^{2}(\lambda\mu n)^{2}}\E\big[\Var\big(|\mathcal Y|\,\big|\,|\mathcal X|\big)\big]+\frac{n}{\tau_n^{2}(\lambda n)^{2}}\Var(|\mathcal X|)\\
 & =\frac{1}{\tau_n^{2}(\lambda\mu)^2n}\E[\mu|\mathcal X|]+\frac{1}{\tau_n^{2}\lambda}\\
 & =\frac{1}{\tau_n^{2}\lambda\mu}+\frac{1}{\tau_n^{2}\lambda}\to0,
\end{align*}
due to $|\mathcal{Y}|\big||\mathcal X|\sim\mathrm{Poiss}(\mu |\mathcal X|)$ and $|\mathcal{X}|\sim\mathrm{Poiss}(\lambda n)$. On $\Lambda$
we have for $\tau_n/\sqrt n\le 1/2$ that
\begin{align*}
  \frac{\lambda\mu n}{|\mathcal{Y}|}\frac{\lambda n}{|\mathcal X|}-1&=\Big(\frac{\lambda\mu n}{|\mathcal{Y}|}-1\Big)\frac{\lambda n}{|\mathcal X|}+\frac{\lambda n}{|\mathcal X|}-1\\
  &=\Big(\frac{1-|\mathcal{Y}|/(\lambda\mu n)}{1-(1-|\mathcal{Y}|/\lambda\mu n)}\Big)\frac{\lambda n}{|\mathcal X|}+\frac{1-|\mathcal{X}|/(\lambda n)}{1-(1-|\mathcal{X}|/\lambda n)}
  \le 6\tau_n/\sqrt n.
\end{align*}
Since $r_n$ is always slower than $n^{-1/2}$, we conclude
\[
r_{n}^{-1}\big|\tilde{f}_{\hat \sigma}^{(1)}(z_{0})-f(z_{0})\big|=O_{\P}\Big(r_{n}^{-1}\big|\bar{f}^{(1)}_{\hat \sigma}(z_{0})-f(z_{0})\big|\Big)+o_{\P}(1).
\]

\emph{Step 2: From $\hat{\sigma}$ to $\sigma$. }Consider the event
\[
\Sigma:=\big\{\hat{\sigma}\in[\sigma(1-\epsilon_n),\sigma(1+\epsilon_n)]\big\},\qquad\epsilon_n=\frac{\log n}{\sqrt{\sigma n}}
\]
satisfying
\begin{align*}
\P(\Sigma^{c})=\P(|\hat{\sigma}-\sigma|>\epsilon_n\sigma) & =\P\big(|\frac{\hat{\sigma}}{\sigma}-1|>\frac{\log  n}{\sqrt{\sigma n}}\big)\to0
\end{align*}
due to Theorem~\ref{thm:sigma}. Writing $\bar{h}_1=(n\bar{\sigma})^{-1/(2s+3)}$
for any $\bar{\sigma}>0$, we have on $\Sigma$ 
\begin{align}
 & r_{n}^{-1}\big|\bar{f}_{\hat \sigma}^{(1)}(z_{0})-f(z_{0})\big|\le\sup_{\bar{\sigma}:|\bar{\sigma}-\sigma|\le\epsilon_n\sigma}r_{n}^{-1}\Big|\bar f_{\bar\sigma}^{(1)}(z_{0})-\E[\bar f_{\bar\sigma}^{(1)}(z_{0})]-\big(f(z_{0})-\E[\bar f_{\bar\sigma}^{(1)}(z_{0})]\big)\Big|\nonumber \\
 & \qquad\le\sup_{\bar{\sigma}:|\bar{\sigma}-\sigma|\le\epsilon_n\sigma}\frac{\sqrt{\bar{\sigma}n\bar{h}_1^{3}}}{r_{n}\sqrt{\bar{\sigma}n\bar{h}_1^{3}}}\Big|\bar f_{\bar\sigma}^{(1)}(z_{0})-\E[\bar f_{\bar\sigma}^{(1)}(z_{0})]\Big|+\sup_{\bar{\sigma}:|\bar{\sigma}-\sigma|\le\epsilon_n\sigma}\frac{\bar{h}_1^{-s}}{r_{n}\bar{h}_1^{-s}}\Big|\big(f(z_{0})-\E[\bar f_{\bar\sigma}^{(1)}(z_{0})]\big)\Big|\nonumber \\
 & \qquad\lesssim\sup_{\bar{\sigma}:|\bar{\sigma}-\sigma|\le\epsilon_n\sigma}\sqrt{\bar{\sigma}n\bar{h}_1^{3}}\Big|\bar f_{\bar\sigma}^{(1)}(z_{0})-\E[\bar f_{\bar\sigma}^{(1)}(z_{0})]\Big|+\sup_{\bar{\sigma}:|\bar{\sigma}-\sigma|\le\epsilon_n\sigma}\bar{h}_1^{-s}\Big|\big(f(z_{0})-\E[\bar f_{\bar\sigma}^{(1)}(z_{0})]\big)\Big|\label{eq:plugInDecomp}
\end{align}
using in the last step that the minimax rate satisfies
$
r_{n}=(h_{1})^{s}=\big(n\sigma(h_{1})^{3}\big)^{-1/2}
$
and thus
\begin{align*}
\frac{1}{r_{n}^{2}\bar{\sigma}n\bar{h}_1^{3}} & =\frac{\sigma n(h_{1})^{3}}{\bar{\sigma}n\bar{h}_1^{3}}=\Big(\frac{\sigma}{\bar{\sigma}}\Big)^{2s/(2s+3)}\le(1-\epsilon_n)^{-2s/(2s+3)},\\
\frac{1}{r_{n}\bar{h}_1^{-s}} & =\Big(\frac{\bar{h}_1}{h_{1}}\Big)^{s}=\Big(\frac{\bar{\sigma}}{\sigma}\Big)^{s/(2s+3)}\le(1+\epsilon_n)^{s/(2s+3)}.
\end{align*}
Subsequently, we will bound both terms in \eqref{eq:plugInDecomp} separately. To this end, we proceed similarly to the proof of Proposition~\ref{thm:rates}(i). To incorporate $\bar\sigma$ we set
\begin{equation}\label{eq:psiBar}
  \bar \psi_1:=\frac{1}{\bar\sigma\bar h_1^2}K'\Big(\frac{z_0}{\bar h_1}-\frac{\cdot}{\bar\sigma\bar h_1}\Big)\qquad\text{and}\qquad
  \bar\psi_2:=\frac{1}{\bar h_2}K\Big(\frac{z_0}{\bar h_2}-\frac{\cdot}{\bar h_2}\Big).
\end{equation}
where $\bar h_2:=9/\bar\sigma$.

\emph{Step 3: Bias.} The analog to decomposition \eqref{eq:exp} leads to 
$$
\bar{\mathcal U}_\sigma(f \ast p) = \int_{\R}\bar\psi_1(y)(\bar\psi_2\ast\eins_{[0,1/\bar \sigma]})(y/\bar\sigma)(f_{\sigma}\ast p)(y)\d y,\quad
\bar{\mathcal V}_\sigma(f)=  \int_{\R}(\bar\psi_1\ast\eins_{[-1,0]})(\sigma z)\bar\psi_2(z)f(z)\d z
$$
where $\sigma$ is the true data-generating parameter. Along the lines of the proof of Propositions~\ref{prop:bias1}(i) and \ref{prop:bias2}(ii) we obtain
\[
  \E[\bar{\mathcal U}_\sigma(f\ast p)]=\frac{1}{\bar \sigma \bar h_2}f_{\sigma/\bar\sigma}(z_0)+O\big(\frac{h_1^s}{\sigma h_2}\big)\qquad\text{and}\qquad
  \E[\bar{\mathcal V}_\sigma(f)]\lesssim \bar h_1^{-1}.
\]
Note that 
\[f_{\sigma/\bar\sigma}(z_0)-f(z_0)=\frac{\bar \sigma}{\sigma}\big(f(\bar\sigma z_0/\sigma)-f(z_0)\big)+\big(\frac{\bar\sigma}{\sigma}-1\big)f(z_0)\lesssim \Big(\frac{\log(n\sigma)}{n\sigma}\Big)^{(1\wedge s)/2}.\]
Therefore, we obtain the following modification of \eqref{eq:biasf1}:
\begin{align*}
\E\Big[\bar{f}^{(1)}_{\bar\sigma}(z_0)\Big] & =\bar \sigma \bar h_2\,\bar{\mathcal U}_{\bar\sigma}(f\ast p)+\frac{\bar h_2}{n\lambda}\bar{\mathcal V}_{\bar \sigma}(f)\notag\\
&=f(z_0)+O\Big(h_{1}^{s}+\Big(\frac{\log(n\sigma)}{n\sigma}\Big)^{(1\wedge s)/2}+\frac{h_{2}}{nh_1}\Big).
\end{align*}
We conclude
\[
\sup_{\bar{\sigma}:|\bar{\sigma}-\sigma|\le\epsilon_n\sigma}\bar{h}_1^{-s}\Big|\big(f(z_{0})-\E[\bar f_{\bar\sigma}^{(1)}(z_{0})]\big)\Big|\lesssim 1+\Big(\frac{\log(\sigma n)}{\sigma n}\Big)^{(1\wedge s)/2}r_{n}^{-1}+\frac{r_n^{-1}}{n\sigma h_1}\lesssim 1.
\]

\emph{Step 4: Stochastic error term.} Recall that $\varepsilon_n = \frac{\log n}{\sqrt{\sigma n}}$ and define 
\[
\sigma_{t}:=\sigma(1-\epsilon_n+2\epsilon_n t),\qquad h_{t}:=(n\sigma_{t})^{-1/(2s+3)},\qquad t\in[0,1],
\]
as well as
\begin{align*}
V_{t} & :=\sqrt{\sigma_{t}nh_{t}^{3}}\big(\bar{f}_{\sigma_{t}}^{(1)}(z_{0})-\E[\bar{f}_{\sigma_{t}}^{(1)}(z_{0})]\big).
\end{align*}
We thus have to prove tightness of the process $(V_{t})_{t\in[0,1]}$
that is $
\sup_{t\in[0,1]}|V_{t}|=O_{\P}(1)$.
As in the proof of Proposition~\ref{prop:sigma2}, we apply the Kolmogorov-Chentsov criterion. Writing
\begin{align*}
V_{t}-V_{s}=&\frac{1}{(\lambda\mu\sqrt{n})(\lambda n)}\sum_{i,j}\big(\Delta^{(1)}_{s,t}(X_i,Y_{j})-\E[\Delta^{(1)}_{s,t}(X_i,Y_{j})]\big)\\
&+\frac{1}{(\lambda\mu\sqrt{n})(\lambda n)}\sum_{i,j}\big(\Delta^{(2)}_{s,t}(X_i,Y_{j})-\E[\Delta^{(2)}_{s,t}(X_i,Y_{j})]\big)
\end{align*}
with
\begin{align*}
  \Delta_{s,t}^{(1)}(x,y)=&\Big(\underbrace{\frac{1}{\sqrt{\sigma_th_t}}K'\Big(\frac{z_0}{h_{t}}-\frac{y}{\sigma_t h_{t}}\Big)-\frac{1}{\sqrt{\sigma_sh_s}}K'\Big(\frac{z_0}{h_{s}}-\frac{y}{\sigma_s h_{s}}\Big)}_{=:\Delta_{s,t}^{(1,1)}(y)}\Big)\underbrace{K\Big(\frac{\sigma_t z_0}{9}-\frac{y-x}{9}\Big)}_{=:\Delta^{(1,2)}_{t}((y-x)/\sigma)},\\
    \Delta_{s,t}^{(2)}(x,y)=&\underbrace{\frac{1}{\sqrt{\sigma_sh_s}}K'\Big(\frac{z_0}{h_{s}}-\frac{y}{\sigma_s h_{s}}\Big)}_{=:\Delta^{(2,1)}_s(y)}\Big(\underbrace{K\Big(\frac{\sigma_t z_0}{9}-\frac{y-x}{9}\Big)-K\Big(\frac{\sigma_s z_0}{9}-\frac{y-x}{9}\Big)}_{=:\Delta^{(2,2)}_{s,t}((y-x)/\sigma)}\Big).
\end{align*}
Proposition~\ref{prop:variance} yields
\begin{align}
  \E[(V_t-V_s)^2]\lesssim&\Big((\frac\sigma n+\sigma^2)\big\|\Delta_t^{(1,2)}\big\|^2_{L^1}+(\frac\sigma n+\frac{1}{n^2})\big\|\Delta_{t}^{(1,2)}\big\|^2_{L^2}\Big)
\big\|\Delta_{s,t}^{(1,1)}\big\|^2_{L^2}\notag\\
  &+\Big((\frac\sigma n+\sigma^2)\big\|\Delta_{s,t}^{(2,2)}\big\|^2_{L^1}+(\frac\sigma n+\frac{1}{n^2})\big\|\Delta_{s,t}^{(2,2)}\big\|^2_{L^2}\Big)
\big\|\Delta_{s}^{(2,1)}\big\|^2_{L^2}\label{eq:DiffV}\\
\lesssim&\Big(\frac1{n\sigma}+1+\frac1 n+\frac{1}{n^2\sigma}\Big)\big\|\Delta_{s,t}^{(1,1)}\big\|^2_{L^2}\notag\\
  &+\Big((\frac\sigma n+\sigma^2)\big\|\Delta_{s,t}^{(2,2)}\big\|^2_{L^1}+(\frac\sigma n+\frac{1}{n^2})\big\|\Delta_{s,t}^{(2,2)}\big\|^2_{L^2}\Big)\nonumber\label{eq:DiffV}
\end{align}
We have to bound the norms in the previous line. We have
\begin{align*}
\|\Delta_{s,t}^{(1,1)}\|_{L^{2}}^{2}& \le3\Big(\frac{1}{\sqrt{\sigma_{t}h_{t}}}-\frac{1}{\sqrt{\sigma_{s}h_{s}}}\Big)^{2}\int K'\Big(\frac{z_{0}}{h_{t}}-\frac{y}{\sigma_{t}h_{t}}\Big)\Big)^{2}\d y\\
 & \qquad+\frac{3}{\sigma_{s}h_{s}}\int\Big(K'\Big(\frac{z_{0}}{h_{t}}-\frac{y}{\sigma_{t}h_{t}}\Big)-K'\Big(\frac{z_{0}}{h_{s}}-\frac{y}{\sigma_{t}h_{s}}\Big)\Big)^{2}\d y\\
 &\qquad+\frac{3}{\sigma_{s}h_{s}}\int\Big(K'\Big(\frac{z_{0}}{h_{s}}-\frac{y}{\sigma_{t}h_{s}}\Big)-K'\Big(\frac{z_{0}}{h_{s}}-\frac{y}{\sigma_{s}h_{s}}\Big)\Big)^{2}\d y\\
 & =:3T_{1}+3T_{2}+3T_3.
\end{align*}
Using $K'\in L^{2}$, we have 
\begin{align*}
T_{1} & \lesssim\Big(1-\frac{\sqrt{\sigma_{t}h_{t}}}{\sqrt{\sigma_{s}h_{s}}}\Big)^{2}=\Big(1-\Big(\frac{\sigma_{t}}{\sigma_{s}}\Big)^{(s+1)/(2s+3)}\Big)^{2}\lesssim\big(\frac{\sigma_{s}-\sigma_{t}}{\sigma_{s}}\big)^{2}\lesssim(t-s)^{2}.
\end{align*}
Moreover, for some intermediate point $r\in[s,t]$ we have
\begin{align*}
  T_2=\frac{1}{h_{s}}\int\Big(K'\Big(\frac{y}{h_{t}}\Big)-K'\Big(\frac{y}{h_{s}}\Big)\Big)^{2}\d y
  &=\frac{1}{h_{s}}\int\Big(\frac y{h_t}-\frac y{h_s}\Big)^2K''\Big(\frac{y}{h_{r}}\Big)^{2}\d y\\
  &\le\sup_{r\in[s,t]}\frac{h_r}{h_{s}}\Big(\frac {h_r}{h_t}-\frac {h_r}{h_s}\Big)^2\int y^2K''(y)^{2}\d y\\
  &\lesssim (t-s)^2,
  \end{align*}
  and
  \begin{align*}
  T_3=\frac{1}{\sigma_s}\int\Big(K'\Big(\frac{z_{0}}{h_{s}}-\frac{y}{\sigma_{t}}\Big)-K'\Big(\frac{z_{0}}{h_{s}}-\frac{y}{\sigma_{s}}\Big)\Big)^{2}\d y
  &=\frac{1}{\sigma_{s}}\int\Big(\frac y{\sigma_t}-\frac y{\sigma_s}\Big)^2K''\Big(\frac{z_0}{h_s}-\frac{y}{\sigma_{r}}\Big)^{2}\d y\\
  &\le\sup_{r\in[s,t]}\frac{\sigma_r}{\sigma_{s}}\Big(\frac {\sigma_r}{\sigma_t}-\frac {\sigma_r}{\sigma_s}\Big)^2\int y^2K''(\frac{z_0}{h_s}-y)^{2}\d y\\
  &\lesssim \frac{(\sigma_t-\sigma_s)^2}{\sigma_s^2h_s^2}\\
  &\lesssim\frac{\epsilon_n^2(t-s)^2}{h_s^2}
  \lesssim(t-s)^2
\end{align*}
because $\epsilon_n^2 h_s^{-2}\to0$.
For $\Delta_{s,t}^{(2,2)}$ we have similarly
\begin{align*}
  \big\|\Delta_{s,t}^{(2,2)}\big\|_{L^1}&=\int\Big|K\Big(\frac{\sigma_t z_0}{9}-\frac{\sigma z}{9}\Big)-K\Big(\frac{\sigma_s z_0}{9}-\frac{\sigma z}{9}\Big)\Big|\d z\\
  &=\frac1{81}\int\big|\sigma_tz_0-\sigma_sz_0\big|\Big|K\Big(\frac{\sigma_r z_0}{9}-\frac{\sigma z}{9}\Big)\Big|\d z
  \lesssim\frac{|\sigma_t-\sigma_s|}{\sigma}\lesssim|t-s|,
  \end{align*}
  and
  \begin{align*}
  \big\|\Delta_{s,t}^{(2,2)}\big\|_{L^2}^2&=\int\Big(K\Big(\frac{\sigma_t z_0}{9}-\frac{\sigma z}{9}\Big)-K\Big(\frac{\sigma_s z_0}{9}-\frac{\sigma z}{9}\Big)\Big)^{2}\d z\\
  &=\frac1{81}\int\big(\sigma_tz_0-\sigma_sz_0\Big)^2K\Big(\frac{\sigma_r z_0}{9}-\frac{\sigma z}{9}\Big)\Big)\Big)^{2}\d z
  \lesssim\frac{(\sigma_t-\sigma_s)^2}{\sigma}\lesssim\sigma(t-s)^2
\end{align*}
It follows that 
$\E\big[\big(V_{t}-V_{s}\big)^{2}\big]\lesssim(t-s)^{2}$
and $(V_{t})$ has an $\alpha$-H\"older regular modification for
any $\alpha\in(0,1/2)$ implying tightness. We have shown (i). 
\medskip

\emph{(ii)} We verify the convergence rate for the direct estimator similiarly to (i).

\emph{Step 1: Reduction.}  We define
\begin{align*}
\bar{f}_{\hat \sigma}^{(2)}(z_0) =\frac{1}{ \hat h_{2} \lambda\mu n}\sum_{i,j} 2K'\big(2(\hat\sigma z_0-Y_{j})\big) K\Big(\frac{z_0}{\hat h_{2}}-\frac{Y_{j}-X_{i}}{\hat \sigma \hat h_{2}}\Big)-\hat \sigma |\mathcal X|.
\end{align*}
Exactly as in Step~1 of the proof of (i), we see that
%Proposition~\ref{prop:plugIn1} we see that
\[
r_{n}^{-1}\big|\tilde{f}_{\hat \sigma}^{(2)}(z_{0})-f(z_{0})\big|=O_{\P}\Big(r_{n}^{-1}\big|\bar{f}^{(2)}_{\hat \sigma}(z_{0})-f(z_{0})\big|\Big)+o_{\P}(1).
\]
For $\sigma=o(n^{-2/3})$ we have $\P(\Sigma^c)\to0$ for $\Sigma:=\big\{\hat{\sigma}\in[\sigma(1-\epsilon_n),\sigma(1+\epsilon_n)]\big\}$ with
\[
\epsilon_n=\kappa_n\Big(\sqrt{n\sigma^2+n^{-1}} \wedge \frac{1}{\sqrt{n \sigma}}\Big),\qquad \kappa_n=\sqrt{\log n}
\]
due to Theorem~\ref{thm:sigma}. Writing $\bar{h}_2=(n\wedge \bar{\sigma}^{-1})^{-1/(2s+1)}$ for $\bar \sigma\in[\sigma(1-\epsilon_n),\sigma(1+\epsilon_n)]$ we note that 
\begin{align*}
  r_n \bar h_2^{-s}\gtrsim\big((n\wedge \sigma^{-1})^{-s(2s+1)}+\sqrt{n\sigma^2}\big)h_2^{-s}\ge& 1\qquad\text{and}\\
  r_n \big((n\wedge \sigma^{-1})^{-1/2}\bar h_2^{-1/2}+\sqrt{ n\sigma^2}\big)^{-1}\gtrsim&1.
\end{align*}
Hence, as in Step~2 of the proof of (i)
%Proposition~\ref{prop:plugIn1} 
we have on $\Sigma$ 
\begin{align}
 r_{n}^{-1}\big|\bar{f}_{\hat \sigma}^{(2)}(z_{0})-f(z_{0})\big|
 \lesssim& \sup_{\bar{\sigma}:|\bar{\sigma}-\sigma|\le\epsilon_n\sigma}\sqrt{\min(n\bar h_2,{\sigma}^{-1}\bar{h}_2, (n\sigma^2)^{-1})}\Big|\bar f_{\bar\sigma}^{(2)}(z_{0})-\E[\bar f_{\bar\sigma}^{(2)}(z_{0})]\Big|\nonumber\\
 & +\sup_{\bar{\sigma}:|\bar{\sigma}-\sigma|\le\epsilon_n\sigma}\bar{h}_2^{-s}\Big|\big(f(z_{0})-\E[\bar f_{\bar\sigma}^{(1)}(z_{0})]\big)\Big|\label{eq:plugInInteraction}
\end{align}

\emph{Step 2: Bias.} We will use the notation $\bar\psi_1$ and $\bar\psi_2$ from \eqref{eq:psiBar} and the resulting $\bar{\mathcal U}_\sigma(f\ast p)$ and $\bar{\mathcal V}_\sigma(f)$. Note that $\bar h_1=1/(2\bar \sigma)$ in this case. With minor modifications in the proofs of Propositions~\ref{prop:bias1}(ii) and \ref{prop:bias2}(i) we obtain
\[
  \E[\bar{\mathcal U}_\sigma(f\ast p)]=\frac{1}{\bar h_1}\qquad\text{and}\qquad
  \E[\bar{\mathcal V}_\sigma(f)]= \bar h_1^{-1}f(z_0)+O(h_1^{-1}h_2^s).
\]
Therefore, we obtain the following modification of \eqref{eq:biasf2}:
\begin{align*}
\E\Big[\bar{f}^{(2)}_{\bar\sigma}(z_0)\Big] & =\bar h_1\bar{\mathcal V}_\sigma(f)+\bar \sigma n\lambda h_1\,\bar{\mathcal U}_\sigma(f\ast p)-\bar\sigma\E[|\mathcal X|]
=f(z_0)+O\big(h_{2}^{s}\big).
\end{align*}
We conclude
\[
\sup_{\bar{\sigma}:|\bar{\sigma}-\sigma|\le\epsilon_n\sigma}\bar{h}_2^{-s}\Big|\big(f(z_{0})-\E[\bar f_{\bar\sigma}^{(2)}(z_{0})]\big)\Big|\lesssim 1.
\]

\emph{Step 3: Stochastic error term.} First note that due to the bias correction we have the additional stochastic error term: 
\[
  \sup_{\bar{\sigma}:|\bar{\sigma}-\sigma|\le\epsilon_n\sigma}\sqrt{\min(n\bar h_2,{\sigma}^{-1}\bar{h}_2, (n\sigma^2)^{-1})}\Big|\bar \sigma|\mathcal X|-\bar \sigma n\lambda\Big|\le\frac1{\sigma n^{1/2}} 2\sigma \big||\mathcal X|-n\lambda\big|=O_\P(1),
\]
where we used $|\mathcal X|\sim\mathrm{Poiss}(n\lambda)$. To bound the stochastic error due to the terms involving $\bar \psi_1$ and $\bar\psi_2$, we use again the Kolmogorov-Chentsov criterion for the process
\begin{align*}
V_{t} & :=\varpi_t\Big(\bar{f}_{\sigma_{t}}^{(1)}(z_{0})-\E[\bar{f}_{\sigma_{t}}^{(1)}(z_{0})]+\sigma_t(|\mathcal X|-n\lambda)\Big)
\end{align*}
with
$$\varpi_t := \sqrt{\min(n h_t,{\sigma}^{-1}h_t, (n\sigma^2)^{-1})},$$
\[
\sigma_{t}:=\sigma(1-\epsilon_n+2\epsilon_n t),\qquad h_{t}:=(n\wedge \sigma_t^{-1})^{-1/(2s+1)},\qquad t\in[0,1].
\]
We decompose
\begin{align*}
V_{t}-V_{s}=&\frac{1}{\lambda\mu \sqrt n}\sum_{i,j}\big(\Delta^{(1)}_{s,t}(X_i,Y_{j})-\E[\Delta^{(1)}_{s,t}(X_i,Y_{j})]\big)\\
&+\frac{1}{\lambda\mu \sqrt{n}}\sum_{i,j}\big(\Delta^{(2)}_{s,t}(X_i,Y_{j})-\E[\Delta^{(2)}_{s,t}(X_i,Y_{j})]\big)
\end{align*}
with
\begin{align*}
    \Delta_{s,t}^{(1)}(x,y)=&\Big(\underbrace{2K'\big(2(\sigma_t z_0 -y)\big)-2K'\big(2(\sigma_s z_0-y)\big)}_{=:\Delta^{(1,1)}_{s,t}(y)}\Big)\underbrace{\frac{\varpi_t}{\sqrt{n}h_t}K\Big(\frac{z_0}{h_{t}}-\frac{y-x}{\sigma_t h_{t}}\Big)}_{=:\Delta^{(1,2)}_t((y-x)/\sigma)},\\
  \Delta_{s,t}^{(2)}(x,y)=&\underbrace{2K'\big(2(\sigma_s z_0-y)\big)}_{=:\Delta^{(2,1)}_{s}(y)} \Big(\underbrace{\frac{\varpi_t}{\sqrt{n}h_t}K\Big(\frac{z_0}{h_{t}}-\frac{y-x}{\sigma_t h_{t}}\Big)-\frac{\varpi_s}{\sqrt{n}h_s}K\Big(\frac{z_0}{h_{s}}-\frac{y-x}{\sigma_s h_{s}}\Big)}_{=:\Delta_{s,t}^{(2,2)}((y-x)/\sigma)}\Big).
\end{align*}
With these definitions, the bound \eqref{eq:DiffV} remains valid up to a factor $n^2$ (coming from the missing factor $\frac{1}{n}$ in $V_t-V_s$) and we obtain
\begin{align*}
  \E[(V_t-V_s)^2]& \lesssim\varpi_t\Big(\sigma+n\sigma^2+\frac\sigma{h_t}+\frac{1}{n h_t}\Big)
\big\|\Delta_{s,t}^{(1,1)}\big\|^2_{L^2}\\
  &\quad +(\sigma n+\sigma^2n^2)\big\|\Delta_{s,t}^{(2,2)}\big\|^2_{L^1}+(\sigma n+1)\big\|\Delta_{s,t}^{(2,2)}\big\|^2_{L^2}\\
&  \lesssim \big\|\Delta_{s,t}^{(1,1)}\big\|^2_{L^2}+(\sigma n+(\sigma n)^2)\big\|\Delta_{s,t}^{(2,2)}\big\|^2_{L^1}+(\sigma n+1)\big\|\Delta_{s,t}^{(2,2)}\big\|^2_{L^2}.
\end{align*}
Next, we have
\begin{align*}
  \|\Delta_{s,t}^{(1,1)}\|_{L^2}^2&=4\int\Big(K'\big(2(\sigma_t z_0 -y)\big)-2K'\big(2(\sigma_s z_0-y)\big)\Big)^2\d y\\
  &=8(\sigma_t-\sigma_s)^2 z_0\int K''\big(2(\xi -y)\big)^2\d y\lesssim \sigma (t-s)^2\lesssim (t-s)^2.
\end{align*}
Moreover, the term $ \|\Delta_{s,t}^{(2,2)}\|_{L^1} $ is bounded above by
\begin{align*}
& \Big|\frac{\varpi_t}{\sqrt{n}h_t}-\frac{\varpi_s}{\sqrt{n}h_s}\Big|\int \Big|K\Big(\frac{z_0}{h_{t}}-\frac{\sigma z}{\sigma_t h_{t}}\Big)\Big| \d z\\
 &+\frac{\varpi_s}{\sqrt{n}h_s}\int \Big|K\Big(\frac{z_0}{h_{t}}-\frac{\sigma z}{\sigma_th_t}\Big)-K\Big(\frac{z_0}{h_{s}}-\frac{\sigma z}{\sigma_t h_{s}}\Big)\Big| \d z \\
 &+\frac{\varpi_s}{\sqrt{n}h_s}\int \Big|K\Big(\frac{z_0}{h_{s}}-\frac{\sigma z}{\sigma_th_s}\Big)-K\Big(\frac{z_0}{h_{s}}-\frac{\sigma z}{\sigma_s h_{s}}\Big)\Big| \d z\\
 & \le \Big|\frac{\varpi_t}{\sqrt{n}h_t}-\frac{\varpi_s}{\sqrt{n}h_s}\Big|h_t +\frac{\varpi_s}{\sqrt{n}h_s}\sup_{r\in[s,t]}\frac{\sigma_t^2h_r^2}{\sigma^2}\big|\frac{\sigma}{\sigma_th_t}-\frac{\sigma}{\sigma_th_s}\Big| +\frac{\varpi_s}{\sqrt{n}h_s}\sup_{r\in[s,t]}\frac{\sigma_r^2h_s}{\sigma^2}\big|\frac{\sigma}{\sigma_th_s}-\frac{\sigma}{\sigma_sh_s}\Big|\\
 & \lesssim  \Big|\frac{\varpi_t}{\sqrt{n}h_t}-\frac{\varpi_s}{\sqrt{n}h_s}\Big|h_t +\frac{\varpi_s}{\sqrt{n}h_s}\big(h_s|t-s|+\frac{|\sigma_t-\sigma_s|}{\sigma}\big).
\end{align*}
Since $h_s$ and $h_s$ are of the same order in terms of $n$ and $\sigma$ both minima in the above difference are obtained at the same argument. Separate upper bounds in all three cases yield
\begin{align*}
  (\sigma n+\sigma^2n^2)\big\|\Delta_{s,t}^{(2,2)}\big\|^2_{L^1}
  & \lesssim  \sigma nh_t^2\big(h_s^{-1/2}-h_t^{-1/2}\big)^2\eins_{\sigma<n^{-(2s+1)/(2s+2)}}
  + h_t^2\big(h_t^{-1}-h_s^{-1}\big)^2\\
  &+  \Big(\frac{\sigma n}{h_s}\eins_{\sigma<n^{-(2s+1)/(2s+2)}}+\frac{1}{h_s^2}\eins_{\sigma\ge n^{-(2s+1)/(2s+2)}}\Big)\big(h_s^2+\epsilon_n^2\big)|t-s|^2\\
  &\lesssim  |t-s|^2,
\end{align*}
noting that $n\sigma h_t\le 1$ as well as $\frac{\sigma n\epsilon_n^2}{h_s}<\kappa_n^2 n^{(2-2s)/(2s+2)}\le 1$ for $\sigma<n^{-(2s+1)/(2s+2)}$ and $s > 1$ while 
$\frac{\epsilon_n}{h_s}= 
\kappa_n \min\Big(n^{1/2}\sigma^{\frac{2s}{2s+1}}, n^{-1/2}\sigma^{-\frac{2s+3}{4s+2}}\Big) \lesssim \kappa_n n^{(3-2s)/(12s+6)}\lesssim 1$ 
%for $\sigma^{(4s+1)/(2s+1)} > n\sigma^{3/2}\le 1$ 
for $\sigma\ge n^{-(2s+1)/(2s+2)}$ and $s > 3/2$.
Similarly, $\|\Delta_{s,t}^{(2,2)}\|_{L^2}^2$ is less than
\begin{align*}
& \Big|\frac{\varpi_t}{\sqrt{n}h_t}-\frac{\varpi_s}{\sqrt{n}h_s}\Big|^2\int \Big|K\Big(\frac{z_0}{h_{t}}-\frac{\sigma z}{\sigma_t h_{t}}\Big)\Big|^2 \d z \\
 & \le \frac{\varpi_s}{nh^2_s}\int \Big|K\Big(\frac{z_0}{h_{t}}-\frac{\sigma z}{\sigma_th_t}\Big)-K\Big(\frac{z_0}{h_{s}}-\frac{\sigma z}{\sigma_t h_{s}}\Big)\Big|^2 \d z\\
 &\qquad +\frac{\varpi_s}{nh^2_s}\int \Big|K\Big(\frac{z_0}{h_{s}}-\frac{\sigma z}{\sigma_th_s}\Big)-K\Big(\frac{z_0}{h_{s}}-\frac{\sigma z}{\sigma_s h_{s}}\Big)\Big|^2 \d z\\
 & \le \Big|\frac{\varpi_t}{\sqrt{n}h_t}-\frac{\varpi_s}{\sqrt{n}h_s}\Big|^2h_t +\frac{\varpi_s^2}{nh^2_s}\sup_{r\in[s,t]}\frac{\sigma_t^3h_r^3}{\sigma^3}\Big|\frac{\sigma}{\sigma_th_t}-\frac{\sigma}{\sigma_th_s}\Big|^2+\frac{\varpi_s}{\sqrt{n}h_s}\sup_{r\in[s,t]}\frac{\sigma_r^3h_s}{\sigma^3}\Big|\frac{\sigma}{\sigma_th_s}-\frac{\sigma}{\sigma_sh_s}\Big|^2\\
 &\lesssim \Big|\frac{\varpi_t}{\sqrt{n}h_t}-\frac{\varpi_s}{\sqrt{n}h_s}\Big|^2h_t +\frac{\varpi_s}{nh^2_s}\Big(h_s|t-s|^2+\frac{|\sigma_t-\sigma_s|^2}{\sigma^2 h_s}\Big)
\end{align*}
and thus we conclude
\begin{align*}
  \big(\sigma n+1\big)\big\|\Delta_{s,t}^{(2,2)}\big\|^2_{L^2}
  & \lesssim  h_t\big(h_s^{-1/2}-h_t^{-1/2}\big)^2+ \frac{h_t}{\sigma n}\big(h_t^{-1}-h_s^{-1}\big)^2\eins_{\sigma\ge n^{-(2s+1)/(2s+2)}}\\
  &+  \big(\sigma n+1\big)\frac{\min(n h_s,{\sigma}^{-1}{h}_s, (n\sigma^2)^{-1})}{nh_s^2}\big(h_s+\frac{\epsilon_n^2}{h_s}\big)|t-s|^2\lesssim |t-s|^2
\end{align*}
by distinguishing the three different cases where the minima can be attained. In particular, we have for the last term: 
\begin{align*}
 \big(\sigma n+1\big)\frac{\min(n h_s,\frac{h_s}{\sigma}, \frac{1}{n\sigma^2})}{nh_s^3}\epsilon_n^2
  \lesssim \begin{cases}
    \frac{\kappa_n^2}{n^2h_s^2}=\kappa_n^2n^{-4s/(2s+1)},\;& \sigma\le n^{-1},\\
    \frac{\kappa_n^2\sigma^2 n}{h_s^2}\le \kappa_n^2n^{-(s-1)/(s+1)},\; & n^{-1}<\sigma<n^{-\frac{2s+1}{2s+2}},\\
    \frac{\kappa_n^2\sigma}{h_s^{3}}=\kappa_n^2\sigma^{(2s-2)/(2s+1)},\;&\text{otherwise},
  \end{cases}
\end{align*}
which is uniformly bounded if $s>1$.
\end{proof}

\appendix
\section{Remaining proofs}

\subsection{Proof of the covariance structure of $(M,N)$}
\begin{proof}[Proof of Lemma~\ref{lem:higherOrder}]
$(i)$ to $(iii)$: For $A_{1},A_{2}\subset[-1,1]$ and $B_{1},B_{2}\subset\R$
with $A_{1}\cap A_{2}=\emptyset$ and $B_{1}\cap B_{2}=\emptyset$
we write in view of Proposition~\ref{prop:expFormula}:
\begin{align*}
\Psi(\eta_{1},\eta_{2},\xi_{1},\xi_{2}):= & \E\Big[e^{\eta_{1}M(A_{1})+\eta_{2}M(A_{2})+\xi_{1}N(B_{1})+\xi_{2}N(B_{2})}\Big]\\
= & \exp\Big(n\lambda(e^{\eta_{1}}-1)|A_{1}|+n\lambda(e^{\eta_{2}}-1)|A_{2}|+n\lambda\int_{0}^{1}\big(e^{\psi_{1}(\xi_{1},x)+\psi_{2}(\xi_{2},x)}-1)\d x\\
 & \qquad\qquad+n\lambda\sum_{i=1}^{2}(e^{\eta_{i}}-1)\int_{A_{i}}(e^{\psi_{1}(\xi_{1},x)+\psi_{2}(\xi_{2},x)}-1)\d x\Big).
\end{align*}
where 
\[
\psi_{j}(\xi,x):=\mu(e^{\xi}-1)\int_{B_{j}}f_{\sigma}(y-x)\d y.
\]
We moreover abbreviate
\begin{gather*}
h(x):=(e^{\psi_{1}(\xi_{1},x)+\psi_{2}(\xi_{2},x)}-1),\qquad h_{j}'(x):=\partial_{\xi_{j}}h(x)=e^{\psi_{1}(\xi_{1},x)+\psi_{2}(\xi_{2},x)}\partial_{\xi_{j}}\psi_{j}(\xi_{j},x).\\
h''(x):=\partial_{\xi_{1}}\partial_{\xi_{2}}h(x)=e^{\psi_{1}(\xi_{1},x)+\psi_{2}(\xi_{2},x)}\partial_{\xi_{1}}\psi_{1}(\xi_{1},x)\partial_{\xi_{2}}\psi_{2}(\xi_{2},x).
\end{gather*}
Then the first order partial derivatives are given by:
\begin{align*}
\partial_{\eta_{1}}\Psi(\eta_{1},\eta_{2},\xi_{1},\xi_{2}) & =\Psi(\eta_{1},\eta_{2},\xi_{1},\xi_{2})\,n\lambda e^{\eta_{1}}\Big(|A_{1}|+\int_{A_{1}}h(x)\d x\Big),\\
\partial_{\xi_{1}}\Psi(\eta_{1},\eta_{2},\xi_{1},\xi_{2}) & =\Psi(\eta_{1},\eta_{2},\xi_{1},\xi_{2})\,n\lambda\Big(\int_{0}^{1}h_{1}'(x)\d x+\sum_{i=1}^{2}(e^{\eta_{i}}-1)\int_{A_{i}}h_{1}'(x)\d x\Big).
\end{align*}
We moreover need the second order derivatives 
\begin{align}
\partial_{\eta_{1}}\partial_{\eta_{2}}\Psi(\eta_{1},\eta_{2},\xi_{1},\xi_{2}) & =\Psi(\eta_{1},\eta_{2},\xi_{1},\xi_{2})n^{2}\lambda^{2}e^{\eta_{1}+\eta_{2}}\Big(|A_{1}|+\int_{A_{1}}h(x)\d x\Big)\Big(|A_{2}|+\int_{A_{2}}h(x)\d x\Big),\notag\\
\partial_{\xi_{1}}\partial_{\xi_{2}}\Psi(\eta_{1},\eta_{2},\xi_{1},\xi_{2}) & =\Psi(\eta_{1},\eta_{2},\xi_{1},\xi_{2})n^{2}\lambda^{2}\Big(\int_{0}^{1}h_{1}'(x)\d x+\sum_{i=1}^{2}(e^{\eta_{i}}-1)\int_{A_{i}}h_{1}'(x)\d x\Big)\notag\\
 & \qquad\times\Big(\int_{0}^{1}h_{2}'(x)\d x+\sum_{i=1}^{2}(e^{\eta_{i}}-1)\int_{A_{i}}h_{2}'(x)\d x\Big)\notag\\
 & \quad+\Psi(\eta_{1},\eta_{2},\xi_{1},\xi_{2})n\lambda\Big(\int_{0}^{1}h''(x)\d x+\sum_{i=1}^{2}(e^{\eta_{i}}-1)\int_{A_{i}}h''(x)\d x\Big).\label{eq:VarBB}
\end{align}
Therefore,
\begin{align}
  \frac{\partial_{\eta_{1}}\partial_{\eta_{2}}\partial_{\xi_{1}}\Psi(\eta_{1},\eta_{2},\xi_{1},\xi_{2})}{\Psi(\eta_{1},\eta_{2},\xi_{1},\xi_{2})}
  =&n^{3}\lambda^{3}e^{\eta_{1}+\eta_{2}}\Big(|A_{1}|+\int_{A_{1}}h(x)\d x\Big)\Big(|A_{2}|+\int_{A_{2}}h(x)\d x\Big)\notag\\
 & \qquad\times\Big(\int_{0}^{1}h_{1}'(x)\d x+\sum_{i=1}^{2}(e^{\eta_{i}}-1)\int_{A_{i}}h_{1}'(x)\d x\Big)\notag\\
 & +n^{2}\lambda^{2}e^{\eta_{1}+\eta_{2}}\Big(\int_{A_{1}}h_{1}'(x)\d x\Big)\Big(|A_{2}|+\int_{A_{2}}h(x)\d x\Big)\notag\\
 & +n^{2}\lambda^{2}e^{\eta_{1}+\eta_{2}}\Big(|A_{1}|+\int_{A_{1}}h(x)\d x\Big)\Big(\int_{A_{2}}h_{1}'(x)\d x\Big)\label{eq:VarAAB}
\end{align}
and
\begin{align}
 & \frac{\partial_{\eta_{1}}\partial_{\xi_{1}}\partial_{\xi_{2}}\Psi(\eta_{1},\eta_{2},\xi_{1},\xi_{2})}{\Psi(\eta_{1},\eta_{2},\xi_{1},\xi_{2})}\notag\\
 & =n^{3}\lambda^{3}e^{\eta_{1}}\Big(|A_{1}|+\int_{A_{1}}h(x)\d x\Big)\Big(\int_{0}^{1}h_{1}'(x)\d x+\sum_{i=1}^{2}(e^{\eta_{i}}-1)\int_{A_{i}}h_{1}'(x)\d x\Big)\notag\\
 & \qquad\times\Big(\int_{0}^{1}h_{2}'(x)\d x+\sum_{i=1}^{2}(e^{\eta_{i}}-1)\int_{A_{i}}h_{2}'(x)\d x\Big)\notag\\
 & \quad+n^{2}\lambda^{2}e^{\eta_{1}}\int_{A_{1}}h_{1}'(x)\d x\,\Big(\int_{0}^{1}h_{2}'(x)\d x+\sum_{i=1}^{2}(e^{\eta_{i}}-1)\int_{A_{i}}h_{2}'(x)\d x\Big)\notag\\
 & \quad+n^{2}\lambda^{2}e^{\eta_{1}}\Big(\int_{0}^{1}h_{1}'(x)\d x+\sum_{i=1}^{2}(e^{\eta_{i}}-1)\int_{A_{i}}h_{1}'(x)\d x\Big)\int_{A_{1}}h_{2}'(x)\d x\notag\\
 & \quad+n^{2}\lambda^{2}e^{\eta_{1}}\Big(|A_{1}|+\int_{A_{1}}h(x)\d x\Big)\Big(\int_{0}^{1}h''(x)\d x+\sum_{i=1}^{2}(e^{\eta_{i}}-1)\int_{A_{i}}h''(x)\d x\Big)\notag\\
 & \quad+n\lambda e^{\eta_{1}}\int_{A_{1}}h''(x)\d x.\label{eq:VarABB}
\end{align}
Evaluating \eqref{eq:VarBB}, \eqref{eq:VarAAB} and \eqref{eq:VarABB} at $\eta_{1}=\eta_{2}=\xi_{1}=\xi_{2}=0$ yields (i), (ii) and (iii), respectively.

$(iv)$ It remains to calculate $\partial_{\eta_{1}}\partial_{\eta_{2}}\partial_{\xi_{1}}\partial_{\xi_{2}}\Psi(\eta_{1},\eta_{2},\xi_{1},\xi_{2})$ which can be deduced straightfroward from the previous formulas:
\begin{align*}
 & \frac{\partial_{\eta_{1}}\partial_{\eta_{2}}\partial_{\xi_{1}}\partial_{\xi_{2}}\Psi(\eta_{1},\eta_{2},\xi_{1},\xi_{2})}{\Psi(\eta_{1},\eta_{2},\xi_{1},\xi_{2})}\\
 & =n^{4}\lambda^{4}e^{\eta_{1}+\eta_{2}}\Big(|A_{1}|+\int_{A_{1}}h(x)\d x\Big)\Big(|A_{2}|+\int_{A_{2}}h(x)\d x\Big)\\
 & \qquad\times\Big(\int_{0}^{1}h_{1}'(x)\d x+\sum_{i=1}^{2}(e^{\eta_{i}}-1)\int_{A_{i}}h_{1}'(x)\d x\Big)\Big(\int_{0}^{1}h_{2}'(x)\d x+\sum_{i=1}^{2}(e^{\eta_{i}}-1)\int_{A_{i}}h_{2}'(x)\d x\Big)\\
 & \quad+n^{3}\lambda^{3}e^{\eta_{1}+\eta_{2}}\Big(|A_{1}|+\int_{A_{1}}h(x)\d x\Big)\int_{A_{2}}h_{1}'(x)\d x\Big(\int_{0}^{1}h_{2}'(x)\d x+\sum_{i=1}^{2}(e^{\eta_{i}}-1)\int_{A_{i}}h_{2}'(x)\d x\Big)\\
 & \quad+n^{3}\lambda^{3}e^{\eta_{1}+\eta_{2}}\Big(|A_{1}|+\int_{A_{1}}h(x)\d x\Big)\Big(\int_{0}^{1}h_{1}'(x)\d x+\sum_{i=1}^{2}(e^{\eta_{i}}-1)\int_{A_{i}}h_{1}'(x)\d x\Big)\int_{A_{2}}h_{2}'(x)\d x\\
 & \quad+n^{3}\lambda^{3}e^{\eta_{1}+\eta_{2}}\Big(|A_{2}|+\int_{A_{2}}h(x)\d x\Big)\int_{A_{1}}h_{1}'(x)\d x\Big(\int_{0}^{1}h_{2}'(x)\d x+\sum_{i=1}^{2}(e^{\eta_{i}}-1)\int_{A_{i}}h_{2}'(x)\d x\Big)\\
 & \quad+n^{2}\lambda^{2}e^{\eta_{1}+\eta_{2}}\int_{A_{1}}h_{1}'(x)\d x\int_{A_{2}}h_{2}'(x)\d x\\
 & \quad+n^{3}\lambda^{3}e^{\eta_{1}+\eta_{2}}\Big(|A_{2}|+\int_{A_{2}}h(x)\d x\Big)\Big(\int_{0}^{1}h_{1}'(x)\d x+\sum_{i=1}^{2}(e^{\eta_{i}}-1)\int_{A_{i}}h_{1}'(x)\d x\Big)\int_{A_{1}}h_{2}'(x)\d x\\
 & \quad+n^{2}\lambda^{2}e^{\eta_{1}+\eta_{2}}\int_{A_{2}}h_{1}'(x)\d x\int_{A_{1}}h_{2}'(x)\d x\\
 & \quad+n^{3}\lambda^{3}e^{\eta_{1}+\eta_{2}}\Big(|A_{1}|+\int_{A_{1}}h(x)\d x\Big)\Big(|A_{2}|+\int_{A_{2}}h(x)\d x\Big)\\
 & \qquad\times\Big(\int_{0}^{1}h''(x)\d x+\sum_{i=1}^{2}(e^{\eta_{i}}-1)\int_{A_{i}}h''(x)\d x\Big)\\
 & \quad+n^{2}\lambda^{2}e^{\eta_{1}+\eta_{2}}\Big(\Big(|A_{1}|+\int_{A_{1}}h(x)\d x\Big)\int_{A_{2}}h''(x)\d x
 +\Big(|A_{2}|+\int_{A_{2}}h(x)\d x\Big)\int_{A_{1}}h''(x)\d x\Big).
\end{align*}
Evaluating this partial derivative at 0 yields
\begin{align*}
 & \E[M(A_{1})M(A_{2})N(B_{1})N(B_{2})]\\
 & =n^{4}\lambda^{4}\mu^{2}|A_{1}||A_{2}|Q_{\sigma}([0,1],B_{1})Q_{\sigma}([0,1],B_{2})\\
 & \quad+n^{3}\lambda^{3}\mu^{2}\big(|A_{1}|Q_{\sigma}(A_{2},B_{1})Q_{\sigma}([0,1],B_{2})+|A_{1}|Q_{\sigma}([0,1],B_{1})Q_{\sigma}(A_{2},B_{1})\big)\\
 & \quad+n^{3}\lambda^{3}\mu^{2}\big(|A_{2}|Q_{\sigma}(A_{1},B_{1})Q_{\sigma}([0,1],B_{2})+|A_{2}|Q_{\sigma}([0,1],B_{1})Q_{\sigma}(A_{1},B_{2})\big)\\
 & \quad+n^{3}\lambda^{3}\mu^{2}|A_{1}||A_{2}|Q_{\sigma}^{2}([0,1])\\
 & \quad+n^{2}\lambda^{2}\mu^{2}\big(Q_{\sigma}(A_{1},B_{1})Q_{\sigma}(A_{2},B_{2})+Q_{\sigma}(A_{2},B_{1})Q_{\sigma}(A_{1},B_{2})\big)\\
 & \quad+n^{2}\lambda^{2}\mu^{2}\big(|A_{1}|Q_{\sigma}^{2}(A_{2})+|A_{2}|Q_{\sigma}^{2}(A_{1})\big).
\end{align*}
Combining this formula with Corollary~\ref{cor: intensity} yields
the assertion.
\end{proof}

\section*{Acknowledgment}
We thank our colleagues Marie Doumic and Alexander Goldenshluger for helpful discussions. The analysis and comments of three referees that convinced us to extend the results of a former version to the case of an unknown scale parameter are gratefully acknowledged. M.T. has been financially supported by DFG via the Heisenberg grant TR 1349/4-1.

\bibliographystyle{apalike}
\bibliography{ref}

\end{document}

%% file: IMGp+c_n10.tex
% Created by tikzDevice version 0.12.3 on 2021-07-22 21:31:52
% !TEX encoding = UTF-8 Unicode
\begin{tikzpicture}[x=0.5pt,y=0.5pt]
\definecolor{fillColor}{RGB}{255,255,255}
\path[use as bounding box,fill=fillColor,fill opacity=0.00] (0,0) rectangle (722.70,361.35);
\begin{scope}
\path[clip] ( 49.20, 49.20) rectangle (697.50,336.15);
\definecolor{fillColor}{RGB}{69,139,0}

\path[fill=fillColor] (225.20, 59.83) circle (  4.05);

\path[fill=fillColor] (225.20,148.39) circle (  4.05);

\path[fill=fillColor] (225.20,236.96) circle (  4.05);

\path[fill=fillColor] (225.20,325.52) circle (  4.05);

\path[fill=fillColor] (304.93, 59.83) circle (  4.05);

\path[fill=fillColor] (304.93,148.39) circle (  4.05);

\path[fill=fillColor] (304.93,236.96) circle (  4.05);

\path[fill=fillColor] (304.93,325.52) circle (  4.05);

\path[fill=fillColor] (362.80, 59.83) circle (  4.05);

\path[fill=fillColor] (362.80,148.39) circle (  4.05);

\path[fill=fillColor] (362.80,236.96) circle (  4.05);

\path[fill=fillColor] (362.80,325.52) circle (  4.05);

\path[fill=fillColor] (636.20, 59.83) circle (  4.05);

\path[fill=fillColor] (636.20,148.39) circle (  4.05);

\path[fill=fillColor] (636.20,236.96) circle (  4.05);

\path[fill=fillColor] (636.20,325.52) circle (  4.05);

\path[fill=fillColor] (170.79, 59.83) circle (  4.05);

\path[fill=fillColor] (170.79,148.39) circle (  4.05);

\path[fill=fillColor] (170.79,236.96) circle (  4.05);

\path[fill=fillColor] (170.79,325.52) circle (  4.05);

\path[fill=fillColor] (288.10, 59.83) circle (  4.05);

\path[fill=fillColor] (288.10,148.39) circle (  4.05);

\path[fill=fillColor] (288.10,236.96) circle (  4.05);

\path[fill=fillColor] (288.10,325.52) circle (  4.05);

\path[fill=fillColor] (595.25, 59.83) circle (  4.05);

\path[fill=fillColor] (595.25,148.39) circle (  4.05);

\path[fill=fillColor] (595.25,236.96) circle (  4.05);

\path[fill=fillColor] (595.25,325.52) circle (  4.05);

\path[fill=fillColor] (624.75, 59.83) circle (  4.05);

\path[fill=fillColor] (624.75,148.39) circle (  4.05);

\path[fill=fillColor] (624.75,236.96) circle (  4.05);

\path[fill=fillColor] (624.75,325.52) circle (  4.05);

\path[fill=fillColor] (600.57, 59.83) circle (  4.05);

\path[fill=fillColor] (600.57,148.39) circle (  4.05);

\path[fill=fillColor] (600.57,236.96) circle (  4.05);

\path[fill=fillColor] (600.57,325.52) circle (  4.05);

\path[fill=fillColor] (640.80, 59.83) circle (  4.05);

\path[fill=fillColor] (640.80,148.39) circle (  4.05);

\path[fill=fillColor] (640.80,236.96) circle (  4.05);

\path[fill=fillColor] (640.80,325.52) circle (  4.05);

\path[fill=fillColor] (363.77, 59.83) circle (  4.05);

\path[fill=fillColor] (363.77,148.39) circle (  4.05);

\path[fill=fillColor] (363.77,236.96) circle (  4.05);

\path[fill=fillColor] (363.77,325.52) circle (  4.05);

\path[fill=fillColor] (498.20, 59.83) circle (  4.05);

\path[fill=fillColor] (498.20,148.39) circle (  4.05);

\path[fill=fillColor] (498.20,236.96) circle (  4.05);

\path[fill=fillColor] (498.20,325.52) circle (  4.05);

\path[fill=fillColor] (321.40, 59.83) circle (  4.05);

\path[fill=fillColor] (321.40,148.39) circle (  4.05);

\path[fill=fillColor] (321.40,236.96) circle (  4.05);

\path[fill=fillColor] (321.40,325.52) circle (  4.05);
\end{scope}
\begin{scope}
\path[clip] (  0.00,  0.00) rectangle (722.70,361.35);
\definecolor{drawColor}{RGB}{0,0,0}

\path[draw=drawColor,line width= 0.4pt,line join=round,line cap=round] ( 80.04, 49.20) -- (673.49, 49.20);

\path[draw=drawColor,line width= 0.4pt,line join=round,line cap=round] ( 80.04, 49.20) -- ( 80.04, 43.20);

\path[draw=drawColor,line width= 0.4pt,line join=round,line cap=round] (198.73, 49.20) -- (198.73, 43.20);

\path[draw=drawColor,line width= 0.4pt,line join=round,line cap=round] (317.42, 49.20) -- (317.42, 43.20);

\path[draw=drawColor,line width= 0.4pt,line join=round,line cap=round] (436.11, 49.20) -- (436.11, 43.20);

\path[draw=drawColor,line width= 0.4pt,line join=round,line cap=round] (554.80, 49.20) -- (554.80, 43.20);

\path[draw=drawColor,line width= 0.4pt,line join=round,line cap=round] (673.49, 49.20) -- (673.49, 43.20);

\node[text=drawColor,anchor=base,inner sep=0pt, outer sep=0pt, scale = 0.9] at ( 80.04, 27.60) {0.0};

\node[text=drawColor,anchor=base,inner sep=0pt, outer sep=0pt, scale = 0.9] at (198.73, 27.60) {0.2};

\node[text=drawColor,anchor=base,inner sep=0pt, outer sep=0pt, scale = 0.9] at (317.42, 27.60) {0.4};

\node[text=drawColor,anchor=base,inner sep=0pt, outer sep=0pt, scale = 0.9] at (436.11, 27.60) {0.6};

\node[text=drawColor,anchor=base,inner sep=0pt, outer sep=0pt, scale = 0.9] at (554.80, 27.60) {0.8};

\node[text=drawColor,anchor=base,inner sep=0pt, outer sep=0pt, scale = 0.9] at (673.49, 27.60) {1.0};

\path[draw=drawColor,line width= 0.4pt,line join=round,line cap=round] ( 49.20, 49.20) --
	(697.50, 49.20) --
	(697.50,336.15) --
	( 49.20,336.15) --
	( 49.20, 49.20);
\end{scope}
\begin{scope}
\path[clip] (  0.00,  0.00) rectangle (722.70,361.35);
\definecolor{drawColor}{RGB}{0,0,0}

\node[text=drawColor,rotate= 90.00,anchor=base west,inner sep=0pt, outer sep=0pt, scale = 0.9] at ( 10.80,190.31) {$\sigma$};
\end{scope}
\begin{scope}
\path[clip] (  0.00,  0.00) rectangle (722.70,361.35);
\definecolor{drawColor}{RGB}{0,0,0}

\path[draw=drawColor,line width= 0.4pt,line join=round,line cap=round] ( 49.20, 59.83) -- ( 49.20,325.52);

\path[draw=drawColor,line width= 0.4pt,line join=round,line cap=round] ( 49.20, 59.83) -- ( 43.20, 59.83);

\path[draw=drawColor,line width= 0.4pt,line join=round,line cap=round] ( 49.20,148.39) -- ( 43.20,148.39);

\path[draw=drawColor,line width= 0.4pt,line join=round,line cap=round] ( 49.20,236.96) -- ( 43.20,236.96);

\path[draw=drawColor,line width= 0.4pt,line join=round,line cap=round] ( 49.20,325.52) -- ( 43.20,325.52);

\node[text=drawColor,rotate= 90.00,anchor=base west,inner sep=0pt, outer sep=0pt, scale = 0.9] at ( 34.80, 49.73) {$n^{-1.5}$};

\node[text=drawColor,rotate= 90.00,anchor=base west,inner sep=0pt, outer sep=0pt, scale = 0.9] at ( 34.80,141.56) {$n^{-1}$};

\node[text=drawColor,rotate= 90.00,anchor=base west,inner sep=0pt, outer sep=0pt, scale = 0.9] at ( 34.80,226.86) {$n^{-0.5}$};

\node[text=drawColor,rotate= 90.00,anchor=base west,inner sep=0pt, outer sep=0pt, scale = 0.9] at ( 34.80,320.09) {$n^0$};

\end{scope}
\begin{scope}
\path[clip] ( 49.20, 49.20) rectangle (697.50,336.15);
\definecolor{drawColor}{RGB}{190,190,190}

\path[draw=drawColor,line width= 0.4pt,line join=round,line cap=round] (361.60, 59.83) --
	(359.01,148.39) --
	(350.79,236.96) --
	(324.81,325.52);

\path[draw=drawColor,line width= 0.4pt,line join=round,line cap=round] (633.82, 59.83) --
	(628.68,148.39) --
	(612.43,236.96) --
	(561.02,325.52);

\path[draw=drawColor,line width= 0.4pt,line join=round,line cap=round] (369.41, 59.83) --
	(381.61,148.39) --
	(420.19,236.96) --
	(542.19,325.52);

\path[draw=drawColor,line width= 0.4pt,line join=round,line cap=round] (171.25, 59.83) --
	(172.25,148.39) --
	(175.43,236.96) --
	(185.47,325.52);

\path[draw=drawColor,line width= 0.4pt,line join=round,line cap=round] (624.05, 59.83) --
	(622.55,148.39) --
	(617.80,236.96) --
	(602.79,325.52);

\path[draw=drawColor,line width= 0.4pt,line join=round,line cap=round] (226.29, 59.83) --
	(228.64,148.39) --
	(236.08,236.96) --
	(259.60,325.52);

\path[draw=drawColor,line width= 0.4pt,line join=round,line cap=round] (360.33, 59.83) --
	(352.88,148.39) --
	(329.32,236.96) --
	(254.83,325.52);

\path[draw=drawColor,line width= 0.4pt,line join=round,line cap=round] (302.97, 59.83) --
	(298.74,148.39) --
	(285.34,236.96) --
	(242.99,325.52);

\path[draw=drawColor,line width= 0.4pt,line join=round,line cap=round] (367.33, 59.83) --
	(375.02,148.39) --
	(399.34,236.96) --
	(476.23,325.52);

\path[draw=drawColor,line width= 0.4pt,line join=round,line cap=round] (224.80, 59.83) --
	(223.92,148.39) --
	(221.14,236.96) --
	(212.34,325.52);

\path[draw=drawColor,line width= 0.4pt,line join=round,line cap=round] (620.49, 59.83) --
	(611.28,148.39) --
	(582.16,236.96) --
	(490.09,325.52);

\path[draw=drawColor,line width= 0.4pt,line join=round,line cap=round] (355.40, 59.83) --
	(337.28,148.39) --
	(280.01,236.96) --
	( 98.90,325.52);

\path[draw=drawColor,line width= 0.4pt,line join=round,line cap=round] (635.30, 59.83) --
	(633.36,148.39) --
	(627.21,236.96) --
	(607.78,325.52);

\path[draw=drawColor,line width= 0.4pt,line join=round,line cap=round] (592.15, 59.83) --
	(585.46,148.39) --
	(564.29,236.96) --
	(497.35,325.52);

\path[draw=drawColor,line width= 0.4pt,line join=round,line cap=round] (354.70, 59.83) --
	(337.17,148.39) --
	(281.73,236.96) --
	(106.43,325.52);

\path[draw=drawColor,line width= 0.4pt,line join=round,line cap=round] (167.70, 59.83) --
	(161.03,148.39) --
	(139.93,236.96) --
	( 73.21,325.52);

\path[draw=drawColor,line width= 0.4pt,line join=round,line cap=round] (596.04, 59.83) --
	(586.23,148.39) --
	(555.23,236.96) --
	(457.18,325.52);
\definecolor{fillColor}{RGB}{104,34,139}

\path[fill=fillColor] (357.55, 59.83) --
	(361.60, 63.88) --
	(365.65, 59.83) --
	(361.60, 55.78) --
	cycle;

\path[fill=fillColor] (354.96,148.39) --
	(359.01,152.44) --
	(363.06,148.39) --
	(359.01,144.34) --
	cycle;

\path[fill=fillColor] (346.74,236.96) --
	(350.79,241.01) --
	(354.84,236.96) --
	(350.79,232.91) --
	cycle;

\path[fill=fillColor] (320.76,325.52) --
	(324.81,329.57) --
	(328.86,325.52) --
	(324.81,321.47) --
	cycle;

\path[fill=fillColor] (629.77, 59.83) --
	(633.82, 63.88) --
	(637.87, 59.83) --
	(633.82, 55.78) --
	cycle;

\path[fill=fillColor] (624.63,148.39) --
	(628.68,152.44) --
	(632.73,148.39) --
	(628.68,144.34) --
	cycle;

\path[fill=fillColor] (608.38,236.96) --
	(612.43,241.01) --
	(616.48,236.96) --
	(612.43,232.91) --
	cycle;

\path[fill=fillColor] (556.97,325.52) --
	(561.02,329.57) --
	(565.07,325.52) --
	(561.02,321.47) --
	cycle;

\path[fill=fillColor] (365.36, 59.83) --
	(369.41, 63.88) --
	(373.46, 59.83) --
	(369.41, 55.78) --
	cycle;

\path[fill=fillColor] (377.56,148.39) --
	(381.61,152.44) --
	(385.66,148.39) --
	(381.61,144.34) --
	cycle;

\path[fill=fillColor] (416.14,236.96) --
	(420.19,241.01) --
	(424.24,236.96) --
	(420.19,232.91) --
	cycle;

\path[fill=fillColor] (538.14,325.52) --
	(542.19,329.57) --
	(546.24,325.52) --
	(542.19,321.47) --
	cycle;

\path[fill=fillColor] (167.20, 59.83) --
	(171.25, 63.88) --
	(175.30, 59.83) --
	(171.25, 55.78) --
	cycle;

\path[fill=fillColor] (168.20,148.39) --
	(172.25,152.44) --
	(176.30,148.39) --
	(172.25,144.34) --
	cycle;

\path[fill=fillColor] (171.38,236.96) --
	(175.43,241.01) --
	(179.48,236.96) --
	(175.43,232.91) --
	cycle;

\path[fill=fillColor] (181.42,325.52) --
	(185.47,329.57) --
	(189.52,325.52) --
	(185.47,321.47) --
	cycle;

\path[fill=fillColor] (620.00, 59.83) --
	(624.05, 63.88) --
	(628.10, 59.83) --
	(624.05, 55.78) --
	cycle;

\path[fill=fillColor] (618.50,148.39) --
	(622.55,152.44) --
	(626.60,148.39) --
	(622.55,144.34) --
	cycle;

\path[fill=fillColor] (613.75,236.96) --
	(617.80,241.01) --
	(621.85,236.96) --
	(617.80,232.91) --
	cycle;

\path[fill=fillColor] (598.74,325.52) --
	(602.79,329.57) --
	(606.84,325.52) --
	(602.79,321.47) --
	cycle;

\path[fill=fillColor] (222.24, 59.83) --
	(226.29, 63.88) --
	(230.34, 59.83) --
	(226.29, 55.78) --
	cycle;

\path[fill=fillColor] (224.59,148.39) --
	(228.64,152.44) --
	(232.69,148.39) --
	(228.64,144.34) --
	cycle;

\path[fill=fillColor] (232.03,236.96) --
	(236.08,241.01) --
	(240.13,236.96) --
	(236.08,232.91) --
	cycle;

\path[fill=fillColor] (255.55,325.52) --
	(259.60,329.57) --
	(263.65,325.52) --
	(259.60,321.47) --
	cycle;

\path[fill=fillColor] (356.28, 59.83) --
	(360.33, 63.88) --
	(364.38, 59.83) --
	(360.33, 55.78) --
	cycle;

\path[fill=fillColor] (348.83,148.39) --
	(352.88,152.44) --
	(356.93,148.39) --
	(352.88,144.34) --
	cycle;

\path[fill=fillColor] (325.27,236.96) --
	(329.32,241.01) --
	(333.37,236.96) --
	(329.32,232.91) --
	cycle;

\path[fill=fillColor] (250.78,325.52) --
	(254.83,329.57) --
	(258.88,325.52) --
	(254.83,321.47) --
	cycle;

\path[fill=fillColor] (298.92, 59.83) --
	(302.97, 63.88) --
	(307.02, 59.83) --
	(302.97, 55.78) --
	cycle;

\path[fill=fillColor] (294.69,148.39) --
	(298.74,152.44) --
	(302.79,148.39) --
	(298.74,144.34) --
	cycle;

\path[fill=fillColor] (281.29,236.96) --
	(285.34,241.01) --
	(289.39,236.96) --
	(285.34,232.91) --
	cycle;

\path[fill=fillColor] (238.94,325.52) --
	(242.99,329.57) --
	(247.04,325.52) --
	(242.99,321.47) --
	cycle;

\path[fill=fillColor] (363.28, 59.83) --
	(367.33, 63.88) --
	(371.38, 59.83) --
	(367.33, 55.78) --
	cycle;

\path[fill=fillColor] (370.97,148.39) --
	(375.02,152.44) --
	(379.07,148.39) --
	(375.02,144.34) --
	cycle;

\path[fill=fillColor] (395.29,236.96) --
	(399.34,241.01) --
	(403.39,236.96) --
	(399.34,232.91) --
	cycle;

\path[fill=fillColor] (472.18,325.52) --
	(476.23,329.57) --
	(480.28,325.52) --
	(476.23,321.47) --
	cycle;

\path[fill=fillColor] (220.75, 59.83) --
	(224.80, 63.88) --
	(228.85, 59.83) --
	(224.80, 55.78) --
	cycle;

\path[fill=fillColor] (219.87,148.39) --
	(223.92,152.44) --
	(227.97,148.39) --
	(223.92,144.34) --
	cycle;

\path[fill=fillColor] (217.09,236.96) --
	(221.14,241.01) --
	(225.19,236.96) --
	(221.14,232.91) --
	cycle;

\path[fill=fillColor] (208.29,325.52) --
	(212.34,329.57) --
	(216.39,325.52) --
	(212.34,321.47) --
	cycle;

\path[fill=fillColor] (616.44, 59.83) --
	(620.49, 63.88) --
	(624.54, 59.83) --
	(620.49, 55.78) --
	cycle;

\path[fill=fillColor] (607.23,148.39) --
	(611.28,152.44) --
	(615.33,148.39) --
	(611.28,144.34) --
	cycle;

\path[fill=fillColor] (578.11,236.96) --
	(582.16,241.01) --
	(586.21,236.96) --
	(582.16,232.91) --
	cycle;

\path[fill=fillColor] (486.04,325.52) --
	(490.09,329.57) --
	(494.14,325.52) --
	(490.09,321.47) --
	cycle;

\path[fill=fillColor] (351.35, 59.83) --
	(355.40, 63.88) --
	(359.45, 59.83) --
	(355.40, 55.78) --
	cycle;

\path[fill=fillColor] (333.23,148.39) --
	(337.28,152.44) --
	(341.33,148.39) --
	(337.28,144.34) --
	cycle;

\path[fill=fillColor] (275.96,236.96) --
	(280.01,241.01) --
	(284.06,236.96) --
	(280.01,232.91) --
	cycle;

\path[fill=fillColor] ( 94.85,325.52) --
	( 98.90,329.57) --
	(102.95,325.52) --
	( 98.90,321.47) --
	cycle;

\path[fill=fillColor] (631.25, 59.83) --
	(635.30, 63.88) --
	(639.35, 59.83) --
	(635.30, 55.78) --
	cycle;

\path[fill=fillColor] (629.31,148.39) --
	(633.36,152.44) --
	(637.41,148.39) --
	(633.36,144.34) --
	cycle;

\path[fill=fillColor] (623.16,236.96) --
	(627.21,241.01) --
	(631.26,236.96) --
	(627.21,232.91) --
	cycle;

\path[fill=fillColor] (603.73,325.52) --
	(607.78,329.57) --
	(611.83,325.52) --
	(607.78,321.47) --
	cycle;

\path[fill=fillColor] (588.10, 59.83) --
	(592.15, 63.88) --
	(596.20, 59.83) --
	(592.15, 55.78) --
	cycle;

\path[fill=fillColor] (581.41,148.39) --
	(585.46,152.44) --
	(589.51,148.39) --
	(585.46,144.34) --
	cycle;

\path[fill=fillColor] (560.24,236.96) --
	(564.29,241.01) --
	(568.34,236.96) --
	(564.29,232.91) --
	cycle;

\path[fill=fillColor] (493.30,325.52) --
	(497.35,329.57) --
	(501.40,325.52) --
	(497.35,321.47) --
	cycle;

\path[fill=fillColor] (350.65, 59.83) --
	(354.70, 63.88) --
	(358.75, 59.83) --
	(354.70, 55.78) --
	cycle;

\path[fill=fillColor] (333.12,148.39) --
	(337.17,152.44) --
	(341.22,148.39) --
	(337.17,144.34) --
	cycle;

\path[fill=fillColor] (277.68,236.96) --
	(281.73,241.01) --
	(285.78,236.96) --
	(281.73,232.91) --
	cycle;

\path[fill=fillColor] (102.38,325.52) --
	(106.43,329.57) --
	(110.48,325.52) --
	(106.43,321.47) --
	cycle;

\path[fill=fillColor] (163.65, 59.83) --
	(167.70, 63.88) --
	(171.75, 59.83) --
	(167.70, 55.78) --
	cycle;

\path[fill=fillColor] (156.98,148.39) --
	(161.03,152.44) --
	(165.08,148.39) --
	(161.03,144.34) --
	cycle;

\path[fill=fillColor] (135.88,236.96) --
	(139.93,241.01) --
	(143.98,236.96) --
	(139.93,232.91) --
	cycle;

\path[fill=fillColor] ( 69.16,325.52) --
	( 73.21,329.57) --
	( 77.26,325.52) --
	( 73.21,321.47) --
	cycle;

\path[fill=fillColor] (591.99, 59.83) --
	(596.04, 63.88) --
	(600.09, 59.83) --
	(596.04, 55.78) --
	cycle;

\path[fill=fillColor] (582.18,148.39) --
	(586.23,152.44) --
	(590.28,148.39) --
	(586.23,144.34) --
	cycle;

\path[fill=fillColor] (551.18,236.96) --
	(555.23,241.01) --
	(559.28,236.96) --
	(555.23,232.91) --
	cycle;

\path[fill=fillColor] (453.13,325.52) --
	(457.18,329.57) --
	(461.23,325.52) --
	(457.18,321.47) --
	cycle;
\definecolor{drawColor}{RGB}{255,255,255}
\definecolor{fillColor}{RGB}{255,255,255}

\path[draw=drawColor,line width= 0.4pt,line join=round,line cap=round,fill=fillColor] ( 50.36,102.24) rectangle (156.32, 60.33);
\definecolor{fillColor}{RGB}{69,139,0}

\path[fill=fillColor] ( 61.46, 88.27) circle (  2.70);
\definecolor{fillColor}{RGB}{104,34,139}

\path[fill=fillColor] ( 58.76, 74.30) --
	( 61.46, 77.00) --
	( 64.16, 74.30) --
	( 61.46, 71.60) --
	cycle;
\definecolor{drawColor}{RGB}{0,0,0}

\node[text=drawColor,anchor=base west,inner sep=0pt, outer sep=0pt, scale=  0.8] at ( 72.56, 83.31) {parent};

\node[text=drawColor,anchor=base west,inner sep=0pt, outer sep=0pt, scale=  0.8] at ( 72.56, 69.34) {offspring};
\end{scope}
\end{tikzpicture}

%% file: IMGrates-rev.tex
% Created by tikzDevice version 0.12.3 on 2021-07-22 21:38:31
% !TEX encoding = UTF-8 Unicode
\begin{tikzpicture}[x=.5pt,y=.5pt]
\definecolor{fillColor}{RGB}{255,255,255}
\path[use as bounding box,fill=fillColor,fill opacity=0.00] (0,0) rectangle (722.70,361.35);

\begin{scope}
\path[clip] (  0.00,  0.00) rectangle (722.70,361.35);
\definecolor{drawColor}{RGB}{0,0,0}

%Kasten
\path[draw=drawColor,line width= 0.4pt,line join=round,line cap=round] ( 49.20, 49.20) --
	(673.50, 49.20) --
	(673.50,348.15) --
	( 49.20,348.15) --
	( 49.20, 49.20);
\end{scope}

\begin{scope} %label
\definecolor{drawColor}{RGB}{0,0,0}

\node[text=drawColor,anchor=base,inner sep=0pt, outer sep=0pt, scale= 0.9] at ( 28.39, 45.35) {0};

\node[text=drawColor,anchor=base,inner sep=0pt, outer sep=0pt, scale= 0.9] at (255.30, 30.60) {-1};

\node[text=drawColor,anchor=base,inner sep=0pt, outer sep=0pt, scale= 0.9] at (673.50, 30.60) {0};

\node[text=drawColor,anchor=base west,inner sep=0pt, outer sep=0pt, scale= 0.9] at ( 25.39,340.94) {$\frac12$};

\node[text=drawColor,anchor=base west,inner sep=0pt, outer sep=0pt, scale= 0.9] at ( 10.02,244.78) {$\frac{s}{2s+1}$};

\node[text=drawColor,anchor=base west,inner sep=0pt, outer sep=0pt, scale= 0.9] at ( 10.02,193.95) {$\frac{s}{2s+2}$};

\node[text=drawColor,anchor=base west,inner sep=0pt, outer sep=0pt, scale= 0.9] at ( 10.02,93.30) {$\frac{s}{6s+6}$};

\node[text=drawColor,anchor=base west,inner sep=0pt, outer sep=0pt, scale= 0.9] at (680.94,165.06) {$\frac{s}{2s+3}$};

\node[text=drawColor,anchor=base west,inner sep=0pt, outer sep=0pt, scale= 0.9] at (332.96, 30.14) {$-\frac{2s+1}{2s+2}$};

\node[text=drawColor,anchor=base west,inner sep=0pt, outer sep=0pt, scale= 0.9] at (402.33, 30.14) {$-\frac{4s+3}{6s+6}$};

\node[text=drawColor,anchor=base west,inner sep=0pt, outer sep=0pt, scale= 0.9] at (305, 8) {$\log_n\sigma _n$};
\node[text=drawColor,rotate= 90.00,anchor=base,inner sep=0pt, outer sep=0pt, scale = 0.9] at ( 0,200) { $\log_nr_n$};

\end{scope}

\begin{scope}
\path[clip] ( 49.20, 49.20) rectangle (673.50,348.15);
\definecolor{drawColor}{RGB}{0,0,0}

% Dicke Striche
\path[draw=drawColor,line width= 0.6pt,line join=round,line cap=round] ( 49.20,248.50) -- (257.30,248.50);

\path[draw=drawColor,line width= 0.6pt,line join=round,line cap=round] (361.35,198.67) -- (430.72, 99.03);
	
\path[draw=drawColor,line width= 0.6pt,line join=round,line cap=round] (430.72, 99.03) -- (673.50,168.78);

\path[draw=drawColor,line width= 0.6pt,line join=round,line cap=round] (257.30,248.50) -- (361.35,198.67);

% Dicke Hilfslinien
\path[draw=drawColor,line width= 0.4pt,dash pattern=on 4pt off 4pt ,line join=round,line cap=round] (257.30,348.15) -- (465.40, 49.20);

\path[draw=drawColor,line width= 0.4pt,dash pattern=on 4pt off 4pt ,line join=round,line cap=round] (257.30,248.50) -- (673.50, 49.20);

\path[draw=drawColor,line width= 0.4pt,dash pattern=on 4pt off 4pt ,line join=round,line cap=round] (430.72, 99.03) -- (257.30, 49.20);

% Feine Hilfslinien	
% senkrecht -1
\path[draw=drawColor,line width= 0.4pt,dash pattern=on 1pt off 3pt ,line join=round,line cap=round] (257.30, 49.20) -- (257.30,248.50);

% waagerecht s/(2s+1)
\path[draw=drawColor,line width= 0.4pt,dash pattern=on 1pt off 3pt ,line join=round,line cap=round] ( 49.20,198.67) -- (361.35,198.67);

% waagrecht s/(6s+6)
\path[draw=drawColor,line width= 0.4pt,dash pattern=on 1pt off 3pt ,line join=round,line cap=round] ( 49.20, 99.03) -- (430.72, 99.03);

%senkrecht 4s+3 ..
\path[draw=drawColor,line width= 0.4pt,dash pattern=on 1pt off 3pt ,line join=round,line cap=round] (430.72, 49.20) -- (430.72, 99.03);

\path[draw=drawColor,line width= 0.4pt,dash pattern=on 1pt off 3pt ,line join=round,line cap=round] (361.35, 49.20) -- (361.35,198.67);

\end{scope}
\definecolor{drawColor}{RGB}{255,0,0}
%\path[draw=drawColor,line width= 0.2pt,dash pattern=on 1pt off 1pt ,line join=round,line cap=round] (257.30, 49.20) -- (673.50,348.15);

%\node[text=drawColor,anchor=base west,inner sep=0pt, outer sep=0pt, scale= 0.9] at (560, 250) {$\hat\sigma^{(1)}$};

%\path[draw=drawColor,line width= 0.4pt,dash pattern=on 1pt off 1pt ,line join=round,line cap=round] (257.30,348.15) -- (465.40, 49.20);

%\node[text=drawColor,anchor=base west,inner sep=0pt, outer sep=0pt, scale= 0.9] at (330, 250) {$\hat\sigma^{(2)}$)};

%\path[draw=drawColor,line width= 0.4pt,dash pattern=on 1pt off 1pt ,line join=round,line cap=round] (257.30,348.15) -- (49.20,348.15);

%\node[text=drawColor,anchor=base west,inner sep=0pt, outer sep=0pt, scale= 0.9] at (110, 330) {$\hat\sigma^{(2)}$)};

\begin{scope}
 
\end{scope}

\end{tikzpicture}

%% file: IMGrates.tex
% Created by tikzDevice version 0.12.3 on 2021-07-22 21:38:31
% !TEX encoding = UTF-8 Unicode
\begin{tikzpicture}[x=.5pt,y=.5pt]
\definecolor{fillColor}{RGB}{255,255,255}
\path[use as bounding box,fill=fillColor,fill opacity=0.00] (0,0) rectangle (722.70,361.35);

\begin{scope}
\path[clip] (  0.00,  0.00) rectangle (722.70,361.35);
\definecolor{drawColor}{RGB}{0,0,0}

%Kasten
\path[draw=drawColor,line width= 0.4pt,line join=round,line cap=round] ( 49.20, 49.20) --
	(673.50, 49.20) --
	(673.50,348.15) --
	( 49.20,348.15) --
	( 49.20, 49.20);
\end{scope}

\begin{scope} %label
\definecolor{drawColor}{RGB}{0,0,0}

\node[text=drawColor,anchor=base,inner sep=0pt, outer sep=0pt, scale= 0.9] at ( 28.39, 45.35) {0};

\node[text=drawColor,anchor=base,inner sep=0pt, outer sep=0pt, scale= 0.9] at (255.30, 30.60) {-1};

\node[text=drawColor,anchor=base,inner sep=0pt, outer sep=0pt, scale= 0.9] at (673.50, 30.60) {0};

\node[text=drawColor,anchor=base west,inner sep=0pt, outer sep=0pt, scale= 0.9] at ( 25.39,340.94) {$\frac12$};

\node[text=drawColor,anchor=base west,inner sep=0pt, outer sep=0pt, scale= 0.9] at ( 10.02,244.78) {$\frac{s}{2s+1}$};

\node[text=drawColor,anchor=base west,inner sep=0pt, outer sep=0pt, scale= 0.9] at ( 25.39,139) {$\frac{1}{6}$};

\node[text=drawColor,anchor=base west,inner sep=0pt, outer sep=0pt, scale= 0.9] at (680.94,165.06) {$\frac{s}{2s+3}$};

\node[text=drawColor,anchor=base west,inner sep=0pt, outer sep=0pt, scale= 0.9] at (385, 30.14) {$-\frac{2}{3}$};

\node[text=drawColor,anchor=base west,inner sep=0pt, outer sep=0pt, scale= 0.9] at (305, 8) {$\log_n\sigma _n$};
\node[text=drawColor,rotate= 90.00,anchor=base,inner sep=0pt, outer sep=0pt, scale = 0.9] at ( 0,200) { $\log_n \frac{\hat \sigma-\sigma}{\sigma}$};

\end{scope}

\begin{scope}
\path[clip] ( 49.20, 49.20) rectangle (673.50,348.15);
\definecolor{drawColor}{RGB}{0,0,0}

% Dicke Striche
\path[draw=drawColor,line width= 0.6pt, dash pattern=on 2pt off 2pt,line join=round,line cap=round] ( 49.20,248.50) -- (257.30,248.50);

\path[draw=drawColor,line width= 0.6pt,dash pattern=on 2pt off 2pt,line join=round,line cap=round] (361.35,198.67) -- (430.72, 99.03);
	
\path[draw=drawColor,line width= 0.6pt,dash pattern=on 2pt off 2pt,line join=round,line cap=round] (430.72, 99.03) -- (673.50,168.78);

\path[draw=drawColor,line width= 0.6pt,dash pattern=on 2pt off 2pt,line join=round,line cap=round] (257.30,248.50) -- (361.35,198.67);

%% Feine Hilfslinien	
%% senkrecht -2/3
\path[draw=drawColor,line width= 0.4pt,dash pattern=on 1pt off 3pt ,line join=round,line cap=round] (403, 49.20) -- (403,140);

%% senkrecht -1
\path[draw=drawColor,line width= 0.4pt,dash pattern=on 1pt off 3pt ,line join=round,line cap=round] (257.30, 49.20) -- (257.30,348.15);

%% waagerecht 1/6
\path[draw=drawColor,line width= 0.4pt,dash pattern=on 1pt off 3pt ,line join=round,line cap=round] ( 49.20,139) -- (403,139);
%
%% waagrecht s/(6s+6)
%\path[draw=drawColor,line width= 0.4pt,dash pattern=on 1pt off 3pt ,line join=round,line cap=round] ( 49.20, 99.03) -- (430.72, 99.03);
%
%%senkrecht 4s+3 ..
%\path[draw=drawColor,line width= 0.4pt,dash pattern=on 1pt off 3pt ,line join=round,line cap=round] (430.72, 49.20) -- (430.72, 99.03);
%
%\path[draw=drawColor,line width= 0.4pt,dash pattern=on 1pt off 3pt ,line join=round,line cap=round] (361.35, 49.20) -- (361.35,198.67);

\end{scope}
\definecolor{drawColor}{RGB}{104,34,139}
\path[draw=drawColor,line width= 0.9pt ,line join=round,line cap=round] (403, 139) -- (673.50,348.15);

\node[text=drawColor,anchor=base west,inner sep=0pt, outer sep=0pt, scale= 0.9] at (560, 245) {$\hat\sigma^{(1)}$};

\path[draw=drawColor,line width= 0.9pt ,line join=round,line cap=round] (257.30,348.15) -- (403, 139);

\node[text=drawColor,anchor=base west,inner sep=0pt, outer sep=0pt, scale= 0.9] at (330, 250) {$\hat\sigma^{(2)}$};

\path[draw=drawColor,line width= 0.9pt ,line join=round,line cap=round] (257.30,348.15) -- (49.20,348.15);

\node[text=drawColor,anchor=base west,inner sep=0pt, outer sep=0pt, scale= 0.9] at (110, 330) {$\hat\sigma^{(2)}$};

\begin{scope}
 
\end{scope}

\end{tikzpicture}

%% file: IMGmc.tex
% Created by tikzDevice version 0.12.3 on 2021-07-22 21:10:24
% !TEX encoding = UTF-8 Unicode
\begin{tikzpicture}[x=.8pt,y=.8pt]
\definecolor{fillColor}{RGB}{255,255,255}
\path[use as bounding box,fill=fillColor,fill opacity=0.00] (0,0) rectangle (433.62,216.81);
\begin{scope}
\path[clip] ( 49.20, 49.20) rectangle (420.42,203.61);
\definecolor{drawColor}{RGB}{104,34,139}

\path[draw=drawColor,line width= 0.4pt,line join=round,line cap=round] (211.07,216.81) --
	(215.71,213.60) --
	(253.91,205.22) --
	(292.10,184.84) --
	(330.29,166.61) --
	(368.48,150.37) --
	(406.67,133.56);
\end{scope}
\begin{scope}
\path[clip] (  0.00,  0.00) rectangle (433.62,216.81);
\definecolor{drawColor}{RGB}{0,0,0}

\path[draw=drawColor,line width= 0.4pt,line join=round,line cap=round] (120.24, 49.20) -- (406.67, 49.20);

\path[draw=drawColor,line width= 0.4pt,line join=round,line cap=round] (120.24, 49.20) -- (120.24, 43.20);

\path[draw=drawColor,line width= 0.4pt,line join=round,line cap=round] (215.71, 49.20) -- (215.71, 43.20);

\path[draw=drawColor,line width= 0.4pt,line join=round,line cap=round] (311.19, 49.20) -- (311.19, 43.20);

\path[draw=drawColor,line width= 0.4pt,line join=round,line cap=round] (406.67, 49.20) -- (406.67, 43.20);

\node[text=drawColor,anchor=base,inner sep=0pt, outer sep=0pt, scale = 0.9] at (120.24, 27.60) {-1.5};

\node[text=drawColor,anchor=base,inner sep=0pt, outer sep=0pt, scale = 0.9] at (215.71, 27.60) {-1.0};

\node[text=drawColor,anchor=base,inner sep=0pt, outer sep=0pt, scale = 0.9] at (311.19, 27.60) {-0.5};

\node[text=drawColor,anchor=base,inner sep=0pt, outer sep=0pt, scale = 0.9] at (406.67, 27.60) {0.0};

\path[draw=drawColor,line width= 0.4pt,line join=round,line cap=round] ( 49.20, 78.45) -- ( 49.20,197.89);

\path[draw=drawColor,line width= 0.4pt,line join=round,line cap=round] ( 49.20, 78.45) -- ( 43.20, 78.45);

\path[draw=drawColor,line width= 0.4pt,line join=round,line cap=round] ( 49.20,108.31) -- ( 43.20,108.31);

\path[draw=drawColor,line width= 0.4pt,line join=round,line cap=round] ( 49.20,138.17) -- ( 43.20,138.17);

\path[draw=drawColor,line width= 0.4pt,line join=round,line cap=round] ( 49.20,168.03) -- ( 43.20,168.03);

\path[draw=drawColor,line width= 0.4pt,line join=round,line cap=round] ( 49.20,197.89) -- ( 43.20,197.89);

\node[text=drawColor,anchor=base,inner sep=0pt, outer sep=0pt, scale = 0.9] at ( 32.80, 75.45) {-0.3};

\node[text=drawColor,,anchor=base,inner sep=0pt, outer sep=0pt, scale = 0.9] at ( 32.80,105.31) {-0.2};

\node[text=drawColor,anchor=base,inner sep=0pt, outer sep=0pt, scale = 0.9] at ( 32.80,135.17) {-0.1};

\node[text=drawColor,anchor=base,inner sep=0pt, outer sep=0pt, scale = 0.9] at ( 32.80,165.03) {0.0};

\node[text=drawColor,anchor=base,inner sep=0pt, outer sep=0pt, scale = 0.9] at ( 32.80,194.89) {0.1};

\path[draw=drawColor,line width= 0.4pt,line join=round,line cap=round] ( 49.20, 49.20) --
	(420.42, 49.20) --
	(420.42,203.61) --
	( 49.20,203.61) --
	( 49.20, 49.20);
\end{scope}
\begin{scope}
\path[clip] (  0.00,  0.00) rectangle (433.62,216.81);
\definecolor{drawColor}{RGB}{0,0,0}

\node[text=drawColor,anchor=base,inner sep=0pt, outer sep=0pt, scale = 0.9] at (234.81,  10.60) {$\log_n\sigma_n$};

\node[text=drawColor,rotate= 90.00,anchor=base,inner sep=0pt, outer sep=0pt, scale = 0.9] at ( 10.80,126.41) { $\log_n(\E[(\hat f(0)-f(0))^2]^{1/2})$};

\end{scope}
\begin{scope}
\path[clip] ( 49.20, 49.20) rectangle (420.42,203.61);
\definecolor{drawColor}{RGB}{104,34,139}

\path[draw=drawColor,line width= 0.4pt,dash pattern=on 4pt off 4pt ,line join=round,line cap=round] ( 24.76,188.31) --
	( 35.71,216.81);
\node[text=drawColor,anchor=base,inner sep=0pt, outer sep=0pt, scale = 0.9] at ( 390,125) { $\hat f_{h_1}^{\mathrm{dec}}$};

\path[draw=drawColor,line width= 0.4pt,dash pattern=on 4pt off 4pt ,line join=round,line cap=round] (211.48,216.81) --
	(215.71,213.96) --
	(253.91,205.09) --
	(292.10,184.69) --
	(330.29,166.88) --
	(368.48,150.12) --
	(406.67,133.04);
\node[text=drawColor,anchor=base,inner sep=0pt, outer sep=0pt, scale = 0.9] at ( 350,165) { $\hat f_{h_1}^{(1)}$};

\definecolor{drawColor}{RGB}{69,139,0}

\path[draw=drawColor,line width= 0.4pt,line join=round,line cap=round] ( 24.76, 88.25) --
	( 62.95, 88.36) --
	(101.14, 88.68) --
	(139.33, 90.34) --
	(177.52, 94.73) --
	(215.71,109.88) --
	(253.91,138.56) --
	(292.10,188.71) --
	(310.65,216.81);
\node[text=drawColor,anchor=base,inner sep=0pt, outer sep=0pt, scale = 0.9] at ( 245,115) { $\hat f_{h_2}^{(2)}$};

\path[draw=drawColor,line width= 0.4pt,dash pattern=on 4pt off 4pt ,line join=round,line cap=round] ( 24.76, 56.37) --
	( 62.95, 55.80) --
	(101.14, 54.92) --
	(139.33, 61.16) --
	(177.52,105.76) --
	(215.71,167.03) --
	(247.32,216.81);
\node[text=drawColor,anchor=base,inner sep=0pt, outer sep=0pt, scale = 0.9] at ( 200,155) { $\hat f_{h_2}^{\mathrm{int}}$};
\end{scope}
\end{tikzpicture}